\def\todaysdate{20\textsuperscript{th} November 2025}
\documentclass[reqno,a4paper]{article}
\usepackage{etex}
\usepackage[T1]{fontenc}
\usepackage[utf8]{inputenc}
\usepackage{verbatim}
\usepackage{lmodern}
\usepackage{subcaption}

\usepackage{amssymb,amsmath,amsthm}
\usepackage{thmtools,xcolor}
\usepackage{units}
\usepackage{graphicx}
\usepackage[all]{xy}
\usepackage{tikz}
\usetikzlibrary{arrows,calc,positioning,decorations.pathreplacing,cd}
\usepackage{relsize}
\usepackage{mhsetup}
\usepackage{mathtools}
\usepackage{float}
\usepackage{stmaryrd}
\usepackage{tocloft}
\usepackage{paralist} 
\usepackage[left=3cm,right=3cm,top=3cm,bottom=3cm]{geometry} 
\usepackage[bottom]{footmisc}
\usepackage[colorlinks,citecolor=blue,linkcolor=blue,urlcolor=blue,filecolor=blue,breaklinks]{hyperref}
\usepackage{footnotebackref}
\usepackage[format=plain,indention=1em,labelsep=quad,labelfont={up},textfont={normalsize},margin=2em]{caption}
\usepackage[mathscr]{euscript}
\usepackage{accents}

\definecolor{lightblue}{rgb}{0.8,0.8,1}
\definecolor{carmine}{rgb}{0.59, 0.0, 0.09}

\usepackage{enumitem}
\setlist[enumerate]{nosep}
\setlist[itemize]{nosep}

\numberwithin{equation}{section}
\numberwithin{figure}{section}

\definecolor{vdarkred}{rgb}{0.7,0,0}
\declaretheoremstyle[
  spaceabove=\topsep,
  spacebelow=\topsep,
  headpunct=,
  numbered=no,
  postheadspace=1ex,
  headfont=\color{vdarkred}\normalfont\bfseries,
  bodyfont=\normalfont\itshape,
]{colored}
\declaretheoremstyle[
  spaceabove=\topsep,
  spacebelow=\topsep,
  headpunct=,
  numbered=no,
  postheadspace=1ex,
  headfont=\normalfont\bfseries,
  bodyfont=\normalfont\itshape,
]{italic}
\declaretheoremstyle[
  spaceabove=\topsep,
  spacebelow=\topsep,
  headpunct=,
  numbered=no,
  postheadspace=1ex,
  headfont=\normalfont\bfseries,
  bodyfont=\normalfont\upshape,
]{upright}

\declaretheorem[style=italic,name=Theorem,numbered=yes,numberwithin=section]{thm}
\declaretheorem[style=italic,name=Lemma,numbered=yes,numberlike=thm]{lem}
\declaretheorem[style=italic,name=Proposition,numbered=yes,numberlike=thm]{prop}

\declaretheorem[style=italic,name=Conjecture,numbered=yes,numberlike=thm]{conjecture}

\declaretheorem[style=upright,name=Definition,numbered=yes,numberlike=thm]{defn}
\declaretheorem[style=upright,name=Remark,numbered=yes,numberlike=thm]{rmk}

\declaretheorem[style=upright,name=Notation,numbered=yes,numberlike=thm]{notation}

\makeatletter
\renewcommand*{\@seccntformat}[1]{\upshape\csname the#1\endcsname.\hspace{1ex}}
\renewcommand*{\section}{\@startsection{section}{1}{\z@}%
	{2.5ex \@plus 1ex \@minus 0.2ex}%
	{1.5ex \@plus 0.2ex}%
	{\normalfont\normalsize\bfseries}}
\renewcommand*{\subsection}{\@startsection{subsection}{2}{\z@}%
	{2.5ex \@plus 1ex \@minus 0.2ex}%
	{-1.5ex \@plus -0.2ex}%
	{\normalfont\normalsize\bfseries}}
\renewcommand*{\subsubsection}{\@startsection{subsubsection}{3}{\z@}%
	{2.5ex \@plus 1ex \@minus 0.2ex}%
	{-1.5ex \@plus -0.2ex}%
	{\normalfont\normalsize\bfseries}}
\renewcommand*{\paragraph}{\@startsection{paragraph}{4}{\z@}%
	{2.5ex \@plus 1ex \@minus 0.2ex}%
	{-1.5ex \@plus -0.2ex}%
	{\normalfont\normalsize\bfseries}}
\renewcommand*{\subparagraph}{\@startsection{subparagraph}{5}{\z@}%
	{2.5ex \@plus 1ex \@minus 0.2ex}%
	{-1.5ex \@plus -0.2ex}%
	{\normalfont\normalsize\slshape}}
\makeatother

\newcommand{\FA}{\mathscr F_{\bar{i},\cN}}
\newcommand{\LA}{\mathscr L_{\bar{i},\cN}}
\newcommand{\FAnn}{\mathscr F_{\bar{i},\cN+1}}
\newcommand{\LAnn}{\mathscr L_{\bar{i},\cN+1}}
\newcommand{\FAM}{\mathscr F_{\bar{i},\cM}}
\newcommand{\LAM}{\mathscr L_{\bar{i},\cM}}

\newcommand{\PJ}{J{ \Gamma}^{\bar{N}}(\beta_n)}
\newcommand{\PJk}{J{ \Gamma}^{\cN}(\beta_n)}
\newcommand{\PA}{\Gamma^{\cN}(\beta_n)}
\newcommand{\PAo}{\Gamma^{\cM,A}(\beta_n)}

\newcommand{\PAA}{\hat{\Gamma}^{\cN}(L)}
\newcommand{\PAAJ}{\hat{\Gamma}^{\cN,J}(L)}
\newcommand{\PAAJr}{\hat{\Gamma}^{\cN,J,R}}
\newcommand{\PAAJw}{\hat{\Gamma}^{\cN,J}(K)}
\newcommand{\PAAJk}{\hat{\Gamma}^{\cN,J,k}(K)}
\newcommand{\PAAN}{\hat{\Gamma}^{\cN+1}(L)}
\newcommand{\PAANJ}{\hat{\Gamma}^{\cN+1,J}(L)}

\newcommand{\PAAR}{\hat{\Gamma}^{R,\cN}(\beta_n)}

\newcommand{\PAAW}{\hat{\Gamma}^{\cN,W}(K)}

\newcommand{\LL}{\mathbb L_{\cN}}
\newcommand{\LLh}{\hat{\mathbb L}}
\newcommand{\LLhJ}{\hat{\mathbb L}^J}
\newcommand{\LLhJk}{\hat{\mathbb L}^{J,k}}
\newcommand{\LLhW}{\hat{\mathbb L}^W}

\newcommand{\LLN}{\hat{\mathbb L}_{\mathcal N}}
\newcommand{\LLNJ}{\hat{\mathbb L}^{J}_{\mathcal N}}
\newcommand{\LLNJk}{\hat{\mathbb L}^{J,k}_{\mathcal N}}
\newcommand{\LLNW}{\hat{\mathbb L}^{W}_{\mathcal N}}
\newcommand{\LLNN}{\hat{\mathbb L}_{\mathcal N+1}}
\newcommand{\LLNNJ}{\hat{\mathbb L}^{J}_{\mathcal N+1}}

\newcommand{\La}{\hat{\Lambda}}

\newcommand{\LN}{\Lambda_{\mathcal N}}
\newcommand{\LNJ}{\Lambda^{J}_{\bar{N}}}
\newcommand{\LM}{\Lambda_{\mathcal M}}
\newcommand{\LMJ}{\Lambda_{\bar{N}}^{J}}
\newcommand{\LNt}{\hat{\Lambda}_{\mathcal N}}
\newcommand{\LNtJ}{\hat{\Lambda}_{\mathcal N}^{J}}
\newcommand{\LNNt}{\hat{\Lambda}_{\mathcal N+1}}
\newcommand{\LNNtJ}{\hat{\Lambda}_{\mathcal N+1}^{J}}
\newcommand{\LNN}{\Lambda_{\mathcal N+1}}
\newcommand{\LNNJ}{\Lambda_{\mathcal N+1}^{J}}

\newcommand{\I}{{\large \hat{{\Huge {\Gamma}}}}(L)}

\newcommand{\IR}{{\large \hat{\Huge {\Gamma^R}}}}

\newcommand{\AM}{\Phi^{\mathcal M}(L)}
\newcommand{\AN}{\Phi^{\mathcal N}(L)}
\newcommand{\ANt}{\hat{\Phi}^{\mathcal N}(L)}
\newcommand{\ANtJ}{\hat{J}^{\mathcal N}(L)}

\newcommand{\ANNt}{\hat{\Phi}^{\mathcal N+1}(L)}

\newcommand{\Aa}{\hat{\Phi}(L)}

\newcommand{\IJJ}{{\large \hat{{\Huge {\Gamma}}}^{J}}(L)}

\newcommand{\IJJw}{{\large \hat{{\Huge {\Gamma}}^{J}}}(K)}
\newcommand{\IJJk}{{\large \hat{{\Huge {\Gamma}}^{J,k}}}(K)}
\newcommand{\IW}{{\large \hat{{\Huge {\Gamma}}^{W}}}(K)}
\newcommand{\IRJJ}{{\large \hat{\Huge {\Gamma}}^{J,R}}}

\newcommand{\ANJ}{J_{\bar{N}}(L)}

\newcommand{\su}{\hat{\psi}}
\newcommand{\usn}{\psi^{u}_{\mathcal N}}
\newcommand{\usnJ}{\psi^{u,J}_{\bar {N}}}

\newcommand{\usm}{\psi^{u}_{\mathcal M}}

\newcommand{\snJ}{\psi^{\cN}_{\bar{N}}}
\newcommand{\sm}{\psi_{\mathcal M}}

\newcommand{\snt}{\tilde{\psi}_{\mathcal N}}
\newcommand{\sntJ}{\tilde{\psi}^{J}_{\mathcal N}}
\newcommand{\snnt}{\tilde{\psi}_{\mathcal N+1}}

\newcommand{\snb}{\bar{\psi}_{\mathcal N}}
\newcommand{\snbt}{\hat{\psi}^{\mathcal N}_{\mathcal N}}
\newcommand{\snbJ}{\bar{\psi}^{\mathcal N,J}_{\mathcal N}}
\newcommand{\smnbt}{\hat{\psi}^{\mathcal N}_{\mathcal M}}
\newcommand{\smnbJt}{\hat{\psi}^{\mathcal N}_{\bar{N}}}
\newcommand{\snbJt}{\hat{\psi}^{\mathcal N,J}_{\mathcal N}}
\newcommand{\snbbb}{\bar{\psi}_{\mathcal N+1}}

\newcommand{\snbb}{\bar{\bar{\psi}}_{\mathcal N}}
\newcommand{\snbbJ}{\bar{\bar{\psi}}^{J}_{\mathcal N}}
\newcommand{\snnbb}{\bar{\bar{\psi}}_{\mathcal N+1}}

\newcommand{\PPN}{\text{pr}_{\mathcal{N}}}
\newcommand{\PPNJ}{\text{pr}_{\mathcal{N}}^{J}}

\newcommand{\PMN}{p^{\mathcal{N}}_{\mathcal{M}}}
\newcommand{\PMNJ}{p^{\mathcal{N},J}_{\bar{N}}}

\newcommand{\PMNN}{p^{\mathcal{N}}_{\mathcal{M}}}

\newcommand{\PNN}{p^{\mathcal{N}}_{\mathcal{N}}}
\newcommand{\PNNJ}{p^{\mathcal{N},J}_{\bar{N}}}

\newcommand{\PhM}{\hat{p}_{\mathcal{M}}}

\newcommand{\PhN}{\hat{p}_{\mathcal{N}}}

\newcommand{\N}{\mathbb N}
\newcommand{\Z}{\mathbb Z}

\newcommand{\Imm}{\mathrm{Im}}

\newcommand{\Conf}{\mathrm{Conf}}

\definecolor{dgreen}{RGB}{0,150,0}

\newcommand{\incl}[3][right]%
{%
\draw[<-,>=#1 hook] #2 to ($ #2!0.5!#3 $);
\draw[->,>=stealth'] ($ #2!0.5!#3 $) to #3;%
}
\newcommand{\inclusion}[5][right]%
{%
\draw[<-,>=#1 hook] #4 to ($ #4!0.5!#5 $) node[#2,font=\small]{#3};
\draw[->,>=stealth'] ($ #4!0.5!#5 $) to #5;%
}

{\begin{compactitem}

}%
{\end{compactitem}}
{\begin{compactitem}[#1]

}%
{\end{compactitem}}
{\begin{compactdesc}

}%
{\end{compactdesc}}

\newcommand{\cM}{\mathcal{M}}
\newcommand{\cN}{\mathcal{N}}

\newcommand{\C}{\mathbb{C}}

\newcommand{\Q}{\mathbb{Q}}

\newcommand{\xn}{{\xi_{\mathcal N}}}

\newcommand{\Clm}{C_{n,m}}
\newcommand{\ClmN}{C_{n,m}^{\bar{\cN}}}

\newcommand{\ccN}{N^C}

\newcommand{\FJ}{\mathscr F_{\bar{i},\bar{N}}}
\newcommand{\LJ}{\mathscr L_{\bar{i},\bar{N}}}

\newcommand{\HJi}{H^{\bar{\cN}(\bar{N})}_{n,m_J(\bar{N})}}
\newcommand{\HJdi}{H^{\bar{\cN}(\bar{N}),\partial}_{n,m_J(\bar{N})}}

\newcommand{\sJt}{\psi^{\cN}_{\bar{N}}}
\newcommand{\sJtk}{\psi^{\cN,J}_{N}}
\newcommand{\sJto}{\psi^{\cN,J,k}_{N}}
\newcommand{\sJtkt}{\tilde{\psi}^{\cN,J}}
\newcommand{\sJtot}{\tilde{\psi}^{\cN,J,k}}

\newcommand{\sAt}{\psi^{\cN}_{\cM}}
\newcommand{\sAto}{\psi^{A}_{\cM}}

\newcommand{\sJtt}{\psi^{1,\cN,\bar{N}^b}_{q,\bar{N}^{\alpha},\overline{[N]}}}
\newcommand{\sAtt}{\psi^{1-\cM,\cN,\bar{\cM}^b}_{\xi_{\cM},\bar{\lambda},\overline{d(\lambda)}}}

\newcommand{\sAtuJ}{\psi^{J}_{\bar{N}}}

\newcommand{\HAi}{H^{\bar{\cN}}_{n,m_A(\cN)}}
\newcommand{\HAdi}{H^{\bar{\cN},\partial}_{n,m_A(\cN)}}

\newcommand{\HA}{H^{\bar{\cN}}_{n,2+(n-1)(\cN-1)}}
\newcommand{\HAd}{H^{\bar{\cN},\partial}_{n,2+(n-1)(\cN-1)}}

\renewcommand{\geq}{\geqslant}
\renewcommand{\leq}{\leqslant}
\renewcommand{\footnoterule}{%
  \kern -3pt
  \hrule width \textwidth height 0.4pt
  \kern 2.6pt
}

\definecolor{dgreen}{RGB}{0,150,0}

\begin{document}
\title{\vspace{-14mm} \Large\bfseries Universal geometrical link invariants}
\author{ \small Cristina Anghel \quad $/\!\!/$\quad \todaysdate\vspace{-3ex}}
\date{}
\maketitle
{
\makeatletter
\renewcommand*{\BHFN@OldMakefntext}{}
\makeatother
}
\vspace{-10mm}
\begin{abstract}
We construct geometrically two {\bf \em universal link invariants}:  universal ADO invariant and universal Jones invariant, as limits of invariants given by graded intersections in configuration spaces. More specifically, for a fixed level $\cN$, we define new link invariants: {\bf \em``$\cN^{th}$ Unified Jones invariant''} and {\bf \em``$\cN^{th}$ Unified Alexander invariant''}. They globalise topologically all coloured Jones polynomials for links with multi-colours bounded by $\cN$ and all ADO polynomials with bounded colours. These invariants both come from the same { \bf \em weighted Lagrangian intersection} supported on configurations on {\bf \em arcs and ovals} in the disc.
 
The question of providing a universal non semi-simple link invariant, recovering all the ADO polynomials was an open problem. A parallel question about semi-simple invariants for the case of knots is the subject of Habiro's famous universal knot invariant \cite{H3}. Habiro's universal construction is well defined for knots and can be extended just for certain classes of links. Our universal Jones link invariant is defined for any link and recovers all coloured Jones polynomials, providing a {\bf \em new semi-simple} universal link invariant. The {\bf \em first non semi-simple} universal link invariant that we construct unifies all ADO invariants for links, answering the open problem about the globalisation of these invariants.  


 \end{abstract}
 
 \vspace{-2mm}

\section{Introduction}\label{introduction}

\vspace{-3mm}
Coloured Jones and coloured Alexander polynomials are two sequences of quantum link invariants originating in representation theory (\cite{J},\cite{RT},\cite{Witt}). The geometric information encoded by these invariants is a fundamental problem in quantum topology. Conjectures from physics such as the Volume Conjecture (Kashaev \cite{K},\cite{M2}) and Gukov-Manolescu's Conjecture (\cite{GM}) state that these invariants recover asymptotically geometry of knot complements. Coming from representation theory, Habiro defined at this asymptotic level a universal knot invariant (\cite{H3}) providing a unification of all coloured Jones polynomials for {\em knots}. 
This lead to his celebrated unification of Witten-Reshetikhin-Turaev invariants for homology spheres (\cite{H3}). 
Passing from knots to general link invariants, such asymptotic results remained open. 

{\bf Open problem:} {\bf \em Unification of coloured Alexander link invariants}\\
Can one construct a unification of coloured Alexander invariants for coloured links?\\
(Parallel to Habiro's famous program unifying coloured Jones polynomials for knots \cite{H3}). 

Motivated by the conjectured geometry that is beside this story, our goal is to study quantum invariants from a new topological viewpoint: as graded intersections in configuration spaces. Such models appeared in \cite{Big}, \cite{Law}, \cite{SM}, \cite{MM}, \cite{CrI}, \cite{Cr1}, \cite{Cr3}, \cite{Crsym}, \cite{Cr2}, \cite{WRT}, \cite{CrG}. Our main result is an answer to the above open problem that comes from representation theory, using purely topological tools.
\begin{thm}[\bf \em Universal ADO link invariant]\label{THEOREMUniv}\label{TH}
We construct a geometric link invariant $\I$, in a completion of a polynomial ring $\LLh$ (\eqref{rna}, \eqref{ra}), that recovers all coloured Alexander link invariants. $\I$ is defined geometrically, as a limit of intersections in configuration spaces:\begin{equation}
\I:= \underset{\longleftarrow}{\mathrm{lim}} \ \PAA \in \LLh
\end{equation}
where $\PAA$ is a link invariant recovering all ADO invariants at levels bounded by $\cN$.
\end{thm}
\vspace{-2mm}
\begin{thm}[\bf \em Universal Jones link invariant]\label{THEOREMUnivJ}\label{THJ}
We define a link invariant $\IJJ$ taking values in a universal ring $\LLhJ$ (\eqref{rnj}, \eqref{rj}) that recovers all coloured Jones link polynomials. 
This invariant $\IJJ$ is a limit of link invariants defined in configuration spaces:\begin{equation}
\IJJ:= \underset{\longleftarrow}{\mathrm{lim}} \ \PAAJ \in \LLhJ
\end{equation}
where $\PAAJ$ is a link invariant recovering all coloured Jones polynomials for links with multicolours that are all bounded by $\cN$.

\end{thm}
Our strategy is to build these universal invariants via two sequences of link invariants that unify more and more coloured Jones and coloured Alexander link polynomials, as below.

\begin{thm}[{\bf New level $\cN$ link invariants}] For each $\cN$, we construct geometrically two link invariants $\PAA$ and $\PAAJ$ via the same Lagrangian intersection in a fixed configuration space.

$\bullet$ We call $\PAA$ the {\color{blue} $\cN^{th}$ unified ADO invariant} and show that this link invariant recovers all ADO polynomials up to level $\cN$ (as in Theorem \ref{INV}).

$\bullet$ We call $\PAAJ$ {\color{blue} the $\cN^{th}$ unified Jones invariant} and prove that it unifies all coloured Jones polynomials up to level $\cN$ (as in Theorem \ref{INVJ}).
\end{thm}
\vspace{-2mm}
More precisely, these invariants will be defined via a {\em weighted Lagrangian intersection} $$\PA \in \LL=\Z[w^1_1,...,w^{l}_{\cN-1},u_1^{\pm 1},...,u_l^{\pm 1},x_1^{\pm 1},...,x_l^{\pm 1},y^{\pm 1}, d^{\pm 1}]$$ in the configuration space of $(n-1)(\cN-1)+2$ points in the disc (presented in   Subsection  \ref{weightedintersection}), seen in two different quotient rings of $\LL$, described in relations \eqref{rnj} and \eqref{rna}.  

\begin{figure}[H]
\begin{center}
\hspace*{-6mm}\begin{tikzpicture}
[x=1.2mm,y=1.4mm]
\node (Al)               at (10,10)    {$ \PAA \in \ \LLN$};
\node (Jl)               at (-65,10)    {$ \LLNJ \ni \PAAJ$};

\node (Jn)               at (-40,-2)    {$J_{\cN,...,\cN}(L)\ \ \ \cdots$};
\node (Jn-1)               at (-65,-2)    {$J_{N_1,...,N_l}(L)$};
\node (Jm)               at (-85,-2)    {$J_{2,...,2}(L)$};


\node (An)               at (28,-2)    {${\Phi^{\cN}}(L) \ \ \ \ \cdots$};
\node (An-1)               at (10,-2)    {$\Phi^{\cN-1}(L)$};
\node (Am)               at (-12,-2)    {${\Phi^{2}}(L) $};

\node(IA)[draw,rectangle,inner sep=3pt,color=red] at (20,19) {\begin{tabular}{l} $\cN^{th}$ unified ADO\\ \ \ \ \ invariant \end{tabular}};
\node(IJ)[draw,rectangle,inner sep=3pt,color=red] at (-80,19) {\begin{tabular}{l} $\cN^{th}$ unified Jones\\ \ \ \ \ invariant \end{tabular}};

\node(LI)[draw,rectangle,inner sep=3pt,color=blue] at (-27,18) {\begin{tabular}{l}
 \ \ \ \ \ \ \ \ Weighted \\ Lagrangian intersection \\ $ \  \color{black} \PA \in \LL \ (\text{Def.} \ref{D:w})$ 
\end{tabular}};

\node(UA)[draw,rectangle,inner sep=3pt,color=red] at (5,35) {Universal ADO link invariant};
\node(UJ)[draw,rectangle,inner sep=3pt,color=red] at (-60,35) {Universal Jones link invariant};

\node(Q)[draw,rectangle,inner sep=3pt,color=red, minimum width = 6.9cm, 
    minimum height = 1.1cm] at (-63,-4) {$\phantom{A}$};
\node(Q)[draw,rectangle,inner sep=3pt,color=red, minimum width = 6.2cm, 
    minimum height = 1.1cm] at (6,-4) {$\phantom{A}$};

\node (J)               at (-50,30)    {$\IJJ \in \LLhJ$};
\node (A)               at (-10,30)    {$\I \in \LLh$};

\draw[->]             (Jl)      to node[right,xshift=2mm,font=\large]{}   (Jn);
\draw[->]             (Jl)      to node[right,xshift=2mm,font=\large]{}   (Jn-1);
\draw[->]             (Jl)      to node[right,xshift=2mm,font=\large]{}   (Jm);

\draw[->]             (J)      to node[right,xshift=2mm,font=\large]{}   (Jl);


\draw[->]             (A)      to node[right,xshift=2mm,font=\large]{}   (Al);

\draw[->]             (Al)      to node[right,xshift=2mm,font=\large]{}   (An);
\draw[->]             (Al)      to node[right,xshift=2mm,font=\large]{}   (An-1);
\draw[->]             (Al)      to node[right,xshift=2mm,font=\large]{}   (Am);

\draw[->,dashed, red]             (IJ)      to node[right,xshift=2mm,font=\large]{$\underset{\longleftarrow}{\mathrm{lim}}$}   (UJ);
\draw[->,dashed, red]             (IA)      to node[right,xshift=2mm,font=\large]{$\underset{\longleftarrow}{\mathrm{lim}}$}   (UA);

\draw[->,dashed, blue]             (LI)      to node[right,xshift=2mm,font=\large]{$$}   (Jl);
\draw[->,dashed, blue]             (LI)      to node[right,xshift=2mm,font=\large]{$$}   (Al);

\end{tikzpicture}
\vspace*{-10mm}
$$\hspace{-20mm} \ \ \ \ \ \ \ \ \ \ \ \ \text{Coloured Jones polynomials} \hspace{35mm} \ \ \ \ \ \ \text{     Coloured Alexander polynomials}$$

\end{center}
\vspace{-3mm}
\caption{Two geometrical universal link invariants}
\end{figure}
\vspace{-6mm}
\subsection{Motivation and Results}

Coloured Jones polynomials come from the semi-simple representation theory of $U_q(sl(2))$, and they associate to a $l$-component link together with $l$ colours $N_1,...,N_l \in \N$ a one-variable polynomial $J_{N_1,...,N_l}(L)$. Dually, for $\cN \in \N$ the quantum group at the root of unity $\xi_N=e^{\frac{2 \pi i}{2\cN}}$ gives non-semisimple link invariants, in $l$ variables, called coloured Alexander polynomials $\Phi^{\cN}(L)$ (or ADO \cite{ADO}). 
These coloured link invariants allow the construction of 3-manifold invariants: the Witten-Reshetikhin-Turaev invariant ($WRT_{\cN}(M)$) and the Costantino-Geer-Patureau invariant ($CGP_{\cN}(M)$). These are powerful invariants that recover the Reidemeister torsion and detect lens spaces (\cite{W}, \cite{CGP14}, \cite{BCGP}). A 3-dimensional topological description and categorification for $WRT$ and $CGP$ invariants are major problems in quantum topology. As discussed, from representation theory's perspective, the story for asymptotics of quantum link invariants is different to the knot case. Habiro's invariant is defined for the first sequence of semi-simple knot invariants. The existence of a universal invariant remained an {\bf \em open problem for the case of links} in both semi-simple and non semi-simple setting.

 \begin{itemize}
\item{\bf Open Question 1} Construct {\em universal Jones link invariants} recovering all coloured Jones polynomials for {\em coloured links rather than for knots}.\\ Up to this moment no such model is known for the coloured Jones invariants for {\em links}. 
\item{\bf Open Question 2} Construct {\em universal ADO link invariants}.\\ Up to now no such model was known for coloured Alexander invariants (even for knots).  
\item{\bf Question 3} Are there {\em unifications of CGP invariants for $3$-manifolds}? 
\end{itemize}

 We answer the first two open questions. The answers to {Question 1} and {Question 2}  provide a new perspective on universal link invariants, from a purely topological viewpoint. We provide the first unification for non-semisimple link invariants and this is the building block for a sequel paper on topological models for non-semisimple $3$-manifold invariants, opening avenues for {Question 3}.
\subsection{Idea- First universal invariants for links via weighted intersections}
In order to construct the unification for the link case, our strategy is to define {\em weighted Lagrangian intersections}. The core idea is to add weights that provide additional variables. For knots, if we count all the weights as $1$ we obtain our (non-weighted) universal geometrical Jones invariant from \cite{CrI}. 
However, this invariant as well as Habiro's invariant, could {\em not be extended to the link case}. The addition of the weights leads to a key point: Theorem \ref{THEOREMAU}. This creates a {\em geometric perspective} that allows us to read all coloured Alexander and coloured Jones invariants of level lower than $\cN$ from the set of intersections between Lagrangian submanifolds in a fixed configuration space.

This unification is not immediately expected from representation theoretic perspectives. Our novel topological tools open avenues on algebraic aspects of unifications, from the quantum group point of view. This creates new interactions between representation theory and topology, and shows that in the case of this open problem topology has allowed us to understand algebra more deeply.
\vspace*{-2mm}

\subsection{Summary of the construction}
\subsubsection{Weighted intersection}  
\label{ssAu}
 
 We look at links $L$ as closures of braids $\beta_n$ with $n$ strands. For a colouring $N_1,...,N_l \leq \cN$, let $\bar{N}:=(N_1,...,N_l)$ and we say that $\bar{N} \leq \cN$ if $N_1,...,N_l \leq \cN$. \\
 For a fixed level $\cN$, we construct a {\em weighted Lagrangian intersection}: $$\PA \in \LL=\Z[w^1_1,...,w^{l}_{\cN-1},u_1^{\pm 1},...,u_l^{\pm 1},x_1^{\pm 1},...,x_l^{\pm 1},y^{\pm 1}, d^{\pm 1}]$$ in the configuration space of $(n-1)(\cN-1)+2$ points in the disc (see  Subsection  \ref{weightedintersection}). This weighted intersection $\PA$ is parametrised by a set of intersections between Lagrangian submanifolds, as in Figure \ref{ColouredAlex0} and weighted in a subtle manner via the variables of the ring $\LL$ (see Definition \ref{D:w}).  
\begin{thm}[{\bf \em Unifying coloured Alexander and coloured Jones polynomials}]\label{THEOREMAU}
For a fixed level $\cN\in \N$, $\PA$ recovers all coloured Alexander and all coloured Jones polynomials for links with (multi)-colours bounded by $\cN$, as below:
\begin{equation}
\begin{aligned}
&\Phi^{\cM}(L) =~ \PA \Bigm| _{\sAt}, \ \forall \cM\leq \cN; \ \ \ J_{\bar{N}}(L) =~ \PA \Bigm| _{\sJt}, \forall \bar{N}=(N_1,...,N_l) \leq \cN.
\end{aligned}
\end{equation} 
\end{thm}
{\subsubsection{Universal rings} 
The next part of the construction is dedicated to the definition of two sequences of nested ideals in $\LL$:
$$\cdots \supseteq \tilde{I}_{\cN} \supseteq \tilde{I}_{\cN+1} \supseteq \cdots \text{ and }
\cdots \supseteq \tilde{I}^J_{\cN} \supseteq \tilde{I}^J_{\cN+1} \supseteq \cdots$$ described explicitely in Proposition \ref{structure}. In this manner, we obtain two sequences quotient rings:
$$ \ \ \ \ \ \ \ \ \cdots \LLN \leftarrow \LLNN \leftarrow \cdots;  \ \ \ \ \ \ \ \ \ \ \ \ \cdots \LLNJ \leftarrow \LLNNJ \leftarrow \cdots  \ \ \ (\text{see formulas } \eqref{rnj}, \eqref{rna}).$$ 
\vspace*{-5mm}
\subsubsection{New invariants at level $\cN$} So far, we have the weighted Lagrangian intersection $\PA\in \LL$ that specialises to all coloured Alexander and coloured Jones polynomials of $L$ with bounded colours, which  is defined via a braid representative. Next we prove that this intersection leads to two sequences of link invariants, in the quotient rings (see Theorem \ref{levN} and Theorem \ref{levNJ}).  
\begin{thm}[\bf \em $\cN^{th}$ Unified Jones link invariant] \label{INV} Let $\PAAJ$ be the image of the intersection form $\PA$ in $\LLNJ$.
 Then, $\PAAJ$ is a well-defined link invariant recovering all coloured Jones polynomials with multi-colours up to level $\cN$:
$$  \PAAJ|_{\smnbJt}=J_{\bar{N}}(L), \ \ \ \forall \bar{N} \leq \cN      \ \ (\text{see } \eqref{rnj}, \ \eqref{u1''}).$$
\end{thm}
\begin{thm}[\bf \em $\cN^{th}$ Unified ADO link invariant] \label{INVJ}
Let $\PAA$ be the image of $\PA$ in $\LLN$. Then, $\PAA$ is a well-defined link invariant unifying all ADO polynomials up to level $\cN$:
$$  \PAA|_{\smnbt}=\AM, \ \ \ \forall \cM \leq \cN, \ \  (\text{see } \eqref{rna}, \ \eqref{u1'}).$$
\end{thm}
Let us consider the projective limits of these sequences of rings:
$\LLh:= \underset{\longleftarrow}{\mathrm{lim}} \ \LLN; \ \LLhJ:= \underset{\longleftarrow}{\mathrm{lim}} \ \LLNJ.$ 
Then, in Theorem \ref{TH} and \ref{THJ} we show that the $\cN^{th}$ unified ADO / Jones link invariants have good asymptotic behaviour, leading to well defined invariants in these projective limit rings.

\subsection{Geometric encoding of semi-simplicity versus non semi-simplicity}
The construction of non-semisimple quantum invariants from representation theory perspectives uses as building blocks modified quantum dimensions. In Section \ref{ssmodif} we create a dictionary and explain how we codify these essential algebraic tools through our topological lenses. More precisely, our topological tools are the {\em  weighted Lagrangian intersections} which, in turn come from local systems on configuration spaces of the punctured disc. We present how to encode the algebraic origin of the construction given by modified dimensions through the geometric information provided by the {\em monodromies of our local systems}.

{\bf \em Dictionary: geometric variables and quantum tools}

 The intersection $\PA$ is parametrised by intersection points in the configuration space and graded by $5$ types of variables: $\{\bar{x}, \ \ \ \ y, \ \ \ \  d, \ \ \ \ \bar{u}, \ \ \ \ \bar{w}\}.$
We have the following correspondence:
\begin{enumerate}
\item[•]{\color{blue}  \em Variables of the polynomial $\bar{x}$ $\longleftrightarrow$ Winding around the link} 
\item[] {\em $\bar{x}$} encode {linking numbers with the link} via monodromies around $p$-punctures
\item[]{\color{blue}  \em Quantum variable $d$ $\longleftrightarrow$ Relative twisting}
\item[•] {$d$} counts a relative twisting in the configuration space.
\item[•]{\color{carmine}  \em Quantum variable $y$ $\longleftrightarrow$ Modified dimension}
\item[] {$y$} counts the winding number around the {\color{carmine} $s$-puncture}, and globalises modified dimensions
\item[•]{\color{carmine}   \em Quantum variables $\bar{u}$ $\longleftrightarrow$ Pivotal structure}
\item[] {\em $\bar{u}$} capture the difference between {semisimplicity versus non-semisimplicity}, and gets specialised to a power of ${\bar{x}}$, meaning a {power of the linking number with the link}.
\item[•]{\color{carmine}  \em New quantum weights $\bar{w}$ $\longleftrightarrow$ unification of all quantum levels}
\item[] {\em $\bar{w}$} make it possible to unify and see all quantum invariants of colours bounded by our fixed level $\cN$ directly from one topological viewpoint: the weighted intersection $\PA$. 
\end{enumerate}
\vspace{-2mm}
\begin{figure}[H]
$ \hspace{25mm} \small \LL=\Z[w^1_1,...,w^{l}_{\cN-1},u_1^{\pm 1},...,u_l^{\pm 1},x_1^{\pm 1},...,x_l^{\pm 1},y^{\pm 1}, d^{\pm 1}] $
\begin{equation}
\begin{aligned}
& \ \ \ \ \ \ \ \ \ \ \hspace{-20mm} \LLhJ= \underset{\longleftarrow}{\mathrm{lim}} \ \LLNJ & & \hspace{40mm} \LLh= \underset{\longleftarrow}{\mathrm{lim}} \ \LLN { \ \ \ \ \ \ \ \color{gray} Eq. \eqref{rj}, \eqref{ra}} \\
& \ \ \ \ \ \ \ \ \ \ \hspace{-20mm} \LLNJ= \LL /\tilde{I}^{J}_{\cN} & & \hspace{40mm}\LLN= \LL/\tilde{I}_{\cN} { \ \ \ \ \ \ \color{gray} Eq. \eqref{rnj}, \eqref{rna}}.
\end{aligned}
\end{equation}

\begin{center}
\hspace*{-6mm}\begin{tikzpicture}
[x=1.1mm,y=1.1mm]
\node (Al)               at (16,10)    {$ \PAA \in \ \LLN$};
\node (Jl)               at (-72,10)    {$ \LLNJ \ni \PAAJ$};

\node (Jn)               at (-52,-3)    {$J_{\cN,...,\cN}(L)$};
\node (Jn-1)               at (-52.5,-13)    {$ \ \ \ \ J_{\cN_1,...,\cN_l}(L)$};
\node (Jm)               at (-56,-20)    {$ \ \ \ \ \ \ \ \ \ J_{2,...,2}(L)$};

\node (An)               at (-10,-3)    {${\AN}$};
\node (An-1)               at (-10,-13)    {${\Phi^{\cM}}$};
\node (Am)               at (-10,-20)    {${\Phi^2}(L)$};

\node(V1)[draw,rectangle,inner sep=3pt,color=vdarkred,text width=2.84cm,minimum height=0.5cm] at (-30,17) {\footnotesize{Knots vs links \\ Quantum weight $\bar{w}$} };

\node(V2)[draw,rectangle,inner sep=3pt,color=vdarkred,text width=2.84cm,minimum height=0.5cm] at (-30,7) {\footnotesize{(Non) semi-simplicity \\ Quantum variable $\bar{u}$} };

\node(V2)[draw,rectangle,inner sep=3pt,color=vdarkred,text width=2.84cm,minimum height=0.5cm] at (-30,-3) {\footnotesize{Modified dimension \\ Quantum variable $\bar{y}$} };


\node(IA)[draw,rectangle,inner sep=3pt,color=red] at (20,15) {$\cN^{th}$ unified ADO inv.};
\node(IJ)[draw,rectangle,inner sep=3pt,color=red] at (-70,15) {$\cN^{th}$ unified Jones inv.};

\node(UA)[draw,rectangle,inner sep=3pt,color=red] at (0,30) {Universal ADO link invariant};
\node(UJ)[draw,rectangle,inner sep=3pt,color=red] at (-60,30) {Universal Jones link invariant};

\node (J)               at (-52,25)    {$\IJJ \in \LLhJ$};
\node (A)               at (-10,25)    {$\I \in \LLh$};

\draw[->]             (Jl)      to node[right,xshift=2mm,font=\large]{}   (Jn);
\draw[->]             (Jl)      to node[right,xshift=2mm,font=\large]{}   (Jn-1);
\draw[->]             (Jl)      to node[right,xshift=2mm,font=\large]{}   (Jm);

\draw[->]             (J)      to node[right,xshift=2mm,font=\large]{}   (Jn);


\draw[->]             (A)      to node[right,xshift=2mm,font=\large]{}   (An);

\draw[->]             (Al)      to node[right,xshift=2mm,font=\large]{}   (An);
\draw[->]             (Al)      to node[right,xshift=2mm,font=\large]{}   (An-1);
\draw[->]             (Al)      to node[right,xshift=2mm,font=\large]{}   (Am);

\draw[->,dashed, red]             (IJ)      to node[left,xshift=-2mm,font=\large]{$\underset{\longleftarrow}{\mathrm{lim}}$}   (UJ);
\draw[->>,dashed, red]             (IA)      to node[right,xshift=2mm,font=\large]{$\underset{\longleftarrow}{\mathrm{lim}}$}   (UA);

\end{tikzpicture}
\end{center}

\vspace{-4mm}
\caption{Geometrical encoding of non semi-simplicity}\label{comp'}
\end{figure}
\vspace{-8mm}
\subsection{Three different Universal knot invariants} In this part, we focus on the case of knots, where we compare three constructions:
\vspace{-1mm}
\begin{itemize}
\item the Weighted universal Jones invariant from Theorem \ref{THJ} for the case of knots: $\IJJw\in \LLhJ$ 
\item the Non-weighted universal Jones invariant that we previously defined in \cite{Cr1}: $\IJJk\in \LLhJk$
\item Habiro's universal invariant: $\IW\in \LLhW$ (seen in the ring used by Willetts \cite{W}, \cite{H1}, \cite{H2}).
\end{itemize}
\vspace{-1mm}
These invariants come from different perspectives. The two geometric universal Jones invariants are defined as limits of invariants globalising coloured Jones polynomials. Habiro's invariant $\IW$  originates in Lawrence's construction being defined directly in the limit ring. 

\vspace{-9mm}
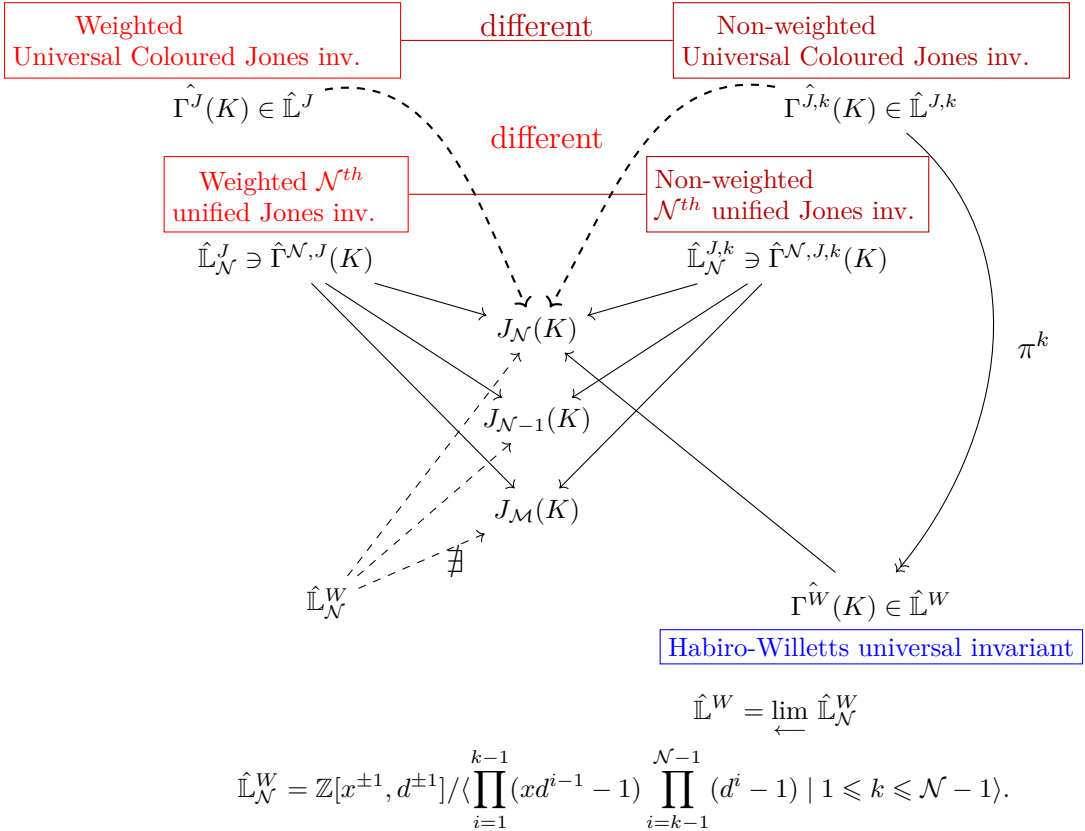
\begin{figure}[H]
\begin{equation*}
\begin{aligned}
&\hspace{2mm} \LLhJ= \underset{\longleftarrow}{\mathrm{lim}} \ \LLNJ \hspace{79mm}\LLhJk= \underset{\longleftarrow}{\mathrm{lim}} \ \LLNJk\\  
&\hspace{-9mm} \LLNJ= \Z[w^1_1,...,w^{1}_{\cN-1},x^{\pm1}, d^{\pm 1}]/  \hspace{58mm}\LLNJk= \Z[x^{\pm1}, d^{\pm 1}] / \\
&\hspace{-9mm} \bigcap_{N \leq \cN} \langle \left. x-d^{1-N}, w^1_j-1 \text{ if } j\leq N-1, \right. \left. w^1_j \ \text{ if } j\geq N, j\in \{1,...,\cN-1\} \right) \rangle\hspace{5mm} \langle \prod_{i=2}^{\cN} (xd^{i-1}-1) \rangle \end{aligned}
\end{equation*}
\begin{center}
\hspace*{-30mm}\begin{tikzpicture}
[x=1mm,y=1mm]
\node (Jl)               at (0,6)    {$ \LLNJk \ni \PAAJk$};
\node (Jlk)               at (-60,6)    {$ \LLNJ \ni \PAAJw$};

\node (Wn)               at (-55,-27)    {$\LLNW$};

\node (Jn)               at (-30,0)    {$J_{\cN}(K)$};
\node (Jn-1)               at (-30,-10)    {$J_{\cN-1}(K)$};
\node (Jm)               at (-30,-20)    {$J_{\cM}(K)$};

\node(IJ)[draw,rectangle,inner sep=3pt,color=vdarkred,text width=3.5cm,minimum height=1cm] at (0,15) {Non-weighted $\cN^{th}$ \\ unified Jones invariant};
\node(IJw)[draw,rectangle,inner sep=3pt,color=red,text width=3.5cm,minimum height=1cm] at (-60,15) { \ \ \ Weighted $\cN^{th}$ \\  unified Jones invariant};
\node(UJ)[draw,rectangle,inner sep=3pt,color=vdarkred,text width=5.5cm,minimum height=1cm] at (10,36) { \ \ \ \ \ \ \ \ \ \ \ \ \ Non-weighted \\ Universal Coloured Jones invariant};
\node(UJw)[draw,rectangle,inner sep=3pt,color=red,text width=5.5cm,minimum height=1cm] at (-70,36) { \ \ \ \ \ \ \ \ \ \ \ \ \ Weighted \\ Universal Coloured Jones invariant};
\node[draw,rectangle,inner sep=3pt,color=blue] at (10,-35) {Habiro's universal invariant};

\node (J)               at (9,27)    {$\IJJk \in \LLhJk$};
\node (Jw)               at (-65,27)    {$\IJJw \in \LLhJ$};
\node (W)               at (10,-27)    {$\IW \in \LLhW$};

\draw[->,dashed]             (Wn)      to node[right,xshift=2mm,font=\large]{}   (Jn);
\draw[->,dashed]             (Wn)      to node[right,xshift=2mm,font=\large]{}   (Jn-1);
\draw[->,dashed]             (Wn)      to node[right,xshift=2mm,font=\large]{$\nexists$}   (Jm);

\draw[->]             (Jl)      to node[right,xshift=2mm,font=\large]{}   (Jn);
\draw[->]             (Jl)      to node[right,xshift=2mm,font=\large]{}   (Jn-1);
\draw[->]             (Jl)      to node[right,xshift=2mm,font=\large]{}   (Jm);
\draw[->]             (Jlk)      to node[right,xshift=2mm,font=\large]{}   (Jn);
\draw[->]             (Jlk)      to node[right,xshift=2mm,font=\large]{}   (Jn-1);
\draw[->]             (Jlk)      to node[right,xshift=2mm,font=\large]{}   (Jm);

\draw[->,,out=-188, in=65,dashed,thick]             (J)      to node[right,xshift=2mm,font=\large]{}   (Jn);
\draw[->,,out=5, in=115,dashed,thick]             (Jw)      to node[right,xshift=2mm,yshift=2mm,font=\large]{\color{red} different}   (Jn);
\draw[->]             (W)      to node[right,xshift=2mm,font=\large]{}   (Jn);



\draw[->>,out=-30,in=40]             (J)      to node[right,xshift=2mm,font=\large]{$\pi^k$}   (W);
\draw[-, vdarkred]             (IJw)      to node[right,xshift=2mm,font=\large]{}   (IJ);
\draw[-, vdarkred]             (UJ)      to node[yshift=4mm,font=\large]{\text{different}}   (UJw);

\end{tikzpicture}
\hspace*{-20mm}\begin{equation}\label{WR}
\LLhW=\underset{\longleftarrow}{ \mathrm{lim }} \ \LLNW;  \ \hspace{5mm}\LLNW= \Z[x^{\pm1}, d^{\pm 1}] / \langle \prod_{i=1}^{k-1} (xd^{i-1}-1) \prod_{i=k-1}^{\cN-1} (d^i-1) \mid 1\leq k \leq \cN-1 \rangle.\end{equation}
\end{center}
\vspace{-4mm}
\caption{Three different universal Jones invariants for knots}\label{re'}
\end{figure}
 \vspace{-8mm}
\subsubsection{Recovering Habiro's universal invariant}  We remark a difference between the intermediary invariants that are used for these constructions. The Non-weighted invariant
$\IJJk$ as well as the Weighted invariant $\PAAJw$ are limits of the (Non-)Weighted Unified Jones invariants $\PAAJk\in \LLNJk$ and $\PAAJk\in \LLNJk$ respectively. Each such level $\cN$ invariant globalises coloured Jones poynomials:
$$\PAAJw \Bigm| _{\color{red}x=d^{\cM-1}}  =~ \PAAJk \Bigm| _{\color{red}x=d^{\cM-1}}  =~ {\color{red} J_{\cM}(K)}, \ \ \ \forall \cM\leq \cN.$$
On the other hand, the $\cN^{th}$ level of Habiro's invariant $\PAAW \in \LLNW$ does not have well-defined specialisations at natural parameters:
\begin{equation}
\PAAW \Bigm| _{\color{blue}{x=d^{\cM-1}}} \text{not well defined}.
\end{equation}
However, at the limit these constructions are related and we defined in \cite{CrI} a map $\pi^k: \LLhJk \rightarrow \  \LLhW$ showing that the universal Jones invariant recovers Habiro's invariant:
 $\pi^k(\IJJk)=\IW$.
\vspace*{-2mm}
\subsubsection{Two different geometric universal Jones invariants for knots } Section \ref{S:knotcase} shows that the weighted construction provides new knot invariants, different than the non-weighted ones.

First we prove in Theorem \ref{r1} that $\PAAJw$ and $\PAAJk$ can be both seen as images of the weighted intersection $\PA \in \LL$. However, the {$\cN^{th}$ Weighted Unified Jones invariant} and the {$\cN^{th}$ Non-weighted Unified Jones invariant} are different: $$\PAAJw \neq\PAAJk.$$
We obtain two different universal Jones invariants for knots: $\IJJw \neq \IJJk$, both given by sequences of invariants that globalise more and more coloured Jones polynomials (Theorem \ref{r2}).
\subsection{Further directions}
\vspace*{-4mm}
\subsubsection{Refined invariants in modules over Habiro type rings} In Subsection \ref{ss:refH} we introduce generalisations of Habiro type rings, which we call the quantised and extended Habiro rings:
 \vspace{-2mm}
 \begin{equation*}
\begin{aligned}
&{\text{\em \color{blue} Quantised Habiro ring \ \ }} \hat{\mathbb L}^{H,q}=\underset{\longleftarrow}{\mathrm{lim}} \ \Z[x_1^{\pm1}, \cdots, x_l^{\pm1}, d^{\pm 1}] \ / \ \langle \prod_{j=2}^{\cN} (x_i d^{j-1}-1),1 \leq i \leq l \rangle\\
&\text{{\em \color{blue} Extended Habiro ring  \ \  \ \ }}\hat{\mathbb L}^{H,e}= \ \underset{\longleftarrow}{\mathrm{lim}} \ \Z[x_1^{\pm1}, \cdots, x_l^{\pm1}, d^{\pm 1}] \ / \ \langle \prod_{\cM=2}^{\cN} (d^{2\cM}-1) \rangle.
\end{aligned}
\end{equation*}
 \vspace{-1mm}

Then we extend our universal rings $\LLhJ$ and $\LLh$ introducing refined universal rings ${\LLhJ}'$ and ${\LLh}'$. These refined rings have a rich structure, being modules over the generalised Habiro rings, and they surject onto our initial universal rings. 
 We conjecture that our universal invariants lift to refined universal invariants, as follows.
\begin{conjecture}[\bf \em Refined universal Jones invariant]\label{l1} The universal Jones invariant $\IJJ$ lifts to a refined invariant $\IRJJ(L)$ in  ${\LLhJ}'$, that is a module over the {quantised Habiro ring} $\hat{\mathbb L}^{H,e}$.
\end{conjecture}
\begin{conjecture}[\bf \em Refined universal ADO invariant]\label{l2} The universal ADO invariant $\I$ lifts to a refined invariant $\IR(L)$ in the ring ${\LLh}'$, that is a module over the {extended Habiro ring} $\hat{\mathbb L}^{H,e}$.
\end{conjecture}

\vspace*{-6mm}
\subsubsection{Quantum Perspectives- Universal quantum group}

From representation theory's perspective, Habiro's invariant comes from Lawrence's universal invariant via actions on Verma modules and quotients through well-chosen rings. We expect that our new weighted model has a counterpart permitting the construction of invariants on the algebraic side. More specifically, our weighted formula should have a representation-theoretic counterpart and we believe that this comes from a weighted version of the universal invariant, defined by weighted actions on Verma modules.

\vspace*{-4mm}
\subsubsection{Both Universal geometrical invariants via the same infinite configuration space}

In Subsection \ref{asym} we discuss how our two universal invariants are asymptotics of the same Lagrangian intersections. We show that both invariants
$\PAA$ and $ \PAAJ$ are both given by graded intersections in the configuration space of $(n-1)(\cN-1)+2$ points in the disc, with good behaviour with respect to the change of levels. This stability phenomenon should be reflected asymptotically. We conjecture that both the universal Jones link invariant $\IJJ$ and the universal ADO link invariant $\I$ can be seen via the same Lagrangian intersections in an infinite configuration space. 

\vspace*{-3mm}
\subsubsection{Sequel-  Unification for non-semisimple $3$-manifold invariants}
Passing to $3$-manifolds, we ask whether there is a unification of the $CGP$ invariants, as a parallel to Habiro's celebrated unification of the WRT invariants (\cite{H3})? As a step towards this, our sequel result uses the models from this paper and describes both $WRT$ and $CGP$ invariants at a fixed level from the same perspective: a set of Lagrangian intersections in a fixed configuration space.

\vspace{-3mm}

{\tableofcontents \vspace{-4mm}}

\section{Geometrical set-up and Notations}
\subsection{Geometry of the universal invariants} \label{asym}
Our two Universal invariants are asymptotics of the same Lagrangian intersections. More specifically, $\I$ and $\IJJ$ are limits of the $\cN^{th}$ Unified Jones invariant $\PAAJ$ and the $\cN^{th}$ Unified Alexander invariant $\PAA$ respectively.    
These level $\cN$ invariants come from the set of graded intersections in the configuration space of $(n-1)(\cN-1)+2$ points in the disc:
$$\PAAJ, \ \PAA \longleftrightarrow \lbrace (\beta_{n} \cup {\mathbb I}_{n+2} ) \ { \color{red} \FA} \cap {\color{dgreen} \LA} \rbrace, \text{ for } {\bar{i} \in \{\bar{0},\dots,\overline{\cN-1}\}}$$
weighted with $\cN-1$ weights: $w^1_1,...,w^{l}_{\cN-1}.$
The change of the level from $\cN$ to $\cN+1$ is reflected explicitly in the geometry of the homology classes: we add $n-1$ arcs to the geometric supports of $\FA$ in order to obtain $\FAnn$. From $\LA$ to $\LAnn$ we add one particle on each oval (Figure \ref{ColouredAlexx}). We remark on a nice property: for a fixed index $\bar{i} \in \{\bar{0},\dots,\overline{\cN-1}\}$ the intersections $\left\langle (\beta_{n} \cup {\mathbb I}_{n+2} ) \ { \color{red} \FA} \cap {\color{dgreen} \LA} \right\rangle$
 and  $\left\langle (\beta_{n} \cup {\mathbb I}_{n+2} ) \ { \color{red} \FAnn} \cap {\color{dgreen} \LAnn} \right\rangle$ coincide. This suggest a geometric stability phenomenon of the link invariants $\PAA$ and $\PAAJ$.

\vspace{-5mm}

\hspace{10mm}\begin{figure}[H]
\begin{equation}
\centering
\begin{split}
\begin{tikzpicture}
[x=1mm,y=1.6mm]
\node (tl) at (95,63) {$\phantom{\LLh}$};
\node (ml) at (95,35) {$\LLNN$};
\node (mr) at (115,35) {$\LNN$};
\node (bl) at (95,-10) {$\LLN$};
\node (br) at (115,-10) {$\LN$};
\node (tlJ) at (75,63) {$\phantom{\LLhJ}$};
\node (mlJ) at (75,35) {$\LLNNJ$};
\node (mrJ) at (50,35) {$\LNNJ$};
\node (blJ) at (75,-10) {$\LLNJ$};
\node (brJ) at (50,-10) {$\LNJ$};
\node (mlei) at (ml.north east) [anchor=south west,color=black] {\phantom{A}\vspace*{20mm}\hspace{-190mm}$ \lbrace (\beta_{n} \cup {\mathbb I}_{n+2} ) \ { \color{red} \FAnn} \cap {\color{dgreen} \LAnn} \rbrace_{{\bar{i} \in \{\bar{0},\dots,\overline{\cN}\}}}$};
\node (blei) at (bl.north east) [anchor=south west,color=black] {\phantom{A}\vspace*{20mm}\hspace{-180mm}$ \lbrace (\beta_{n} \cup {\mathbb I}_{n+2} ) \ { \color{red} \FA} \cap {\color{dgreen} \LA} \rbrace_{{\bar{i} \in \{\bar{0},\dots,\overline{\cN-1}\}}}$};
\node (tle) at (tl.south) [anchor=north,color=red] {$\I$};
\node (mle) at (ml.north) [anchor=south,color=red] {$\PAAN$};
\node (ble) at (bl.north) [anchor=south,color=red] {$\PAA$};
\node (mre) at (mr.north west) [anchor=south west,color=red] {$\Phi^{\cN +1}(L)$};
\node (bre) at (br.north west) [anchor=south west,color=blue] {$\Phi^{\cN}(L)$};
\node (tleJ) at (tlJ.south) [anchor=north,color=blue] {$\IJJ$};
\node (mleJ) at (mlJ.north) [anchor=south,color=blue] {$\PAANJ$};
\node (bleJ) at (blJ.north) [anchor=south,color=blue] {$\PAAJ$};
\node (mreJ) at (mrJ.north west) [anchor=south west,color=blue] {$J_{\overline{N+1}}(L)$};
\node (breJ) at (brJ.north west) [anchor=south west,color=blue] {$J_{\bar{N}}(L)$};
\draw[->] (mle) to node[below,font=\small]{} (mre);
\draw[->,dashed,color=white] (ble) to node[below,font=\small,yshift=8mm]{{\hspace{-110mm}\color{dgreen}Add particles on ovals} \hspace{10mm}  \color{black}Add weights $ \ \ \ w^1_{\cN},\cdots, w^l_{\cN}$} (ml);
\draw[->,dashed,color=white] (ble) to node[below,font=\small,xshift=-45mm,yshift=7mm]{{\hspace{-120mm}\color{carmine}Add arcs}} (ml);
\draw[->,dashed,color=white] (ble) to node[below,font=\small,yshift=17mm]{\hspace{40mm}Level $\cN+1$} (ml);
\draw[->] (ble) to node[below,font=\small]{} (bre);
\draw[->] (tle) to node[above,yshift=5mm,font=\small]{} (mre);
\draw[->] (mleJ) to node[below,font=\small]{} (mreJ);
\draw[->,dashed,color=white] (bleJ) to node[below,font=\small,yshift=17mm]{\hspace{40mm}Level $\cN+1$} (mlJ);
\draw[->] (bleJ) to node[below,font=\small]{} (breJ);
\draw[->] (tleJ) to node[above,yshift=5mm,font=\small]{} (mreJ);
\node[draw,rectangle,inner sep=3pt,color=white] at (-5,59) {\phantom{Limit}};
\node[draw,rectangle,inner sep=3pt,color=red] at (110,65) {Universal ADO link invariant};
\node[draw,rectangle,inner sep=3pt,color=blue] at (50,65) {Universal Jones link invariant};
\node[draw,rectangle,inner sep=3pt,color=blue](A) at (68,0) {${\cN}^{th}$ Unified Jones invariant};
\node[draw,rectangle,inner sep=3pt,color=red] at (115,0) {${\cN}^{th}$ Unified ADO invariant};
\node[draw,rectangle,inner sep=3pt,color=blue](B) at (65,30) {${\cN+1}$ Unified Jones invariant};
\node[draw,rectangle,inner sep=3pt,color=red] at (117,30) {${\cN+1}$ Unified ADO invariant};
\draw[->,color=gray, out=20, in=-20] (A) to node[right,xshift=5mm,yshift=4mm ]{Change the level} (B);
\end{tikzpicture}
\end{split}
\end{equation}

\end{figure}
\vspace{-134.4mm}

\begin{figure}[H]
\centering

{\phantom{A} \ \ \ \ \ \ \ \ \ \ \ \ \ \ \ \ \ \ \ \ \ \ \ \ \ \ \ \ \ \ \ \  \ \ \ \ \ \ \ \ \ } \hspace*{-170mm}\includegraphics[scale=0.35]{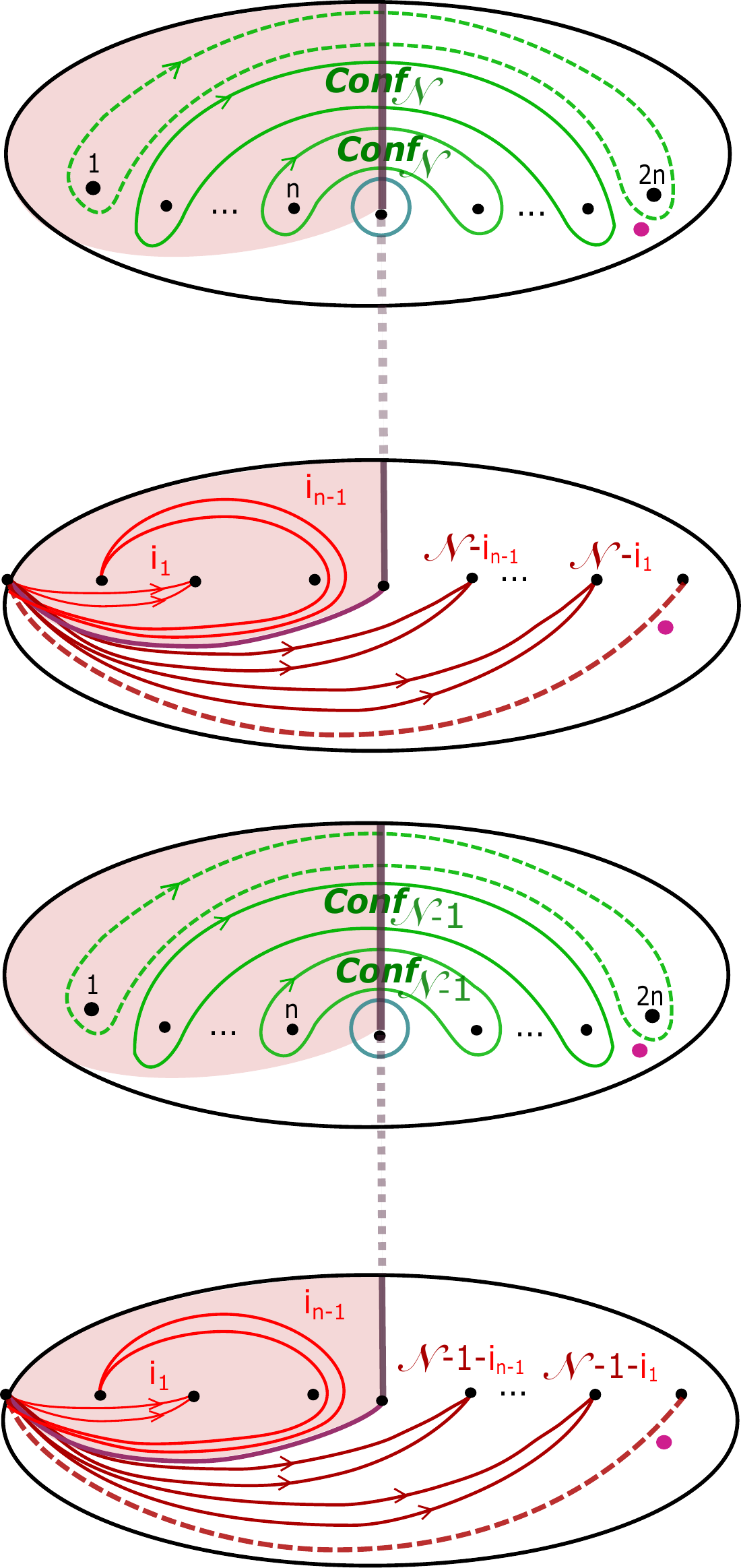}
\vspace{-1mm}
\caption{\normalsize Universal link invariants as limits of level $\cN$ link invariants}
\label{ColouredAlexx}
\end{figure}

\subsection{Geometrical set-up} Let us consider a link $L$ and $\beta_n \in B_n$ a braid such that $L= \widehat{\beta_n}$. We use $\mathbb D_{2n+2}$, the $(2n+2)$-punctured disc and we consider a splitting of the the set of punctures as follows:
\begin{itemize}
\item[•]fix $2n$ punctures on a horizontal line, which we call $p$-punctures (and denote them $\{1,..,2n\}$)
\item[•]fix $1$ puncture which we call $q$-puncture (and label by $\{0\}$)
\item[•]consider also $1$ puncture, called $s$-puncture, as in Figure \ref{ColouredAlex0} and Figure \ref{Localsystt}.
\end{itemize}
Then for $m\in \N$, we denote by $C_{n,m}:=\Conf_{m}(\mathbb D_{2n+2})$ the unordered configuration space of $m$ points in the disc. In \cite{CrI} we constructed a suitable covering space for our geometric context, as below.
We fix $\bar{\cN}=(\cN_1,\cN_2,...,\cN_n)$ which we call ``multi-level'', such that $\cN_1+\cN_2+...+\cN_n=m-1$. 

\begin{defn}[{\bf Covering space at {level $\bar{\cN}$}}]
We construct a well-chosen covering that depends on the choice of the level $\bar{\cN}$ that in turn has the advantage of allowing us to have {\em well-defined lifts} of submanifolds with support on ovals in the disc. Specifically, we define a local system $\bar{\Phi}^{\bar{\cN}}$ on $\Conf_{m}(\mathbb D_{2n+2})$ (Notation \ref{twsyst}), associated to {\em  $\bar{\cN}$}. Our tools are the homologies of the associated covering space, denoted by  $H^{\bar{\cN}}_{n,m}$ and $H^{\bar{\cN},\partial}_{n,m}$ (they are versions of Borel-Moore homology of this covering, see Diagram \ref{diagsl}). 

\end{defn}
\begin{prop}[{Intersection pairing}] There exists a Poincaré-type duality between these homologies: 
$\left\langle ~,~ \right\rangle:  H^{\bar{\cN}}_{n,m} \otimes H^{\bar{\cN},\partial}_{n,m} \rightarrow \C[w^1_1,...,w^{l}_{\cN-1},u_1^{\pm 1},...,u_l^{\pm 1},x_1^{\pm 1},...,x_l^{\pm 1},y{\pm 1}, d^{\pm 1}] $ and additionally a braid group action on $H^{\bar{\cN}}_{n,m}$ (see Proposition \ref{intersection}).
\end{prop}
From now on we will use the following convention for the punctured disc.
\begin{defn}(Splitting of the punctured disc)  We consider the two halves of the punctured disc, defined as follows: 
\begin{itemize}
\item (Left hand side of the disc) This is given by half of the disc from Figure \ref{ColouredAlex0} that passes though the puncture labeled by $0$ and contains the first $n$ $p$-punctures.
\item (Right hand side of the disc) This will be the defined by the complement of the above, the half of the disc that contains the rest of the $p$-punctures.
\end{itemize}

\end{defn}

\subsection{$\cN^{th}$-weighted Lagrangian intersection}\label{weightedintersection}

Let us fix a level $\cN$ and $L$ an oriented framed link with $l$ components and framings $f_1,...,f_l\in \Z$. Let $\beta_n \in B_n$ such that $L= \widehat{\beta_n}$. For the level $\cN$ weighted intersection form, we make use of the above homological set-up for the following parameters:
\begin{itemize}
\item Number of particles $m(\cN):=2+(n-1)(\cN-1)$
\item Space $C_{n,m_(\cN)}:=\Conf_{2+(n-1)(\cN-1)}\left(\mathbb D_{2n+2}\right)$
\item Multi-level: $\bar{\cN}:=(\cN_1,...,\cN_n)=(1,\cN-1,...,\cN-1)$
\item Local system: $\bar{\Phi}^{\bar{\cN}}$
\item Specialisation of coefficients: $\sJt$ and $\sAt$ (Notation \ref{pA1}, Definition \ref{pJ1}).
\item Indexing set: $\{\bar{0},\dots,\overline{\cN-1}\}= \big\{ \bar{i}=(i_1,...,i_{n-1})\in \N^{n-1} \mid 0\leq i_k \leq \cN-1, \  \forall 1 \leq k \leq n-1 \big\}.$
\end{itemize}
\clearpage
\begin{defn}(Coloured Homology classes) We consider homology classes of lifts of Lagrangian submanifolds from the base space. These submanifolds are prescribed via collections of arcs and ovals in the disc, following Notation \ref{paths}.
Then, for $\bar{i}\in \{\bar{0},\dots,\overline{\cN-1}\}$ we consider two classes:
\vspace{-5mm}
\begin{figure}[H]
\centering
$$\hspace{-28mm} \text{ Classes: \ \ \ \ \ } {\color{red} \mathscr F_{\bar{i},\cN} \in \HAi} \ \ \ \ \ \ \ \ \text{ and }\ \ \ \ \ \ \ \ \  {\color{dgreen} \mathscr L_{\bar{i},\cN}\in \HAdi}$$
\vspace{-3mm}

\includegraphics[scale=0.33]{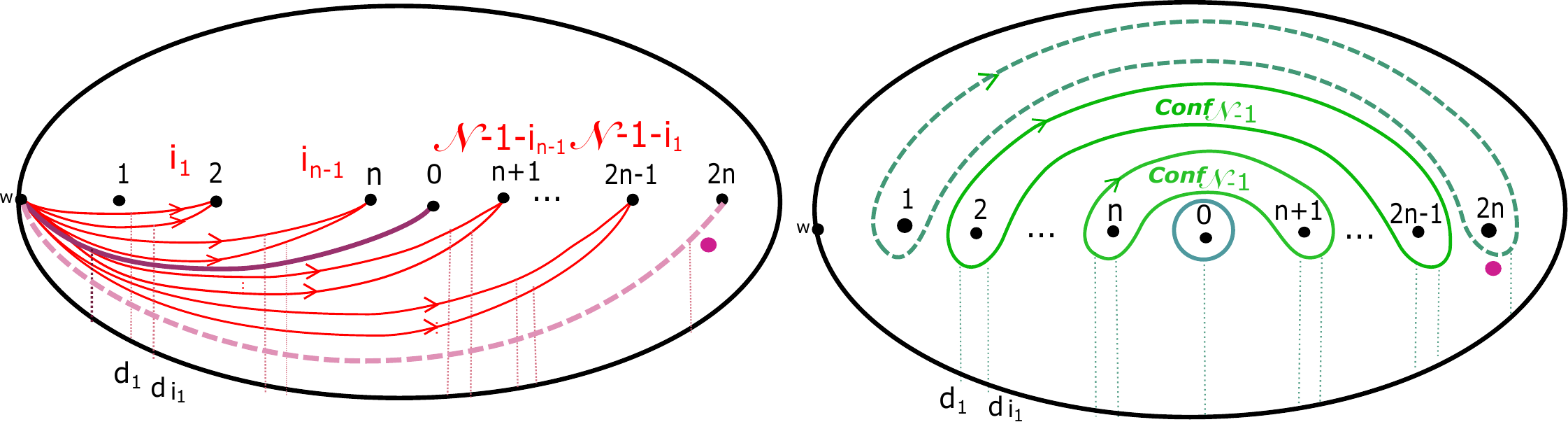}
$$\hspace{-60mm} \text{ Weight: } w^{C(2)}_{i_1}\cdot ... \cdot w^{C(n)}_{i_{n-1}}\hspace{30mm}$$
\vspace{-4mm}
\caption{\normalsize Weighted Lagrangian intersection at level $\cN$}
\label{ColouredAlex0}
\end{figure}
\end{defn}
\vspace{-4mm}
Since the link is the closure of our braid with $n$ strands, we have an induced colouring $C$ of $2n$ points with $l$ colours: $
C:\{1,...,2n\}\rightarrow \{1,...,l\}.$ (as in Definition \ref{br}).
\begin{defn}[{\bf \em Weighted Lagrangian intersection}] \label{D:w} We define the following Lagrangian intersection in ${\Conf_{2+(n-1)(\cN-1)}\left(\mathbb D_{2n+2}\right)}$:
\begin{equation} 
\begin{aligned}
 \PA \in \LL =~&\Z[w^1_1,...,w^{l}_{\cN-1},u_1^{\pm 1},...,u_l^{\pm 1},x_1^{\pm 1},...,x_l^{\pm 1},y^{\pm 1}, d^{\pm 1}]\\
 \PA :=&\prod_{i=1}^l u_{i}^{ \left(f_i-\sum_{j \neq {i}} lk_{i,j} \right)} \cdot \prod_{i=2}^{n} u^{-1}_{C(i)} \cdot \\
 & \cdot \sum_{\bar{i}= \bar{0}}^{\overline{\cN -1}} w^{C(2)}_{i_1}\cdot ... \cdot w^{C(n)}_{i_{n-1}}
 \left\langle(\beta_{n} \cup {\mathbb I}_{n+2} ) \ { \color{red} \FA}, {\color{dgreen} \LA}\right\rangle.
\end{aligned}
\end{equation} 
\end{defn}
The precise definition of this form is presented in Subsection \ref{S:we}.

\subsection{Notations} \label{S:not}
\begin{notation}[Modified dimensions] Let us consider the following quantum numbers:
$$ \{ x \}_q :=q^x-q^{-x} \ \ \ \ [x]_{q}:= \frac{q^x-q^{-x}}{q-q^{-1}}.$$
Also, for $N\in \N$ let $\xi_N=e^{\frac{2\pi i}{2N}}$ and for $\lambda \in \C$ let
$ d[\lambda]_{\xi_N}:= \frac{\{\lambda\}_{\xi_N}}{\{N\lambda\}_{\xi_N}}$,  which is called the modified dimension.
\end{notation}

\begin{defn}$\bullet$ (Multi-indices at bounded multi-levels)

We denote the  following sets of multi-indices:
\begin{equation}\label{mi}
\begin{aligned}
&\{\bar{0},\dots,\overline{\cN-1}\}:= \big\{ \bar{i}=(i_1,...,i_{n-1})\in \N^{n-1} \mid 0\leq i_k \leq \cN-1, \  \forall k\in \{1,...,n-1\} \big\}\\
&\{\bar{\cM},\dots,\overline{\cN-1}\}:= \big\{ \bar{i}=(i_1,...,i_{n-1})\in \N^{n-1} \mid 0\leq i_k \leq \cN-1, \  \forall k\in \{1,...,n-1\} \text{ and }\\
& \hspace{90mm} \exists \ j\in \{1,...,n-1\}, \  \cM\leq i_j \big\}.
\end{aligned}
\end{equation}
$\bullet$ (Bounded colourings) For a fixed  set of colours $N_1,...,N_l \leq \cN$, we denote the multi-index $\bar{N}=(N_1,...,N_l)$ and say that $\bar{N} \leq \cN$ if $N_1,...,N_l \leq \cN$.
\end{defn}
\begin{defn}[Multi-indices]
Let us fix three type of parameters: $l \in \N$, $l\geq1$ and $\alpha_1,..,\alpha_l$, $N_1,...,N_l \in \N$. We consider the multi-indices:
\begin{equation}
\begin{cases}
\bar{\alpha}:=(\alpha_1,..,\alpha_l)\\
\bar{N}:=(N_1,...,N_l).
\end{cases}
\end{equation} 
\end{defn}

\subsection{Specialisations for the homology groups-globalised specialisation}

\begin{defn}[Globalised specialisations] Let $\cN\in \mathbb N$ be fixed and let $t\in \mathbb Z$, and also three multi-indices: $\bar{\alpha}\in \C^{l}$, $\bar{\eta}\in \C[q^{\pm\frac{1}{2}}]$ and $\bar{b}=(b^k_j)\in \N^{l({\cN-1})}, k\in \{1,...,l\}, j\in \{1,...,\cN-1\}$. We denote the specialisation of coefficients:
$$ \psi^{t,\cN,\bar{b}}_{q,\bar{\alpha},\bar{\eta}}: \Z[w^1_1,...,w^{l}_{\cN-1},u_1^{\pm 1},...,u_{l}^{\pm 1},x_1^{\pm 1},...,x_{l}^{\pm 1},y^{\pm 1}, d^{\pm 1}] \rightarrow \C[q^{\pm \frac{1}{2}},q^{\pm \frac{\alpha_1}{2}},...,q^{\pm \frac{\alpha_l}{2}}]$$
\begin{equation}\label{not}
\begin{cases}
&\psi^{t,\cN,\bar{b}}_{q,\bar{\alpha},\bar{\eta}}(u_j)=q^{t\alpha_j}\\
&\psi^{t,\cN,\bar{b}}_{q,\bar{\alpha},\bar{\eta}}(x_j)=q^{\alpha_j}, \ j\in \{1,...,l\}\\
&\psi^{t,\cN,\bar{b}}_{q,\bar{\alpha},\bar{\eta}}(y)=\bar{\eta}\\
&\psi^{t,\cN,\bar{b}}_{q,\bar{\alpha},\bar{\eta}}(d)=q^{-1}\\
&\psi^{t,\cN,\bar{b}}_{q,\bar{\alpha},\bar{\eta}}(w^{k}_j)=1, \ \text{ if } j\leq b^k_j, \\
&\psi^{t,\cN,\bar{b}}_{q,\bar{\alpha},\bar{\eta}}(w^k_j)=0, \ \text{ if } j\geq b^k_j+1, k\in \{1,...,l\}, j\in \{1,...,\cN-1\}.
\end{cases}
\end{equation} 
\end{defn}
\begin{defn}(Colouring maps)\label{spec}\\
For a given colouring $C:\{1,...,n\}\rightarrow \{1,...,l\}$ we denote the specialisation of coefficients as below:
\begin{center}
\begin{tikzpicture}
[x=1.2mm,y=1.4mm]
\node (b1)               at (-37,0)    {$\C[w^1_1,...,w^{l}_{\cN-1},u_1^{\pm 1},...,u_{l}^{\pm 1},x_1^{\pm 1},...,x_n^{\pm 1},y^{\pm1}, d^{\pm 1}]$};
\node (b2)   at (27,-20)   {$\C[w^1_1,...,w^{l}_{\cN-1},u_1^{\pm 1},...,u_{l}^{\pm 1},x_1^{\pm 1},...,x_l^{\pm 1},y^{\pm1}, d^{\pm 1}]$}; 
\node (b3)   at (-37,-50)   {$\C[q^{\pm \frac{1}{2}},q^{\pm \frac{\alpha_1}{2}},...,q^{\pm \frac{\alpha_l}{2}}]$};
,\draw[->]   (b1)      to node[xshift=1mm,yshift=5mm,font=\large]{$f_C$ \eqref{eq:8}}                           (b2);
\draw[->]             (b2)      to node[right,xshift=2mm,font=\large]{$\psi^{t,\cN,\bar{b}}_{q,\bar{\alpha},\bar{\eta}}$\ \eqref{not}}   (b3);
\draw[->,thick,dotted]             (b1)      to node[left,font=\large]{}   (b3);
\end{tikzpicture}
\end{center}
\end{defn}
\begin{notation}\label{not'}
Also, we denote by $N^C_i:=N_{C(i)}$.
\end{notation}

\subsection{Coloured Jones polynomials - generic parameters}
For the case of semi-simple link invariants, we will use the globalised specialisation for the following parameters:
\begin{equation}
\begin{cases}
\bar{N}^{\alpha}:=(N_1-1,...,N_l-1)\\
\overline{[N]}:=([N^C_1]_{q})\\
t=1\\
\bar{N}^b=(n^{k}_{j}), \ n^{k}_{j}=N_{k}-1 .
\end{cases}
\end{equation} 
\begin{notation}[Specialisation for generic $q$]\label{pJ1}
We denote the specialisation:
\begin{equation}
\sJt:=\sJtt.
\end{equation}
\end{notation}

\begin{defn}[Specialisation for coloured Jones polynomials]\label{pJ2}
The associated specialisation of coefficients is given by:
$$ \sJt: \Z[w^1_1,...,w^{l}_{\cN-1},u_1^{\pm 1},...,u_{l}^{\pm 1},x_1^{\pm 1},...,x_{l}^{\pm 1},y^{\pm 1}, d^{\pm 1}]  \rightarrow \Z[d^{\pm 1}]$$
\begin{equation}\label{eq:8''''} 
\begin{cases}
&\sJt(u_i)=\left(\sJt(x_i)\right)^t=d^{1-N_i}\\
&\sJt(x_i)=d^{1-N_i}, \ i\in \{1,...,l\}\\
&\sJt(y)=[N^C_1]_{d^{-1}},\\
&\sJt(w^k_j)=1, \ \text{ if } j\leq N_k-1 ,\\
&\sJt(w^k_j)=0, \ \text{ if } j\geq N_{k},  k\in \{1,...,n-1\}, j\in \{1,...,\cN-1\}.
\end{cases}
\end{equation}
\end{defn}
\subsection{Coloured Alexander polynomials - parameters at roots of unity}

For this case of non-semi simple link invariants, let us fix $\cN$ to be the level associated to the weighted intersection form and let us consider $\cM$ to be the order of the root of unity such that $\cM \leq \cN$. We use the globalised specialisation for the multi-indices:
\begin{equation}
\begin{cases}
\bar{\lambda}:=(\lambda_1,...,\lambda_l) \in \C^l\\
\overline{d(\lambda)}:=(d[\lambda_{C(1)}]_{\xi_{\cM}})\\
t=1-\cM\\
\bar{\cM}^b=(\cM-1,...,\cM-1).
\end{cases}
\end{equation} 

\begin{notation}[Specialisation at roots of unity]\label{pA1}
Let us denote the specialisation associated to the above parameters as:
\begin{equation}
\sAt:=\sAtt.
\end{equation}
\end{notation}
\begin{rmk}[ADO invariant in variables $x_1,...,x_l$]
In the above notations, $\bar{\lambda}$ should be seen as the colours of the ADO link invariant, if these colours are chosen to be generic complex numbers. However, in the literature, it is often used the fact that we can look at the ADO invariant as a polynomial in the variables $\xi_{\cN}^{\lambda_1}$,...,$\xi_{\cN}^{\lambda_l}$, which we denote by $x_1,...,x_l$.

 Via this dictionary, we obtain that the $\cN^{th}$ ADO invariant is a polynomial belonging to the ring $\Q(\xn)(x^{2\cN}_{1}-1,...,x^{2\cN}_{l}-1)^{-1}[x_1^{\pm 1},...,x_{l}^{\pm 1}]$.
\end{rmk}
 In this setting, let us consider the following specialisation of coefficients.
\begin{defn}[Specialisation for coloured Alexander polynomials]\label{pA2}
The associated specialisation of coefficients is given by:
$$ \sAt: \Z[w^1_1,...,w^{l}_{\cN-1},u_1^{\pm 1},...,u_{l}^{\pm 1},x_1^{\pm 1},...,x_{l}^{\pm 1},y^{\pm 1} d^{\pm 1}] \rightarrow \Q(\xi_{\cM})(x^{2\cM}_{1}-1,...,x^{2\cM}_{l}-1)^{-1}[x_1^{\pm 1},...,x_{l}^{\pm 1}]$$
\begin{equation}\label{eq:8''''} 
\begin{cases}
&\sAt(u_j)=x_j^{(1-\cM)}\\
&\sAt(y)=(d[\lambda_{C(1)}]_{\xi_{\cM}}),\\
&\sAt(d)= \xi_{\cM}^{-1}\\
&\sAt(w^k_j)=1, \ \text{ if } j\leq \cM-1,\\
&\sAt(w^k_j)=0, \ \text{ if } j\geq \cM,  k\in \{1,...,l\}, j\in \{1,...,\cN-1\}.
\end{cases}
\end{equation}
\end{defn}
\subsection{Specialisations of coefficients for universal invariants}\label{sumspec}

\begin{defn}[Rings for the universal invariant] Let us denote the following rings:
 
\begin{equation}
 \begin{cases}
 \LL:=\Z[w^1_1,...,w^{l}_{\cN-1},u_1^{\pm 1},...,u_{l}^{\pm 1},x_1^{\pm 1},...,x_{l}^{\pm 1}, y^{\pm 1},d^{\pm 1}]
 =\Z[\bar{w},(\bar{u})^{\pm 1},(\bar{x})^{\pm 1},y^{\pm 1},d^{\pm 1}]\\
 \LN:=\Q(\xn)(x^{2\cN}_{1}-1,...,x^{2\cN}_{l}-1)^{-1}[x_1^{\pm 1},...,x_{l}^{\pm 1}]\\
 \LNJ=\Z[d^{\pm 1}] .
 \end{cases}
 \end{equation}
where we compressed the multi-indices as $$\bar{w}:=(w^1_1,...,w^{l}_{\cN-1}), (\bar{u})^{\pm 1}=(u_1^{\pm 1},...,u_{l}^{\pm 1}), (\bar{x})^{\pm 1}=(x_1^{\pm 1},...,x_{l}^{\pm 1}).$$
\end{defn}
\begin{defn}[Level $\cN$ specialisations] \label{pA3}
We recall the specialisations associated for the generic case and for the case of roots of unity, which we denoted as:
$$ \snJ: \LL=\Z[\bar{w},(\bar{u})^{\pm 1},(\bar{x})^{\pm 1},y^{\pm 1},d^{\pm 1}] \rightarrow \LNJ$$
\begin{equation}\label{u1''} 
\begin{cases}
&\sJt(u_i)=\left(\sJt(x_i)\right)^t=d^{1-N_i}\\
&\sJt(x_i)=d^{1-N_i}, \ i\in \{1,...,l\}\\
&\sJt(y)=[N^C_1]_{d^{-1}},\\
&\sJt(w^k_j)=1, \ \text{ if } j\leq N_k-1 ,\\
&\sJt(w^k_j)=0, \ \text{ if } j\geq N_{k},  k\in \{1,...,l\}, j\in \{1,...,\cN-1\}.
\end{cases}
\end{equation}
$$ \sAt:\LL \rightarrow \LM$$
\begin{equation}\label{u1'}
\begin{cases}
&\sAt(u_j)=x_j^{(1-\cM)}\\
&\sAt(y)=(d[\lambda_{C(1)}]_{\xi_{\cM}}),\\
&\sAt(d)= \xi_{\cM}^{-1}\\
&\sAt(w^k_j)=1, \ \text{ if } j\leq \cM-1,\\
&\sAt(w^k_j)=0, \ \text{ if } j\geq \cM,  k\in \{1,...,l\}, j\in \{1,...,\cN-1\}.
\end{cases}
\end{equation}

\end{defn}

\subsection{Formulas for the universal rings}\label{formex}
In section \ref{S:ringJones} we compute the precise form of the quotient rings and universal rings where the universal Jones and Alexander invariants belong to. As a summary, we present their formulas below.

\begin{prop}[\bf {Ring for the $\cN^{th}$ unified Jones invariant}] The $\cN^{th}$ unified Jones invariant $\PAAJ$ belongs to the quotient ring $\LLNJ$ that has the following form:
\begin{equation} \label{rnj}
\begin{aligned}
\LLNJ= \LL / & \langle u_i-x_i,\bigcap_{\bar{N} \leq \cN} \left(y\left(d-d^{-1}\right)-(d^{N^C_1}-d^{-N^C_1})\right.,\\
& \ \ \ \ \ \ \ \ \ \ \ \ \ \ \  \ \  \left. x_i-d^{1-N_i}, w^k_j-1 \text{ if } j\leq N_k-1 \right.,\\
& \ \ \ \ \ \ \ \ \ \ \ \ \ \ \ \ \ \left. w^k_j \ \text{ if } j\geq N_{k},   i,k\in \{1,...,l\}, j\in \{1,...,\cN-1\} \right) \rangle 
\end{aligned}
\end{equation}
where $\LL=\Z[w^1_1,...,w^{l}_{\cN-1},u_1^{\pm 1},...,u_l^{\pm 1},x_1^{\pm 1},...,x_l^{\pm 1},y^{\pm 1}, d^{\pm 1}]$, as in Proposition \ref{f10}.
\end{prop}

\begin{prop}[\bf {Universal ring for universal Jones invariant}] \label{uj} The universal Jones invariant $\IJJ$ belongs to the universal ring $\LLhJ$, which is the projective limit of this sequence of rings as in equation \eqref{eq10:limit}:
\begin{equation}\label{rj}
\LLhJ:= \underset{\longleftarrow}{\mathrm{lim}} \ \LLNJ.
\end{equation}
Also, we have a well-defined induced universal specialisation map, which we denote:
\begin{equation}
\usnJ: \LLhJ \rightarrow \LNJ.
  \end{equation}
\end{prop}

\begin{prop}[\bf {Ring for the $\cN^{th}$ unified ADO invariant}] The $\cN^{th}$ unified ADO invariant $\PAA$ belongs to the quotient ring $\LLN$ that has the following description:

\begin{equation}\label{rna}
\begin{aligned}
\LL=&\Z[w^1_1,...,w^{l}_{\cN-1},u_1^{\pm 1},...,u_l^{\pm 1},x_1^{\pm 1},...,x_l^{\pm 1},y^{\pm 1}, d^{\pm 1}]\\
\LLN= & \LL /  \bigcap_{\cM \leq \cN} \langle  (u_i-x_i^{1-\cM}), y\left(x^{\cM}_{C(1)}-x^{-\cM}_{C(1)}\right)-(x_{C(1)}-x^{-1}_{C(1)}), \ \varphi_{2\cM}(d)\\
& \left. \ \ \ \ \ \ \ \ \ \ \ w^k_j-1, \ \text{ if } j\leq \cM-1\right.,\\
& \ \ \ \ \ \ \ \ \ \  \  w^k_j \ \text{ if } j\geq \cM,  k,i\in \{1,...,l\}, j\in \{1,...,\cN-1\}  \rangle,
\end{aligned}
\end{equation}
where $\LL=\Z[w^1_1,...,w^{l}_{\cN-1},u_1^{\pm 1},...,u_l^{\pm 1},x_1^{\pm 1},...,x_l^{\pm 1},y^{\pm 1}, d^{\pm 1}]$.
\end{prop}

\begin{prop}[\bf {Universal ring for universal ADO invariant}] \label{ua}
The universal ADO invariant $\I$ belongs to the universal ring $\LL$, which is the projective limit of the above sequence of rings as in equation \eqref{eq10:limit}:
\begin{equation}\label{ra}
\LLh:= \underset{\longleftarrow}{\mathrm{lim}} \ \LLN.
\end{equation}
Then, we have also a well-defined induced universal specialisation map, which we denote:
\begin{equation}
\usn: \LLh \rightarrow \LN.
  \end{equation}
\end{prop}
\subsection{Summary: Diagram with all the specialisations of coefficients for link invariants}
We will use the specialisations of coefficients and homology groups, as in Figure \ref{diagsl}.
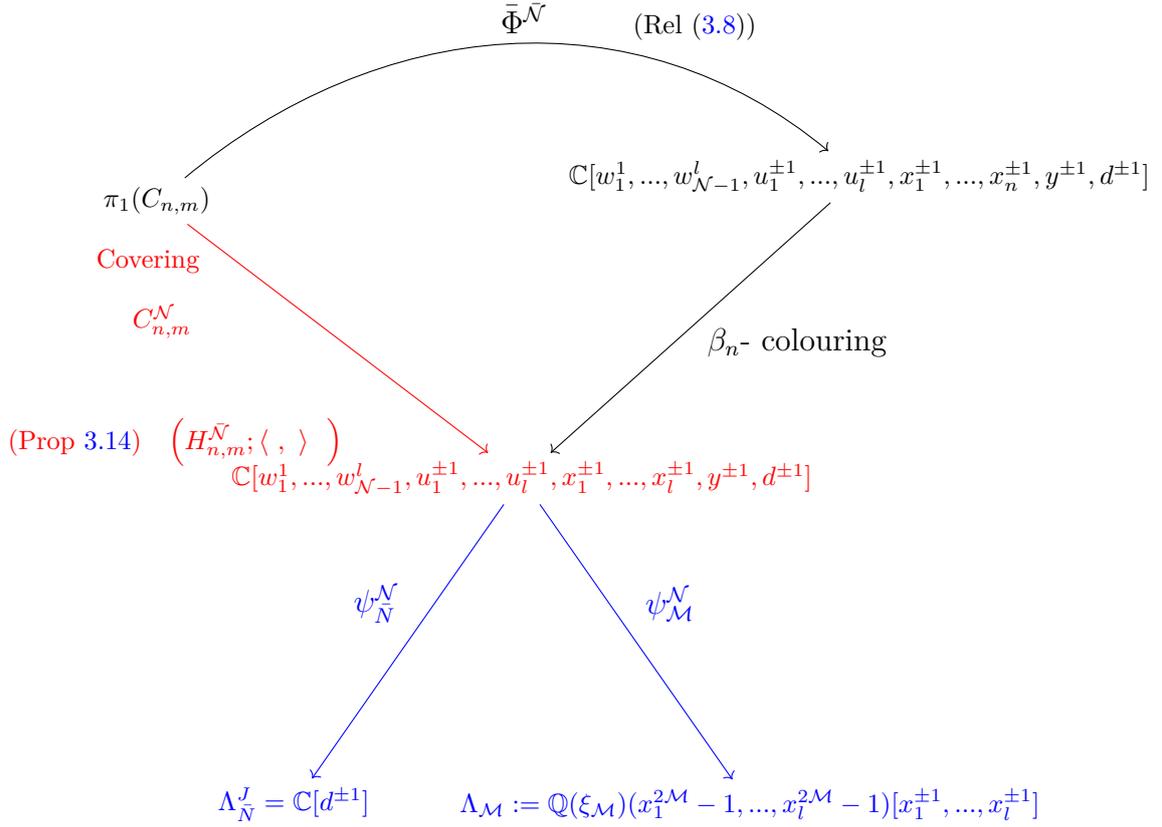
\begin{figure}[H]
\begin{center}
\begin{tikzpicture}
[x=1.2mm,y=1.6mm]


\node (b1)               at (5,50)    {$\pi_1(C_{n,m})$};
\node (b20) [color=red] at (4,45)   {Covering};
\node (b200) [color=red] at (8,40)   {$C_{n,m}^{\cN} \ \ \ \ \ $};
\node (b3) [color=red]   at (45,27)   {$\C[w^1_1,...,w^{l}_{\cN-1},u_1^{\pm 1},...,u_l^{\pm 1},x_1^{\pm 1},...,x_l^{\pm 1},y^{\pm 1}, d^{\pm 1}]$};
\node (b33) [color=red] at (7,30)   {(Prop \ref{D:4}) \ \ $\left( H^{\bar{\cN}}_{n,m}; \left\langle ~,~ \right\rangle \ \ \right)$};
\node (b33') [color=red]   at (96,-3)   {};
\node (b4) [color=blue]  at (20,0)   {$\LNJ=\C[d^{\pm 1}]$};
\node (b44) [color=blue]  at (70,0)   {$ \LM:=\Q(\xi_{\cM})(x^{2\cM}_{1}-1,...,x^{2\cM}_{l}-1)^{-1}[x_1^{\pm 1},...,x_{l}^{\pm 1}]
$};
\node (b1'') [color=black]  at (82,52)   {$\C[w^1_1,...,w^{l}_{\cN-1},u_1^{\pm 1},...,u_l^{\pm 1},x_1^{\pm 1},...,x_n^{\pm 1},y^{\pm 1}, d^{\pm 1}]$};

\draw[->,color=red]   (b1)      to node[yshift=-3mm,xshift=11mm,font=\large]{}                           (b3);
\draw[->,color=blue,yshift=5mm]             (b3)      to node[left,font=\large,yshift=5mm]{$\sJt$}   (b4);
\draw[->,color=blue,yshift=5mm]             (b3)      to node[right,font=\large,yshift=5mm]{$\sAt$}   (b44);

\draw[->,color=black]   (b1)      to [out=40,in=140] node[right,xshift=-2mm,yshift=3mm,font=\large]{$\bar{\Phi}^{\bar{\cN}} \hspace{10mm}$ {\normalsize{(Rel \eqref{sor})}}}                        (b1'');
\draw[->,color=black]   (b1'')      to node[right,yshift=-2mm,xshift=1mm,font=\large]{$\beta_n$- \text{colouring}}                        (b3);
\end{tikzpicture}
\end{center}
\vspace{-3mm}
\caption{\normalsize Specialisations of coefficients: Weighted Lagrangian intersection}\label{diagsl}
\end{figure}

\section{Homological set-up} \label{S3}
Let us consider $n \in \N$. 
We denote by $\mathbb D_{2n+2}$ the $(2n+2)$-punctured disc, and fix a splitting of its punctures as below:

\begin{itemize}
\setlength\itemsep{-0.2em}
\item[•]$2n$ horizontal punctures, called $p$-punctures (and we label them $\{1,..,2n\}$)
\item[•]$1$ puncture called $q$-puncture (which we denote $\{0\}$)
\item[•]$1$ punctures placed as in Figure \ref{Localsystt}, which we call the $s$-puncture.
\end{itemize}
\subsection{Configuration space of the punctured disc}

The first part of our homological construction involves the homology of coverings of configuration spaces and it is precisely the one from \cite[Section 4]{CrI}. We refer to this article for all the details, and in the following part we present a summary of the construction.

For $m\in \N$ we consider the unordered configuration space of $m$ points in the punctured disc $\mathbb D_{2n+2}$:
 $$\Clm:=\Conf_{m}(\mathbb D_{2n+2}).$$ 
We also consider a fixed  base point of this configuration space, defined by a set of points on the boundary of the disc $d_1,..d_m \in \partial \hspace{0.5mm} \mathbb D_{2n+2}$. Let ${\bf d}=(d_1,...,d_m)$ be the associated point in the configuration space. Now we define a local system on the space $\Clm$. 

We suppose that $m \geq 2$. We use the abelianisation to the first homology group of the configuration space, which has the following form.
\begin{prop}[Abelianisation map]
 Let us denote the abelianisation map $[ \ ]: \pi_1(\Clm) \rightarrow H_1\left( \Clm\right)$ for the fundamental group of our configuration space. Its homology has the following structure:
\begin{equation*}
\begin{aligned}
H_1\left( \Clm \right)  \simeq \ \ \ \  \ & \Z^{n+1} \ \ \ \ \oplus \ \ \ \  \Z^{n} \ \ \ \ \ \oplus \ \ \ \ \Z \ \ \ \  \oplus \ \ \ \ \Z\\
&\langle [\sigma_i] \rangle \ \ \ \ \ \ \ \ \ \langle [\bar{\sigma}_{i'}] \rangle \ \ \ \ \ \ \ \ \ \langle [\gamma] \rangle \ \ \ \ \ \ \ \ \langle [\delta]\rangle,  \ \ \ \ {i\in \{0,...,n\}}\\
& \hspace{69mm} {i'\in \{1,...,n\}}.
\end{aligned}
\end{equation*}
The five types of generators are presented in Figure \ref{Localsystt}. 
\begin{figure}[H]
\centering
\includegraphics[scale=0.26]{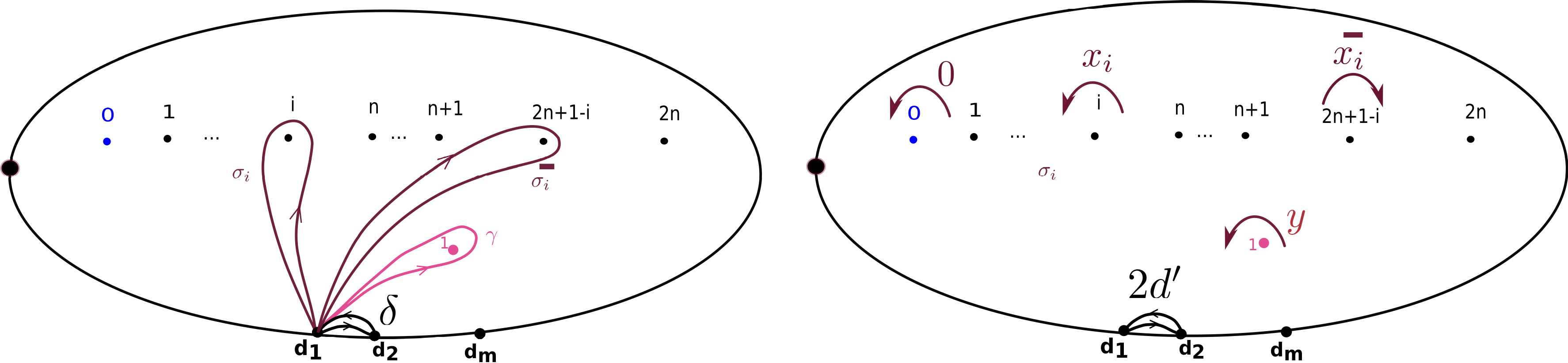}
\caption{\normalsize Local system}
\label{Localsystt}
\end{figure}
\end{prop}
\subsection{Local system and covering space at level $\bar{\cN}$}
Our local system will depend on a choice of a sequence of ``levels''. This will be used in an essential manner in order to make sure that our submanifolds, which have geometric supports encoded by configuration spaces on ovals in the disc, lift to the associated covering space.

\begin{defn}[Multi-level]
We start with a fixed sequence of levels $\cN_1,...,\cN_n \in \N$ and call a ``multi-level'' the following collection:
\begin{equation}
\bar{\cN}:=(\cN_1,...,\cN_n).
\end{equation}
\end{defn}

\begin{defn}[Augmentation map] After this first step, we consider the following augmentation map $$\nu: H_1\left(\Clm\right)\rightarrow \Z^n \oplus \Z \oplus \Z$$ 
$$ \hspace{28mm} \langle x_i \rangle \ \ \langle y\rangle \ \ \langle d' \rangle$$ defined by the formulas:
\begin{equation}
\begin{cases}
&\nu(\sigma_0)=0\\
&\nu(\sigma_i)=2x_i,\\ 
&\nu(\bar{\sigma}_i)=2(x_i+(\cN_i-1)d'), i\in \{1,...,n\}\\
&\nu(\gamma)=y\\
&\nu(\delta)=2d'.
\end{cases}
\end{equation}
\end{defn}
\begin{defn}(Local system)\label{localsystem0}
Let us consider the local system that is given by the following composition of the above maps:
\begin{equation}
\begin{aligned}
&\Phi^{\bar{\cN}}: \pi_1(\Clm) \rightarrow \Z^n \oplus \Z \oplus \Z\\
&\hspace{24mm} \langle x_i \rangle \ \ \langle y\rangle \ \ \langle d' \rangle, \ i \in \{1,...,n\},\\
&\Phi^{\bar{\cN}}= \nu \circ [ \ ]. \ \ \ \ \ \ \ \ \ \ \ \ \ \ \ \ \ \ \ \ 
\end{aligned}
\end{equation}
\end{defn}

\begin{defn}[Level $\bar{\cN}$ covering space]\label{localsystemc}
We consider the covering of the configuration space $\Clm$ which is associated to the level $\bar{\cN}$ local system $\Phi^{\bar{\cN}}$, and denote it by $\ClmN$.
\end{defn}
\begin{notation}[Base point]

\

For the next steps we also fix a base point $\tilde{{\bf d}}$ which belongs to the fiber over ${\bf d}$ in $\ClmN$.
\end{notation}

\subsection{Group rings} Our tools involve the homologies of this level $\bar{\cN}$ covering space. Then, the next step will be to use a Poincaré-Lefschetz duality between these two homology groups.  

The group of deck transformations of $\ClmN$ is given by:
$$\Imm\left(\Phi^{\bar{\cN}}\right)=(2\Z)^{n}\oplus \Z\oplus (2\Z) \subseteq \Z^{n}\oplus \Z\oplus \Z.$$
This means that the homology of this covering space $\ClmN$ is a module over the associated group ring:
$$\Z[x_1^{\pm 2},...,x_n^{\pm 2},y^{\pm1}, d'^{\pm 2}].$$ 
\begin{defn}[Inclusion of group rings] \label{iota}Let us denote the following inclusion map:
$$\iota: \C[x_1^{\pm 2},...,x_n^{\pm 2},y^{\pm1}, d'^{\pm 2}] \subseteq \C[w^1_1,…,w^{n-1}_{\cN-1},u_1^{\pm 1},...,u_{l}^{\pm 1},x_1^{\pm 1},...,x_{n}^{\pm 1}, y^{\pm 1},d^{\pm 1}]
.$$
\end{defn}
where we denote a new variable by $d:=\xi_1d'$ 
where $\xi_1=i=e^{\frac{2\pi i}{2}}$.

We consider the homology of the level $\bar{\cN}$ covering, which we tensor over $\iota$ with the following group ring: $$\C[w^1_1,…,w^{l}_{\cN-1},u_1^{\pm 1},...,u_{l}^{\pm 1},x_1^{\pm 1},...,x_{n}^{\pm 1}, y^{\pm 1},d^{\pm 1}]
.$$
\begin{rmk}[Structure of the homology of the level $\cN$ covering space]\label{homcov}
Using this change of coefficients, we have homology groups which become modules over: 
\begin{equation}
\C[w^1_1,…,w^{l}_{\cN-1},u_1^{\pm 1},...,u_{l}^{\pm 1},x_1^{\pm 1},...,x_{n}^{\pm 1}, y^{\pm 1},d^{\pm 1}]
.
\end{equation}
\end{rmk}

\subsection{Level $\cN$ homology groups}\label{hlgy}
For the next part of our set-up,  we consider the relative middle dimensional homology of the level $\bar{\cN}$ covering space. 

More specifically, we are going to use two homology groups which are relative to a certain splitting of the boundary of the configuration space.

\begin{notation}
a) We denote by $S^{-}\subseteq \partial \mathbb D_{2n+2}$ be the semicircle on the boundary of the disc given by points with negative $x$-coordinate. Let us fix also a point on the boundary of the disc, which we  denote:
 $$w \in S^{-} \subseteq \partial \mathbb D_{2n+2}.$$
b) Let $C^{-}$ be the subspace of the boundary of the configuration space $\Clm$  given by configurations where at least one particle belongs to  $S^{-}$. 

c) Then, we denote by  $P^{-} \subseteq \partial \ClmN$ part of the boundary of level $\bar{\cN}$-covering which is the fiber over $C^{-}$.

\end{notation}

\begin{defn}
\label{T2}
The precise splitting of the infinity part of the configuration space that we are going to use is constructed in \cite[Remark 7.5]{CrM}. 
Using this splitting, we have two homology groups.
\begin{equation*}
\begin{aligned}
& \bullet  H^{\text{lf},\infty,-}_m(\ClmN,P^{-}; \Z) \text{ the homology relative to part of the open boundary of } \ClmN\\
& \text{ given by configurations that project to a point containing a puncture}\\ 
&\text {and also relative to the boundary } P^{-}.\\
&\bullet H^{lf, \Delta}_{m}(\ClmN, \partial; \Z) \text{ the homolgy relative to }\\
& \text{ part of the the boundary of the covering hich is not in } P^{-} \text{and Borel-Moore}\\
& \text{ with respect to collisions of points in the configuration space}.
\end{aligned}
\end{equation*}
\end{defn}
\begin{defn}[Homology of the level $\bar{\cN}$ covering]
We consider the submodules in these Borel-Moore homologies of the covering space $\ClmN$ that are the image of twisted Borel-Moore homology of the base space $\Clm$ (twisted by the local system $\Phi^{\bar{\cN}}$), as in \cite{CrI} and [\cite{CrM}, Theorem E]:
\begin{enumerate}
 \item[$\bullet$]  $\mathscr H^{\bar{\cN}}_{n,m}\subseteq H^{\text{lf},\infty,-}_m(\ClmN, P^{-1};\Z)$ and 
 \item[$\bullet$]  $\mathscr H^{\bar{\cN},\partial}_{n,m} \subseteq H^{\text{lf},\Delta}_m(\ClmN,\partial;\Z)$. 
\end{enumerate}
As we have seen, these homologies are modules over $$\C[w^1_1,…,w^{l}_{\cN-1},u_1^{\pm 1},...,u_{l}^{\pm 1},x_1^{\pm 1},...,x_{n}^{\pm 1}, y^{\pm 1},d^{\pm 1}]
.$$
\end{defn}

\subsection{Specialisations given by colorings}
Up to this moment, in the definition of the homology groups we did not use any information coming from braid representatives of our links. Now we continue our homological set-up for the case where we have a link with $l$ components and braid representative $\beta_n$ with $n$-strands that gives our link by  braid closure. This will induce a colouring, as follows.

\begin{defn}[Colouring the punctures $C$]\label{br}
Let $C:\{1,...,2n\}\rightarrow \{1,...,l\}$ be the associated coloring of the $2n$ $p$-punctures of the disc $\{1,...,2n\}$ with $l$ colours, induced by $\beta_{n} \cup {\mathbb I}_{n}$.
\end{defn}

\begin{defn}[Change of coefficients $f_{C}$]\label{fc}
This induces a change the variables corelated to the punctures of the punctured disc, defined as:\begin{equation}
\begin{aligned}
f_C: &~\C[w^1_1,…,w^{l}_{\cN-1},u_1^{\pm 1},...,u_{l}^{\pm 1},x_1^{\pm 1},...,x_{n}^{\pm 1}, y^{\pm 1},d^{\pm 1}]
 \rightarrow \\
 & \ \C[w^1_1,...,w^{l}_{\cN-1},u_1^{\pm 1},...,u_{l}^{\pm 1},x_1^{\pm 1},...,x_l^{\pm 1},y^{\pm1}, d^{\pm 1}]
 \end{aligned}
 \end{equation}

\begin{equation}\label{eq:8} 
f_C(x_i)=x_{C(i)}, \ i\in \{1,...,n\}.
\end{equation}
\end{defn}
Now, we look at this change of coefficients at the level of the homology groups, via the function $f_C$.
\begin{defn}(Homology groups)\label{D:4} We consider the two homologies defined above, specialised over the ring associated to this new change of coefficients:
\begin{enumerate}
 \item[$\bullet$]  $H^{\bar{\cN}}_{n,m}:=\mathscr H^{\bar{\cN}}_{n,m}|_{f_C}$ 
 \item[$\bullet$]  $H^{\bar{\cN},\partial}_{n,m}:=\mathscr H^{\bar{\cN},\partial}_{n,m}|_{f_C}.$
\end{enumerate}
These homology groups are $\C[w^1_1,...,w^{l}_{\cN-1},u_1^{\pm 1},...,u_{l}^{\pm 1},x_1^{\pm 1},...,x_l^{\pm 1},y^{\pm1}, d^{\pm 1}]$-modules.
\end{defn}

We will use a geometric intersection pairing that is a  Poincaré-Lefschetz type duality for twisted homology (see \cite[Proposition 3.2]{CrM}] and also \cite[Lemma 3.3]{CrM}).
\begin{prop}(\cite[Proposition 7.6]{CrM})\label{P:3'''}
There exists a well-defined topological intersection pairing between these homology groups:
$$\langle ~,~ \rangle: H^{\bar{\cN}}_{n,m} \otimes \mathscr H^{\bar{\cN},\partial}_{n,m} \rightarrow \C[w^1_1,…,w^{l}_{\cN-1},u_1^{\pm 1},...,u_{l}^{\pm 1},x_1^{\pm 1},...,x_{l}^{\pm 1}, y^{\pm 1},d^{\pm 1}].$$
\end{prop}
\subsection{Computation of the geometric intersection pairing}\label{comp}
This intersection form has the nice feature that even if it is defined at the level of the covering space, it is encoded by geometric intersections on the base spase, graded by the local system. We refer to \cite[Section 7]{CrM} for the complete construction, and we present the main steps for the computations below.

\begin{notation}[Twisted local system $\bar{\Phi}^{\bar{\cN}}$]\phantom{A} \label{twsyst}\\ 
Let $\tilde{\Phi}^{\bar{\cN}}$ be the morphism induced by the level $\bar{\cN}$ local system $\Phi^{\bar{\cN}}$, that takes values in the group ring of $\Z^n \oplus \Z \oplus \Z$:
\begin{equation}\label{sgr}
\tilde{\Phi}^{\bar{\cN}}: \pi_1(\Clm) \rightarrow \C[x_1^{\pm 1},...,x_{n}^{\pm 1}, y^{\pm 1},{d'}^{\pm 1}].
\end{equation}
Then, using the change of variables $ \iota$ from Definition \ref{iota} and the change of coefficients $f_C$ from Definition \ref{fc} we consider:
\begin{equation}\label{sor}
\begin{aligned}
&\bar{\Phi}^{\bar{\cN}}: \pi_1(\Clm) \rightarrow \C[w^1_1,…,w^{l}_{\cN-1},u_1^{\pm 1},...,u_{l}^{\pm 1},x_1^{\pm 1},...,x_{l}^{\pm 1}, y^{\pm 1},d^{\pm 1}]\\
&\bar{\Phi}^{\bar{\cN}}=f_C \circ \iota \circ \tilde{\Phi}^{\bar{\cN}}. \ \ \ \ \ \ \ \ \ \ \ \ \ \ \ \ \ \ \ \ 
\end{aligned}
\end{equation}
This means that the monodromies of the associated local system are given by the following expression:
\begin{equation}
\begin{cases}
&\bar{\Phi}^{\bar{\cN}}(\sigma_0)=0\\
&\bar{\Phi}^{\bar{\cN}}(\sigma_i)=x^2_{C(i)},\\ 
&\bar{\Phi}^{\bar{\cN}}(\bar{\sigma}_i)=(-1)^{(\cN_{C(i)}-1)} x^2_{C(i)} \cdot d^{2(\cN_{C(i)}-1)}, i\in \{1,...,n\}\\
&\bar{\Phi}^{\bar{\cN}}(\gamma)=y\\
&\bar{\Phi}^{\bar{\cN}}(\delta)=-d^{2}.
\end{cases}
\end{equation}
 
\end{notation}
In the following part we describe the explicit formula for the intersection pairing, that will use the monodromies introduced in the above definition. 

Let us fix two homology classes $H_1 \in H^{\bar{\cN}}_{n,m}$ and $H_2 \in H^{\bar{\cN},\partial}_{n,m}$. We suppose that these classes are given by two submanifolds $\tilde{X}_1, \tilde{X}_2$ in the covering, which are lifts of immersed submanifolds $X_1,X_2 \subseteq \Clm$. Moreover, we assume that $X_1$ and $X_2$ have a transversal intersection, in a finite number of points. 
  
The intersection pairing is encoded by the geometric intersections between these submanifolds in the base configuration space, graded in specific manner using the local system, as below. 

{\bf 1) Loop associated to an intersection point} The first step is to associate to each intersection point $x \in X_1 \cap X_2$ a loop in the configuration space, denoted by $l_x \subseteq \Clm$. Then, we will grade this loop using the local system $\bar{\Phi}^{\bar{\cN}}$.
  
\begin{defn}[Loop $l_x$]
Let $x \in X_1 \cap X_2$. We suppose that we have two paths $\gamma_{X_1}, \gamma_{X_2}$ which start in $\bf d$ and end on $X_1$ and $X_2$ respectively such that
$\tilde{\gamma}_{X_1}(1) \in \tilde{X}_1$ and $ \tilde{\gamma}_{X_2}(1) \in \tilde{X}_2$. Here, $\tilde{\gamma}_{X_1}, \tilde{\gamma}_{X_2}$ are the lifts of ${\gamma}_{X_1}, {\gamma}_{X_2}$ through the base point of the covering  $\bf \tilde{d}$.

For the next step we choose $\nu_{X_1}, \nu_{X_2}:[0,1]\rightarrow \Clm$ two paths such that:
\begin{equation}
\begin{cases}
\nu_{X_1}(0)=\gamma_{X_1}(1);  \nu_{X_1}(1)=x; Im(\nu_{X_1})\subseteq X_1\\
\nu_{X_2}(0)=\gamma_{X_2}(1);  \nu_{x_2}(1)=x; Im(\nu_{X_2})\subseteq X_2.
\end{cases}
\end{equation}
Our loop is given by the concatenation of these four paths, as below:
$$l_x=\gamma_{X_1}\circ\nu_{X_1}\circ \nu_{X_2}^{-1}\circ \gamma_{X_2}^{-1}.$$
\end{defn}
{\bf 2) Grade the family of loops using the local system}
\begin{prop}[Intersection pairing via geometric intersections in the base space]\label{P:3}
The intersection pairing between the homology classes can be obtained from the set of loops $l_x$ and graded by the local system, as below:
\begin{equation}\label{eq:1}  
\langle H_1,H_2 \rangle= \sum_{x \in X_1 \cap X_2}  \alpha_x \cdot \bar{\Phi}^{\bar{\cN}}(l_x) \in \C[w^1_1,…,w^{l}_{\cN-1},u_1^{\pm 1},...,u_{l}^{\pm 1},x_1^{\pm 1},...,x_{l}^{\pm 1}, y^{\pm 1},d^{\pm 1}]
\end{equation}
where $\alpha_x$ is a sign given by the product of local orientations in the disc around each component of the intersection point $x$.
\end{prop}
\section{Weighted Lagrangian intersection at level $\cN$}\label{S:we}
 In this part, we are going to define the weighted Lagrangian intersection that will be the principal tool for the construction of our universal link invariants. We will use the homological ingredients introduced in the previous sections, associated to the following parameters. 

{\bf Context} Let us fix a level $\cN \in \mathbb N$ and $L$ an oriented link that is the closure of a braid $\beta_{n} \in B_n$.
 For the level $\cN$, we are going to use the following multi-level:
\begin{equation}
\bar{\cN}:=(1,\cN-1,...,\cN-1).
\end{equation}

We then consider the configuration space of $$2+(n-1) (\cN-1)$$ points on the $(2n+2)$-punctured disc and the local system $\Phi^{\bar{\cN}}$ associated to the  parameters:
$$ n \rightarrow n; \ \ \ m\rightarrow 2+(n-1) (\cN-1); \ \ \ \ \bar{\cN}.$$ 
We obtain the two associated homology groups, introduced in the above section, which we denote as:
$$ \HA \ \ \ \ \ \ \ \ \ \ \ \ \ \ \ \text{ and }\ \ \ \ \ \ \ \ \ \ \ \ \ \ \HAd.$$

\subsection{Homology classes}
Once we have all the set up given by the two homology groups and their intersection pairing, we are ready to introduce the main ingredients for our topological model, which are specific homology classes. For their construction, we will use the following procedure.

\begin{notation}[Homology classes from geometric supports]\label{paths}

\

We use a dictionary that encodes homology classes the covering of the configuration space by the following data in the base configuration space:

\begin{itemize}
\item[•] A {\em geometric support}, that is a fixed {\em set of arcs in the punctured disc} or {\em ovals in the punctured disc}. We look at the unordered configurations of a prescribed number of particles on each such set of arcs or ovals. The image of the product of these configurations on all arcs or configurations on all ovals gives a submanifold $F$ in the configuration space (which has half of the dimension of the configuration space). 
\item[•] A set of {\em connecting paths to the base point}, that start in the base points from the punctured disc and end on these curves or ovals. The set of all these paths leads to a path in the configuration space,  that starts in $\bf d$ and ends on the submanifold $F$. 
\end{itemize}

\

Now, we assume that we are in a situation where the submanifold $F$ has a well defined lift to the covering space.
First, we lift the path to a path in the covering space, that starts in $\tilde{\bf{d}}$. The second step is to lift the submanifold $F$ through the end point of this path. The precise construction of such homology classes using this dictionary is presented in \cite[Section 5]{Crsym}. 
\end{notation}

In the sequel we define the specific homology classes that we use for the weighted intersection model at level $\cN$. An important feature of the construction of the local system at level $\cN$, which comes from \cite{CrI}, is that it leads to a covering space at level $\cN$ where we have well-defined lifts of submanifolds that we want to work with.
Let us introduce the following classes. 

\begin{defn} (Level $\cN$ Homology classes)\\
Let $\bar{i}=(i_1,...,i_{n}) \in \{\bar{0},\dots,\overline{\cN-1}\}$ be a fixed multi-index. We consider the homology classes associated to the geometric supports from Figure \ref{Picture0'} (which depend on the components of the multi-index $\bar{i}$):
 \begin{figure}[H]
\centering
$${\color{red} \FA \in \HA} \ \ \ \ \ \ \ \ \ \ \text{ and } \ \ \ \ \ \ \ \ \ \ \ \ \ {\color{dgreen} \LA \in \HAd} .$$
$$\hspace{5mm}\downarrow \text{ lifts }$$
\vspace{-4mm}

\includegraphics[scale=0.4]{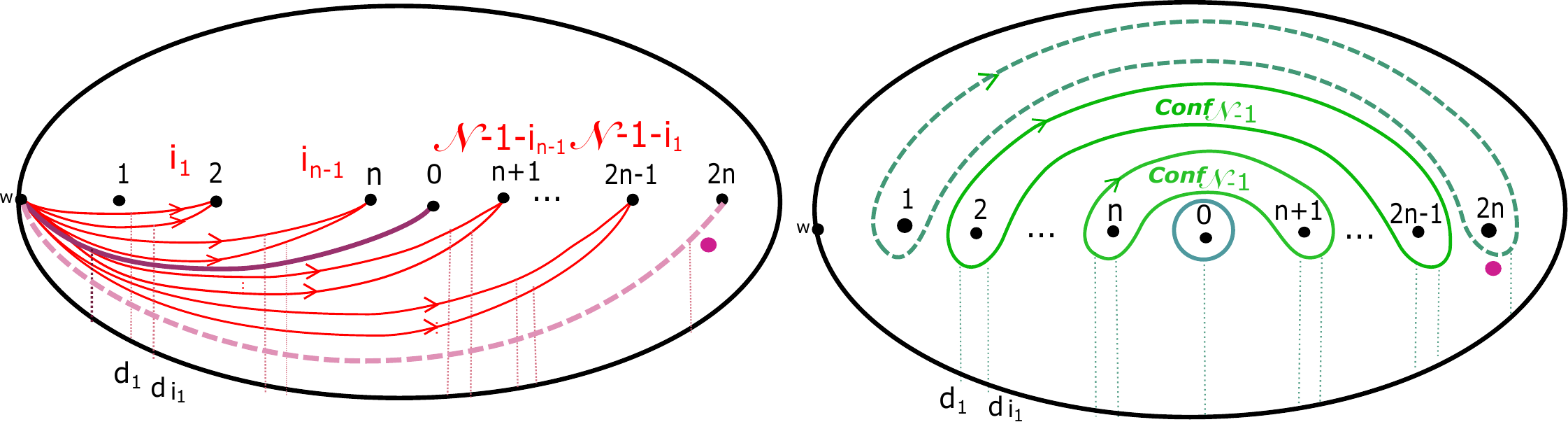}
\vspace{2mm}

\caption{\normalsize Lagrangians for the level $\cN$ weighted intersection}
\label{Picture0'}
\end{figure}
\end{defn}
In \cite{CrI}, we have shown that the geometric support from Figure \ref{Picture0'} leads to a well-defined homology class in the level $\cN$ covering, which we denote by $\LA \in \HAd$.  So, the above homology classes are well defined and they are the main objects that are used for the weighted intersection, as follows.

\begin{prop}[Intersection pairing]\label{intersection} Let us recall that we have an intersection pairing between these homology groups, following Proposition \ref{P:3'''}: 
$$\langle ~,~ \rangle: \HA \otimes \HAd \rightarrow \C[w^1_1,…,w^{l}_{\cN-1},u_1^{\pm 1},...,u_{l}^{\pm 1},x_1^{\pm 1},...,x_{l}^{\pm 1}, y^{\pm 1},d^{\pm 1}].$$
Also, following \cite[Section 4.10]{CrI}, we have a well-defined braid group action of $B^{C}_{2n+2}$ on the homology $\HA$, where $B^{C}_{2n+2}$ are the braids which respect the colouring $C$.
\end{prop}

\begin{defn}[{\bf \em Weighted Lagrangian intersection}]\label{omJo} Let us consider the weighted Lagrangian intersection in $\Conf_{2+(n-1)(\cN-1)}\left(\mathbb D_{2n+2}\right)$, with {\color{red} weights} given by the variables ${\color{red}w^1_1,...,w^{l}_{\cN-1}}$:
\begin{equation} 
\begin{aligned}
 \PA \in \ & \ \LL = \C[w^1_1,...,w^{l}_{\cN-1},u_1^{\pm 1},...,u_l^{\pm 1},x_1^{\pm 1},...,x_l^{\pm 1},y^{\pm1}, d^{\pm 1}]\\
 \PA :=&\prod_{i=1}^l u_{i}^{ \left(f_i-\sum_{j \neq {i}} lk_{i,j} \right)} \cdot \prod_{i=2}^{n} u^{-1}_{C(i)} \cdot \\
 & \cdot \sum_{\bar{i}= \bar{0}}^{\overline{\cN -1}} w^{C(2)}_{i_1}\cdot ... \cdot w^{C(n)}_{i_{n-1}}
 \left\langle(\beta_{n} \cup {\mathbb I}_{n+2} ) \ { \color{red} \FA}, {\color{dgreen} \LA}\right\rangle. 
   \end{aligned}
\end{equation} 
\end{defn}

\section{Unifying all Coloured Jones and ADO link invariants with colours bounded by $\cN$}\label{SAu}

In this part we put together the topological tools and we will prove that for a fixed $\cN$, the weighted intersection at level $\cN$ recovers all coloured Jones polynomials and all coloured Alexander polynomials of levels less than $\cN$, as presented in Theorem \ref{THEOREMAU} which we remind below.

\begin{thm}[{\bf \em Unifying coloured Alexander and coloured Jones polynomials of bounded level}]\label{THEOREMAU'}
Let us fix $\cN\in \N$. Then, $\PA$ recovers all coloured Alexander and Jones polynomials of bounded colours, as below:
\begin{equation}
\begin{aligned}
\Phi^{\cM}(L)& =~ \PA \Bigm| _{\sAt}, \ \ \ \forall \cM\leq \cN\\
J_{\bar{N}}(L)& =~ \PA \Bigm| _{\sJt}, \ \ \ \forall \bar{N} \leq \cN.
\end{aligned}
\end{equation} 
\end{thm}
We will split the proof of this statement in two main steps. First, we prove that we recover the coloured Alexander polynomials from our weighted intersection. Secondly, we turn our attention to the multi-colour case for the coloured Jones polynomials and show that we recover all these invariants with multicolours bounded by $\cN$ from our level $\cN$ weighted intersection. 

\subsection{First case--unifying the non semi-simple ADO link invariants}
\begin{thm}[{\bf \em Recovering coloured Alexander polynomials of bounded levels}]\label{THEOREMA}
The graded intersection $\PA$ recovers the $\cM^{th}$ coloured Alexander polynomial of $L$ as below:
\begin{equation}
\begin{aligned}
&\Phi^{\cM}(L) =~ \PA \Bigm| _{\sAt}  \forall \cM \leq \cN.
\end{aligned}
\end{equation} 
\end{thm}
\begin{proof}
This property relies on the topological model for coloured Alexander polynomials for coloured links that we have constructed in \cite{CrI}.

\begin{thm}[{Non-weighted topological model for coloured Alexander polynomials}]\label{omADO}
Let us define the following non-weighted Lagrangian intersection:
\begin{equation} 
\begin{aligned}
& \PAo \in \Z[u_1^{\pm 1},...,u_l^{\pm 1},x_1^{\pm 1},...,x_l^{\pm 1},y^{\pm 1}, d^{\pm 1}]\\
& \PAo:=\prod_{i=1}^l u_{i}^{ \left(f_i-\sum_{j \neq {i}} lk_{i,j} \right)} \cdot \prod_{i=2}^{n} u^{-1}_{C(i)} \cdot  \sum_{\bar{i}=\bar{0}}^{\overline{\cM -1}}  \left\langle(\beta_{n} \cup {\mathbb I}_{n+2} ) \ { \color{red} \FAM}, {\color{dgreen} \LAM}\right\rangle. 
   \end{aligned}
\end{equation}
Also, we consider the change of coefficients given by the formula:
$$ \sAto: \Z[u_1^{\pm 1},...,u_{l}^{\pm 1},x_1^{\pm 1},...,x_{l}^{\pm 1}, y^{\pm 1},d^{\pm 1}] \rightarrow \Q(\xi_{\cM})(x^{2\cM}_{1}-1,...,x^{2\cM}_{l}-1)^{-1}[x_1^{\pm 1},...,x_{l}^{\pm 1}]$$
\begin{equation}
\begin{cases}
&\sAto(u_j)=x_j^{(1-\cM)}\\
&\sAto(y)=(d[\lambda_{C(1)}]_{\xi_{\cM}}),\\
&\sAto(d)= \xi_{\cM}^{-1}.
\end{cases}
\end{equation}
Then $\PAo$ recovers the $\cM^{th}$ coloured Alexander polynomials 
\begin{equation}
\begin{aligned}
&\Phi^{\cM}(L) =~ \PAo \Bigm| _{\sAto}.
\end{aligned}
\end{equation} 
\end{thm}

Now let us recall the definition of the weighted intersection form at level $\cN$:
\begin{equation}\label{f1-form}
\begin{aligned}
& \PA:=\prod_{i=1}^l u_{i}^{ \left(f_i-\sum_{j \neq {i}} lk_{i,j} \right)} \cdot \prod_{i=2}^{n} u^{-1}_{C(i)} \cdot \\
 & \ \ \ \ \ \ \ \ \ \ \ \ \ \ \ \ \ \ \ \sum_{\bar{i}\in\{\bar{0},\dots,\overline{\cN-1}\}}w^{C(2)}_{i_1}\cdot ... \cdot w^{C(n)}_{i_{n-1}}
   \left\langle(\beta_{n} \cup {\mathbb I}_{n+2} ) \ { \color{red} \FA}, {\color{dgreen} \LA}\right\rangle. 
 \end{aligned}
\end{equation} 

Also, the specialisation of coefficients that we defined for $\PA$ in the case of coloured Alexander polynomials is given by (see Definition \ref{pA2}):
$$ \sAt: \Z[w^1_1,...,w^{l}_{\cN-1},u_1^{\pm 1},...,u_{l}^{\pm 1},x_1^{\pm 1},...,x_{l}^{\pm 1},y^{\pm 1}, d^{\pm 1}] \rightarrow \Q(\xi_{\cM})(x^{2\cM}_{1}-1,...,x^{2\cM}_{l}-1)^{-1}[x_1^{\pm 1},...,x_{l}^{\pm 1}]$$
\begin{equation}
\begin{cases}
&\sAt(u_j)=x_j^{(1-\cM)}\\
&\sAt(y)=(d[\lambda_{C(1)}]_{\xi_{\cM}}),\\
&\sAt(d)= \xi_{\cM}^{-1}\\
&\sAt(w^k_j)=1, \ \text{ if } j\leq \cM-1,\\
&\sAt(w^k_j)=0, \ \text{ if } j\geq \cM,  k\in \{1,...,l\}, j\in \{1,...,\cN-1\}.
\end{cases}
\end{equation}

In the following part we will prove that the intersection forms $\PAo$ and $\PA$ are related. 
\begin{lem}[Intersection forms become equal once non semi-simply specialised at lower levels]\label{wnwA}
The state sums of Lagrangian intersections $\PA$ and $\PAo$ become equal when specialised through $\sAt$ and $\sAto$ respectively:
\begin{equation}\label{f3}
\begin{aligned}
\left(\PA \right)\Bigm| _{\sAt} =~ \left(\PAo \right)\Bigm| _{\sAto}, \forall \cM \leq \cN.
\end{aligned}
\end{equation}

\end{lem}
\begin{proof}

Following \eqref{f1-form}, the specialisation of the weighted intersection form is given by:
\begin{equation}
\begin{aligned}
& \PA\Bigm| _{\sAt}=\left( \prod_{i=1}^l u_{i}^{ \left(f_i-\sum_{j \neq {i}} lk_{i,j} \right)} \cdot \prod_{i=2}^{n} u^{-1}_{C(i)} \cdot \right. \\
 & \ \ \ \ \ \ \ \ \ \ \ \ \ \ \ \ \ \ \ \left. \sum_{\bar{i}\in\{\bar{0},\dots,\overline{\cN-1}\}}w^{C(2)}_{i_1}\cdot ... \cdot w^{C(n)}_{i_{n-1}}
   \left\langle(\beta_{n} \cup {\mathbb I}_{n+2} ) \ { \color{red} \FA}, {\color{dgreen} \LA}\right\rangle\right) \Bigm| _{\sAt}. 
 \end{aligned}
\end{equation} 

We will seprate the above sum into two parts, associated to multi-indices bounded by $\cM$ and the rest of the multi-indices, bounded just by $\cN$, as follows:
\begin{equation}
\begin{aligned}
& \PA\Bigm| _{\sAt}= \left( \prod_{i=1}^l u_{i}^{ \left(f_i-\sum_{j \neq {i}} lk_{i,j} \right)} \cdot \prod_{i=2}^{n} u^{-1}_{C(i)} \cdot \right. \\
 & \ \ \ \ \ \ \ \ \ \ \ \ \ \ \ \ \ \ \ \left. \sum_{\bar{i}\in\{\bar{0},\dots,\overline{\cM-1}\}}w^{C(2)}_{i_1}\cdot ... \cdot w^{C(n)}_{i_{n-1}}
   \left\langle(\beta_{n} \cup {\mathbb I}_{n+2} ) \ { \color{red} \FA}, {\color{dgreen} \LA}\right\rangle \right) \Bigm| _{\sAt}+\\
& \ \ \ \ \ \ \ \ \ \ \ \ \ \ \ + \left( \prod_{i=1}^l u_{i}^{ \left(f_i-\sum_{j \neq {i}} lk_{i,j} \right)} \cdot \prod_{i=2}^{n} u^{-1}_{C(i)} \cdot \right. \\
 & \ \ \ \ \ \ \ \ \ \ \ \ \ \ \ \ \ \ \ \left. \sum_{\bar{i}\in\{\overline{\cM},\dots,\overline{\cN-1}\}}w^{C(2)}_{i_1}\cdot ... \cdot w^{C(n)}_{i_{n-1}}
   \left\langle(\beta_{n} \cup {\mathbb I}_{n+2} ) \ { \color{red} \FA}, {\color{dgreen} \LA}\right\rangle \right) \Bigm| _{\sAt}.  
 \end{aligned}
\end{equation} 
(see \eqref{mi} for the definition of sets of multi-indices).

We remark that if we have an index such that $\bar{i}\in \{\overline{\cM},\dots,\overline{\cN-1}\}$, then  there exists $j \in \{1,...,n-1\}$ such that $i_j \geq \cM$. In turn, using the definition of the specialisation $\sAt$, this means that the coefficient $w^{C(2)}_{i_1}\cdot ... \cdot w^{C(n)}_{i_{n-1}}
$ will vanish through this specialisation:

\begin{equation}
\sAt\left(w^{C(2)}_{i_1}\cdot ... \cdot w^{C(n)}_{i_{n-1}}
\right)=0, \ \forall \ \bar{i}\in\{\bar{\cM},\dots,\overline{\cN-1}\}.
\end{equation}

From this we conclude that:
\begin{equation}
\begin{aligned}
& \ \ \ \ \ \ \ \ \ \ \ \ \ \ \  \left( \prod_{i=1}^l u_{i}^{ \left(f_i-\sum_{j \neq {i}} lk_{i,j} \right)} \cdot \prod_{i=2}^{n} u^{-1}_{C(i)} \cdot \right. \\
 & \ \ \ \ \ \ \ \ \ \ \ \ \ \ \ \ \ \ \ \left. \sum_{\bar{i}\in\{\overline{\cM},\dots,\overline{\cN-1}\}}w^{C(2)}_{i_1}\cdot ... \cdot w^{C(n)}_{i_{n-1}}
   \left\langle(\beta_{n} \cup {\mathbb I}_{n+2} ) \ { \color{red} \FA}, {\color{dgreen} \LA}\right\rangle \right) \Bigm| _{\sAt}=0. 
 \end{aligned}
\end{equation} 
This property shows the fact that the weighted intersection sees just the classes associated to indices that are bounded by $\cM$, once we do the specialisation at level $\cM$, and we have the formula:

\begin{equation}
\begin{aligned}
& \PA\Bigm| _{\sAt}= \left( \prod_{i=1}^l u_{i}^{ \left(f_i-\sum_{j \neq {i}} lk_{i,j} \right)} \cdot \prod_{i=2}^{n} u^{-1}_{C(i)} \cdot \right. \\
 & \ \ \ \ \ \ \ \ \ \ \ \ \ \ \ \ \ \ \ \left. \sum_{\bar{i}\in\{\bar{0},\dots,\overline{\cM-1}\}}w^{C(2)}_{i_1}\cdot ... \cdot w^{C(n)}_{i_{n-1}}
   \left\langle(\beta_{n} \cup {\mathbb I}_{n+2} ) \ { \color{red} \FA}, {\color{dgreen} \LA}\right\rangle \right) \Bigm| _{\sAt}.
 \end{aligned}
\end{equation} 

Secondly, we remark that:
\begin{equation}
\sAt(w^{C(2)}_{i_1}\cdot ... \cdot w^{C(n)}_{i_{n-1}}
)=1, \forall \bar{i}\in\{\bar{0},\dots,\overline{\cM-1}\}.
\end{equation}
So, our intersection becomes:
\begin{equation}
\begin{aligned}
& \PA\Bigm| _{\sAt}= \left( \prod_{i=1}^l u_{i}^{ \left(f_i-\sum_{j \neq {i}} lk_{i,j} \right)} \cdot \prod_{i=2}^{n} u^{-1}_{C(i)} \cdot \right. \\
 & \ \ \ \ \ \ \ \ \ \ \ \ \ \ \ \ \ \ \ \left. \sum_{\bar{i}\in\{\bar{0},\dots,\overline{\cM-1}\}} \left\langle(\beta_{n} \cup {\mathbb I}_{n+2} ) \ { \color{red} \FA}, {\color{dgreen} \LA}\right\rangle \right) \Bigm| _{\sAt}.
 \end{aligned}
\end{equation}

This formula is close to the formula for the non-weighted intersection $ \PAo$, the only difference is that the non-weighted sum uses the classes $$\FAM \text{   and   } \LAM$$ and the weighted intersection $\PA$ is given by the classes

$$\FA \text{ and } \LA.$$ 

Even so, through the intersection pairing, these classes lead to the same result, if we know that the index is bounded by $\cM$. This comes following a geometric result which we proved in \cite[Lemma 7.2]{CrI} and we present below.
\begin{lem}[Intersections between classes associated to indices less than $\cM$ \cite{CrI}]\label{pr1}

The intersection pairings between classes associated to indices $\bar{i}$ bounded by $\bar{\cM}$ give the same result, even before applying the specialisation of coefficients:
\begin{equation}
\begin{aligned}
 \left\langle(\beta_{n} \cup {\mathbb I}_{n+2} ) \ { \color{red} \FAM}, {\color{dgreen} \LAM}\right \rangle= \left\langle(\beta_{n} \cup {\mathbb I}_{n+2} ) \ { \color{red} \FA}, {\color{dgreen} \LA}\right\rangle , \forall \ \bar{i} \in\{\bar{0},\dots,\overline{\cM-1}\}.  
 \end{aligned}
\end{equation}
\end{lem}
Using this property, we obtain that the weighted intersection has the formula:
\begin{equation}
\begin{aligned}
& \PA\Bigm| _{\sAt}= \left( \prod_{i=1}^l u_{i}^{ \left(f_i-\sum_{j \neq {i}} lk_{i,j} \right)} \cdot \prod_{i=2}^{n} u^{-1}_{C(i)} \cdot \right. \\
 & \ \ \ \ \ \ \ \ \ \ \ \ \ \ \ \ \ \ \ \left. \sum_{\bar{i}\in\{\bar{0},\dots,\overline{\cM-1}\}} \left\langle(\beta_{n} \cup {\mathbb I}_{n+2} ) \ { \color{red} \FAM}, {\color{dgreen} \LAM}\right\rangle \right) \Bigm| _{\sAt}.
 \end{aligned}
\end{equation} 
This shows that we recover the non-weighted intersection, once we apply this specialisation of coefficients:
\begin{equation}
\begin{aligned}
\left(\PA \right)\Bigm| _{\sAt} =~ \left(\PAo \right)\Bigm| _{\sAto}, \forall \cM \leq \cN.
\end{aligned}
\end{equation} 

This concludes the proof of Lemma \ref{wnwA}.

\end{proof}
On the other hand, we know that the non-weighted intersection recovers the $\cM^{th}$ ADO invariant, following Theorem \ref{omADO}:
\begin{equation}
\begin{aligned}
&\Phi^{\cM}(L) =~ \PAo \Bigm| _{\sAto}.
\end{aligned}
\end{equation} 

Putting everything together, we conclude that the weighted intersection recovers all the ADO invariants at levels bounded by $\cN$:
\begin{equation}
\begin{aligned}
&\Phi^{\cM}(L) =~ \PA \Bigm| _{\sAt}, \forall \cM \leq \cN.
\end{aligned}
\end{equation}
This concludes the proof of the globalising Theorem for all ADO link invariants.

\end{proof}
\subsection{Second case -- unifying the semi-simple link invariants}
Let us consider a fixed level $\cN$ and let $N_1,...,N_l\in \N$ be a set of colours for our link which are all less or equal than $\cN$. We denote $\bar{N}:=(N_1,...,N_l).$

In this part we will show that for the fixed level $\cN$, the weighted intersection at level $\cN$ recovers all coloured Jones polynomials at bounded colours, as presented in Theorem \ref{THEOREMAU}, which we remind below.

\begin{thm}[{\bf \em Unifying coloured Jones polynomials of bounded multi-colours}]\label{THEOREMJ}
For a fixed $\cN\in \N$, $\PA$ recovers all coloured Jones polynomials of links with multi-colours bounded by $\cN$:
\begin{equation}
J_{\bar{N}}(L) =~ \PA \Bigm| _{\sJt}, \ \ \ \forall \ \bar{N} \leq \cN.
\end{equation} 
\end{thm}
The proof of this theorem will make use of a non-weighted topological model for coloured Jones polynomials for coloured links, which we previously constructed in \cite{CrI}. First, we present a summary of the construction of this model. Then, we will prove that the new weighted Lagrangian intersection model recovers the non-weighted model once we do the appropriate change of coefficients. 

\subsubsection{Non-weighted topological model for coloured Jones polynomials}
 As before, we choose a braid representative for our link. Using the colouring $C$ induced by this braid, we denote $\ccN_i:=N_{C(i)}.$ 
 
{\bf \em Context for the non-weighted topological model for coloured Jones polynomials} 

In this case, we make use of the homological set-up associated to the following parameters (which depend on the individual colours $N_1,...,N_l$):
\begin{itemize}
\item Configuration space: $C_{n,m_J(\bar{N})}:=\Conf_{2+\sum_{i=2}^{n} N^C_i}\left(\mathbb D_{2n+2}\right)$ 
\item Number of particles: $m_J(\bar{N}):=2+\sum_{i=2}^{n} N^C_i$
\item Multi-level: $\bar{\cN}(\bar{N}):=(1,\ccN_2-1,...,\ccN_n-1)$
\item Local system: $\bar{\Phi}^{\bar{\cN}(\bar{N})}$ ({\bf \em  depends on the colours})
\item Specialisation of coefficients: $\sJt$ (Notation \ref{pJ1}, Definition \ref{pJ2}).
\end{itemize}
\begin{rmk}(The non-weighted model for coloured Jones polynomials (\cite{CrI}) depends on the colours)
We would like to emphasise that the number of particles $m_J(\bar{N})$, the multi-level $\bar{\cN}(\bar{N})$ and local system $\bar{\Phi}^{\bar{\cN}(\bar{N})}$ depend on the choice of colours $\bar{N}$. So, when we vary the colouring, the whole topological context for the non-weighted intersection changes. The advantage of the weighted topological model introduced in this paper is that once we fix a level $\cN$, the topology of the weighted  intersection at level $\N$ will capture all the phenomena at colourings bounded by $\cN$.
\end{rmk}
Let us denote the following set of multi-indices:
\begin{equation}\label{stJi}
C(\bar{N}):= \big\{ \bar{i}=(i_1,...,i_{n-1})\in \N^{n-1} \mid 0\leq i_k \leq \ccN_{k+1}-1, \  \forall k \in \{1,...,n-1\} \big\}.
\end{equation}
\begin{defn}(Homology classes for the non-weighted model for coloured Jones polynomials)
For $\bar{i}\in C(\bar{N})$ consider the two homology classes given by the geometric supports from Figure \ref{Picture0i}:
\vspace*{-8mm}
 \begin{figure}[H]
\centering
$${\color{red} \FJ \in \HJi} \ \ \ \ \ \ \ \ \ \ \text{ and } \ \ \ \ \ \ \ \ \ \ \ \ \ {\color{dgreen} \LJ \in \HJdi} .$$
\vspace{-2mm}
\includegraphics[scale=0.4]{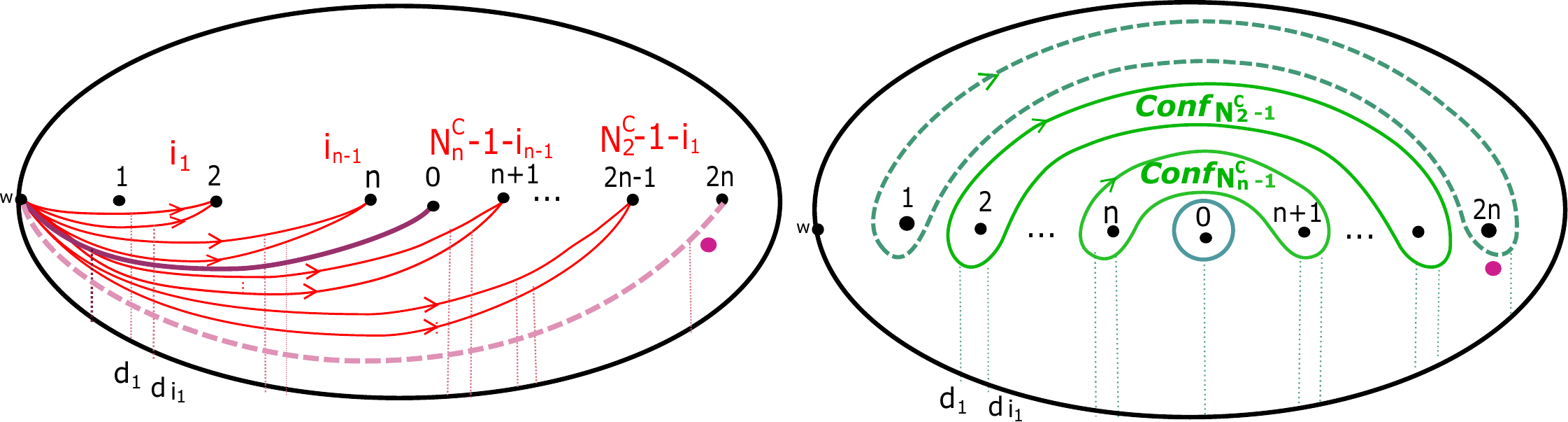}
\caption{ \normalsize Lagrangians for the non-weighted intersection}
\label{Picture0i}
\end{figure}
\end{defn}
\begin{thm}[{Non-weighted topological model for coloured Jones polynomials}]\label{omJones}
Let us define the following Lagrangian intersection:
\begin{equation} 
\begin{aligned}
& \PJ \in \Z[u_1^{\pm 1},...,u_l^{\pm 1},x_1^{\pm 1},...,x_l^{\pm 1},y^{\pm 1}, d^{\pm 1}]\\
& \PJ:=\prod_{i=1}^l u_{i}^{ \left(f_i-\sum_{j \neq {i}} lk_{i,j} \right)} \cdot \prod_{i=2}^{n} u^{-1}_{C(i)} \cdot \\
 & \ \ \ \ \ \ \ \ \ \ \ \ \ \ \ \ \ \ \ \sum_{\bar{i}\in C(\bar{N})} \left\langle(\beta_{n} \cup {\mathbb I}_{n+2} ) \ { \color{red} \FJ}, {\color{dgreen} \LJ}\right\rangle. 
 \end{aligned}
\end{equation}
Also, we consider the change of coefficients given by the formula:
$$ \sJto: \Z[u_1^{\pm 1},...,u_{l}^{\pm 1},x_1^{\pm 1},...,x_{l}^{\pm 1},y^{\pm 1}, d^{\pm 1}] \rightarrow \Z[d^{\pm 1}]$$
\begin{equation}
\begin{cases}
&\sJto(u_j)=\left(\sJt(x_j)\right)^t=d^{1-N_i}\\
&\sJto(x_i)=d^{1-N_i}, \ i\in \{1,...,l\}\\
&\sJto(y)=[N^C_1]_{d^{-1}},\\
\end{cases}
\end{equation}

Then $\PJ$ recovers the coloured Jones polynomials coloured with multicolours $\bar{N}$: 
\begin{equation}
\begin{aligned}
&J_{\bar{N}}(L) =~ \PJ \Bigm| _{\sJto}.
\end{aligned}
\end{equation} 
\end{thm}

\subsubsection{Proof of Theorem \ref{THEOREMJ} in the semi-simple case}
\begin{proof}
We are going to prove that the two intersections: the weighted Lagrangian intersection and the non-weighted Lagrangian intersection become equal once we specialise their coefficients at a multi-level $\bar{N} \leq \cN$, as below.

We recall that in our context, for the weighted Lagrangian intersection $\PA$, the specialisation of coefficients for coloured Jones polynomials from Definition \ref{pJ2} has the following expression:

$$ \sJt: \Z[w^1_1,...,w^{l}_{\cN-1},u_1^{\pm 1},...,u_{l}^{\pm 1},x_1^{\pm 1},...,x_{l}^{\pm 1},y_1^{\pm 1},...,y_{l}^{\pm 1}, d^{\pm 1}]  \rightarrow \Z[d^{\pm 1}]$$
\begin{equation}
\begin{cases}
&\sJt(u_i)=\left(\sJt(x_i)\right)^t=d^{1-N_i}\\
&\sJt(x_i)=d^{1-N_i}, \ i\in \{1,...,l\}\\
&\sJt(y)=[N^C_1]_{d^{-1}},\\
&\sJt(w^k_j)=1, \ \text{ if } j\leq N_k-1 ,\\
&\sJt(w^k_j)=0, \ \text{ if } j\geq N_{k},  k\in \{1,...,l\}, j\in \{1,...,\cN-1\}.
\end{cases}
\end{equation}

\begin{lem}[Intersection forms become equal once semi-simply specialised at lower levels]\label{wnwJ} \phantom{A}\\
Let us fix a level $\cN \in \N$. Then for any multi-level $\bar{N} \leq \cN$ that gives a colouring of our link, the level $\cN$ weighted Lagrangian intersection $\PA$ and the multi-level $\bar{N}$ non-weighted Lagrangian intersection $\PJ$ become equal when specialised through $\sJt$ and $\sJto$ respectively:
\begin{equation}\label{f3}
\begin{aligned}
\left(\PA \right)\Bigm| _{\sJt} =~ \left(\PJ \right)\Bigm| _{\sJto}, \ \  \forall \bar{\cN }\leq \cN.
\end{aligned}
\end{equation}

\end{lem}

The specialisation of the weighted intersection form has the following expression (following \eqref{f1-form}):
\begin{equation}
\begin{aligned}
& \PA\Bigm| _{\sJt}=\left( \prod_{i=1}^l u_{i}^{ \left(f_i-\sum_{j \neq {i}} lk_{i,j} \right)} \cdot \prod_{i=2}^{n} u^{-1}_{C(i)} \cdot \right. \\
 & \ \ \ \ \ \ \ \ \ \ \ \ \ \ \ \ \ \ \ \left. \sum_{\bar{i}\in\{\bar{0},\dots,\overline{\cN-1}\}}w^{C(2)}_{i_1}\cdot ... \cdot w^{C(n)}_{i_{n-1}}
   \left\langle(\beta_{n} \cup {\mathbb I}_{n+2} ) \ { \color{red} \FA}, {\color{dgreen} \LA}\right\rangle\right) \Bigm| _{\sJt}. 
 \end{aligned}
\end{equation} 

Separating this formula into two parts, given by multi-indices bounded by $C(\bar{N})$ and the other multi-indices bounded just by $\cN$, we obtain:
\begin{equation}
\begin{aligned}
& \PA\Bigm| _{\sJt}= \left( \prod_{i=1}^l u_{i}^{ \left(f_i-\sum_{j \neq {i}} lk_{i,j} \right)} \cdot \prod_{i=2}^{n} u^{-1}_{C(i)} \cdot \right. \\
 & \ \ \ \ \ \ \ \ \ \ \ \ \ \ \ \ \ \ \ \left. \sum_{\bar{i}\in C(\bar{N})}w^{C(2)}_{i_1}\cdot ... \cdot w^{C(n)}_{i_{n-1}}
   \left\langle(\beta_{n} \cup {\mathbb I}_{n+2} ) \ { \color{red} \FA}, {\color{dgreen} \LA}\right\rangle \right) \Bigm| _{\sJt}+\\
& \ \ \ \ \ \ \ \ \ \ \ \ \ \ \ + \left( \prod_{i=1}^l u_{i}^{ \left(f_i-\sum_{j \neq {i}} lk_{i,j} \right)} \cdot \prod_{i=2}^{n} u^{-1}_{C(i)} \cdot \right. \\
 & \ \ \ \ \ \ \ \ \ \ \ \ \ \ \ \ \ \ \ \left. \sum_{\bar{i}\in\{\overline{0},\dots,\overline{\cN-1}\}, \bar{i}\notin C(\bar{N})}w^{C(2)}_{i_1}\cdot ... \cdot w^{C(n)}_{i_{n-1}}
   \left\langle(\beta_{n} \cup {\mathbb I}_{n+2} ) \ { \color{red} \FA}, {\color{dgreen} \LA}\right\rangle \right) \Bigm| _{\sJt}.  
 \end{aligned}
\end{equation}

If  $\bar{i}\notin C(\bar{N})$, this means that there exists a component $j \in \{1,...,n-1\}$ such that $i_j \geq N^C_{j+1}$. Then, one of the main properties of the specialisation map shows that, the coefficient $w^{C(2)}_{i_1}\cdot ... \cdot w^{C(n)}_{i_{n-1}}
$ vanishes through the specialisation $\sJt$:

\begin{equation}
\sJt(w^{C(2)}_{i_1}\cdot ... \cdot w^{C(n)}_{i_{n-1}}
)=0, \ \forall \bar{i} \notin C(\bar{N}).
\end{equation}

This means that:
\begin{equation}
\begin{aligned}
& \ \ \ \ \ \ \ \ \ \ \ \ \ \ \  \left( \prod_{i=1}^l u_{i}^{ \left(f_i-\sum_{j \neq {i}} lk_{i,j} \right)} \cdot \prod_{i=2}^{n} u^{-1}_{C(i)} \cdot \right. \\
 & \ \ \ \ \ \ \ \ \ \ \ \ \ \ \ \ \ \ \ \left. \sum_{\bar{i}\in\{\overline{0},\dots,\overline{\cN-1}\}, \ \bar{i}\notin C(\bar{N})}w^{C(2)}_{i_1}\cdot ... \cdot w^{C(n)}_{i_{n-1}}
   \left\langle(\beta_{n} \cup {\mathbb I}_{n+2} ) \ { \color{red} \FA}, {\color{dgreen} \LA}\right\rangle \right) \Bigm| _{\sJt}=0 
 \end{aligned}
\end{equation} 
This shows us that the weighted intersection detects just the classes associated to indices that are bounded by $\bar{N}$ once we do the specialisation, $\sJt$ and we have the formula:

\begin{equation}
\begin{aligned}
& \PA\Bigm| _{\sJt}= \left( \prod_{i=1}^l u_{i}^{ \left(f_i-\sum_{j \neq {i}} lk_{i,j} \right)} \cdot \prod_{i=2}^{n} u^{-1}_{C(i)} \cdot \right. \\
 & \ \ \ \ \ \ \ \ \ \ \ \ \ \ \ \ \ \ \ \left. \sum_{\bar{i}\in C(\bar{N})}w^{C(2)}_{i_1}\cdot ... \cdot w^{C(n)}_{i_{n-1}}
   \left\langle(\beta_{n} \cup {\mathbb I}_{n+2} ) \ { \color{red} \FA}, {\color{dgreen} \LA}\right\rangle \right) \Bigm| _{\sJt}.
 \end{aligned}
\end{equation} 

Also, we notice that:
\begin{equation}
\sJt(w^{C(2)}_{i_1}\cdot ... \cdot w^{C(n)}_{i_{n-1}})=1, \ \forall \bar{i}\in C(\bar{N}).
\end{equation}
So, our intersection has the following expression:
\begin{equation}
\begin{aligned}
& \PA\Bigm| _{\sJt}= \left( \prod_{i=1}^l u_{i}^{ \left(f_i-\sum_{j \neq {i}} lk_{i,j} \right)} \cdot \prod_{i=2}^{n} u^{-1}_{C(i)} \cdot \right. \\
 & \ \ \ \ \ \ \ \ \ \ \ \ \ \ \ \ \ \ \ \left. \sum_{\bar{i}\in C(\bar{N})} \left\langle(\beta_{n} \cup {\mathbb I}_{n+2} ) \ { \color{red} \FA}, {\color{dgreen} \LA}\right\rangle \right) \Bigm| _{\sJt}.
 \end{aligned}
\end{equation}

We have in mind the construction of the non-weighted intersection $ \PJ$. We conclude that the only difference between our intersection and the non weighted version  is that the non-weighted sum uses the classes $$\FJ \text{ and } \LJ$$ and the weighted intersection is given by the classes

$$\FA \text{ and } \LA.$$ 

Now we will show that these classes lead to the same result  through the intersection pairings, as below.  
\begin{lem}[Intersections between classes associated to indices less than $\bar{N}$ \cite{CrI}]\label{pr1}

The intersection of the homology classes associated to indices $\bar{i}$ bounded by $\bar{N}$ give the same result:
\begin{equation}
\begin{aligned}
 \left\langle(\beta_{n} \cup {\mathbb I}_{n+2} ) \ { \color{red} \FJ}, {\color{dgreen} \LJ}\right \rangle= \left\langle(\beta_{n} \cup {\mathbb I}_{n+2} ) \ { \color{red} \FA}, {\color{dgreen} \LA}\right\rangle , \forall \bar{i}\in C(\bar{N}).  
 \end{aligned}
\end{equation}
\end{lem}
\begin{proof}

This follows through an analog argument as the proof of Lemma $7.2$ proved in \cite{CrI}. The key point is that even if we work in different configuration spaces, since we act with the braid $\beta_n$ on the left hand side of the disc and by identity on the right hand side, the only potential difference between intersections would originate from points belonging to the left hand side of the disc. On the other hand, we remark that the figures of the two geometrical supports  of the homology classes
$$(\beta_{n} \cup {\mathbb I}_{n+2} ) \ { \color{red} \FJ} \text {   and   } (\beta_{n} \cup {\mathbb I}_{n+2} ) \ { \color{red} \FA}$$ are the same in the left hand side of the disc. So overall we get the same intersections. 

 \end{proof}

So, we see that the weighted intersection is given by the following expression:
\begin{equation}
\begin{aligned}
& \PA\Bigm| _{\sJt}= \left( \prod_{i=1}^l u_{i}^{ \left(f_i-\sum_{j \neq {i}} lk_{i,j} \right)} \cdot \prod_{i=2}^{n} u^{-1}_{C(i)} \cdot \right. \\
 & \ \ \ \ \ \ \ \ \ \ \ \ \ \ \ \ \ \ \ \left. \sum_{\bar{i}\in C(\bar{N})} \left\langle(\beta_{n} \cup {\mathbb I}_{n+2} ) \ { \color{red} \FAM}, {\color{dgreen} \LAM}\right\rangle \right) \Bigm| _{\sJt}.
 \end{aligned}
\end{equation} 
This menas that we recover the non-weighted intersection, once we apply this specialisation of coefficients, as below:
\begin{equation}
\begin{aligned}
\left(\PA \right)\Bigm| _{\sJt} =~ \left(\PJ \right)\Bigm| _{\sJto}, \forall \bar{N} \leq \cN.
\end{aligned}
\end{equation} 

This concludes the proof of the Lemma that relates our two intersection pairings.

On the other hand, the non-weighted intersection recovers the $\bar{N}$ multi-coloured Jones invariant, following Theorem \ref{omJones}:
\begin{equation}
\begin{aligned}
&J_{\bar{N}}(L) =~ \PJ \Bigm| _{\sJto}.
\end{aligned}
\end{equation} 

We conclude that the weighted intersection recovers all the multi-coloured Jones invariants at levels bounded by $\cN$:
\begin{equation}
\begin{aligned}
&J_{\bar{N}}(L) =~ \PA \Bigm| _{\sJt}, \forall \bar{N} \leq \cN.
\end{aligned}
\end{equation}
This finishes the proof of the globalising Theorem for all multi-coloured Jones invariants for links at levels bounded by $\cN$.

\end{proof}

\subsection{Graded intersection in the ring with integer coefficients} In this part we investigate the ring of coefficients for the weighted Lagrangian intersection.  
Based on the construction of our homology groups, the homology classes $\FA$ and $\LA$ belong to
$ \HA \text{ and } \HAd $ that are modules over  $\C[w^1_1,...,w^{l}_{\cN-1},x_1^{\pm 1},...,x_l^{\pm 1},y^{\pm1}, d^{\pm 1}].$ So a priori we have that:  
$$\PA \in \C[w^1_1,...,w^{l}_{\cN-1},u_1^{\pm 1},...,u_l^{\pm 1},x_1^{\pm 1},...,x_l^{\pm 1},y^{\pm 1}, d^{\pm 1}].$$
However, a nice property for our homology classes is that actually their intersection pairing has all coefficients that are integers, as follows.

 \begin{lem}
 The graded weighted intersection $\PA$ takes values in the Laurent polynomial ring with integer coefficients:
\begin{equation}\label{integ}
\PA \in \Z[w^1_1,...,w^{l}_{\cN-1},u_1^{\pm 1},...,u_l^{\pm 1},x_1^{\pm 1},...,x_l^{\pm 1},y^{\pm 1}, d^{\pm 1}].
\end{equation}

\begin{proof}
The weighted intersection form is given by:
\begin{equation}
\begin{aligned}
\PA&:=\prod_{i=1}^l u_{i}^{ \left(f_i-\sum_{j \neq {i}} lk_{i,j} \right)} \cdot \prod_{i=2}^{n} u^{-1}_{C(i)} \cdot \\
 & \ \sum_{\bar{i}\in\{\bar{0},\dots,\overline{\cN-1}\}} w^{C(2)}_{i_1}\cdot ... \cdot w^{C(n)}_{i_{n-1}}
  \left\langle(\beta_{n} \cup {\mathbb I}_{n+2} ) \ { \color{red} \FA}, {\color{dgreen} \LA}\right\rangle. 
 \end{aligned}
\end{equation} 
We will look at the specific geometric support of our homology classes, and prove based on that that in fact their intersection has all coefficients that are integers:

\begin{equation}
\left\langle(\beta_{n} \cup {\mathbb I}_{n+2} ) \ { \color{red} \FA}, {\color{dgreen} \LA}\right\rangle \in \Z[x_1^{\pm 1},...,x_l^{\pm 1},y^{\pm 1}, d^{\pm 1}].
\end{equation}

As we have discussed, the intersection pairing between two homology classes  is encoded by the set of intersection points between their geometric supports, graded by the local system. We remark that the only contribution of the local system that could potentially give complex coefficients would use loops that have non-trivial winding number around the set of $p$-punctures from the right hand side of the disc. On the other hand, the intersection points between our geometric supports have associated loops that do not wind around punctures from the right hand side of the disc. So, overall relation \eqref{integ} holds and so we see that indeed our intersection form has integer coefficients: 
$$\PA \in \Z[w^1_1,...,w^{l}_{\cN-1},u_1^{\pm 1},...,u_l^{\pm 1},x_1^{\pm 1},...,x_l^{\pm 1},y^{\pm 1}, d^{\pm 1}].$$
\end{proof}

 \end{lem}

\section{$\cN^{th}$ Unified Jones invariant and $\cN^{th}$ unified Alexander invariant} \label{S:UN} Theorem \ref{THEOREMA} tells us that the level $\cN$ weighted Lagrangian intersection $\PA$ contains all coloured Jones polynomials and all coloured Alexander polynomials for links with colours bounded by $\cN$.  In this section our aim is to construct link invariants out of this intersection, which is defined using braid representatives. More specifically, we will define two link invariants: $$\PAAJ \text{     and     }  \PAA$$ starting from this weighted intersection $\PA$ in the configuration space $\Conf_{2+(n-1)(\cN-1)}\left(\mathbb D_{2n+2}\right).$ 
They come from the same geometric set-up, and $\PAAJ$ will unify all coloured Jones polynomials up to level $\cN$ and $\PAA$ will unify coloured Alexander polynomials up to level $\cN$.

The first part for this construction is dedicated to the definition of two  ring of coefficients where these invariants will be defined. Secondly, we will show that in these rings, we have indeed well-defined link invariants with the desired globalisation property. 

\subsection{Set-up and notations}

\begin{defn}[Rings for the universal invariant] Let us denote the following rings:
 
\begin{equation}
 \begin{cases}
 \LL:=\Z[w^1_1,...,w^{l}_{\cN-1},u_1^{\pm 1},...,u_{l}^{\pm 1},x_1^{\pm 1},...,x_{l}^{\pm 1}, y^{\pm 1},d^{\pm 1}]
 =\Z[\bar{w},(\bar{u})^{\pm 1},(\bar{x})^{\pm 1},y^{\pm 1},d^{\pm 1}]\\
 \LN:=\Q(\xn)(x^{2\cN}_{1}-1,...,x^{2\cN}_{l}-1)^{-1}[x_1^{\pm 1},...,x_{l}^{\pm 1}]\\
 \LNJ=\Z[d^{\pm 1}]
 \end{cases}
 \end{equation}
where $\bar{w}:=(w^1_1,...,w^{l}_{\cN-1}), (\bar{u})^{\pm 1}=(u_1^{\pm 1},...,u_{l}^{\pm 1}), (\bar{x})^{\pm 1}=(x_1^{\pm 1},...,x_{l}^{\pm 1}).$
\end{defn}

\subsection{Product up to a finite level} 

\begin{defn}[Level $\cN$ specialisations] 
We recall the specialisations associated for the generic case and for the case of roots of unity, presented in Subsection \ref{sumspec} which we denoted as:
$$ \snJ: \LL=\Z[\bar{w},(\bar{u})^{\pm 1},(\bar{x})^{\pm 1},y^{\pm 1},d^{\pm 1}] \rightarrow \LNJ$$
\begin{equation}\label{u1} 
\begin{cases}
&\sJt(u_i)=\left(\sJt(x_i)\right)^t=d^{1-N_i}\\
&\sJt(x_i)=d^{1-N_i}, \ i\in \{1,...,l\}\\
&\sJt(y)=[N^C_1]_{d^{-1}},\\
&\sJt(w^k_j)=1, \ \text{ if } j\leq N_k-1 ,\\
&\sJt(w^k_j)=0, \ \text{ if } j\geq N_{k},  k\in \{1,...,l\}, j\in \{1,...,\cN-1\}.
\end{cases}
\end{equation}
$$ \sAt:\LL \rightarrow \LN$$
\begin{equation}
\begin{cases}
&\sAt(u_j)=x_j^{(1-\cM)}\\
&\sAt(y)=(d[\lambda_{C(1)}]_{\xi_{\cM}}),\\
&\sAt(d)= \xi_{\cM}^{-1}\\
&\sAt(w^k_j)=1, \ \text{ if } j\leq \cM-1,\\
&\sAt(w^k_j)=0, \ \text{ if } j\geq \cM,  k\in \{1,...,n-1\}, j\in \{1,...,\cN-1\}.
\end{cases}
\end{equation}
\end{defn}

\begin{defn}[Product up to level $\cN$]\label{prN} For a fixed level $\cN$, let us define the product of the rings for smaller levels as:
 \begin{equation}
\begin{aligned}
& \LNt:=\prod_{\cM \leq \cN}\LM \ \ \ \ \ \ \ \ \ \  \LNtJ:=\prod_{\bar{N} \leq \cN}\LNJ.
 \end{aligned}
 \end{equation}
 and $$\PMN:\LNt \rightarrow \LM , \ \ \ \ \ \ \ \ \PMNJ:\LNt \rightarrow \LNJ$$ the projections onto the corresponding components.
 
 Also, let us denote by $$\snt: \LL \rightarrow \LNt, \ \ \ \ \ \sntJ: \LL \rightarrow \LNtJ$$ the product of specialisations:
\begin{equation}
 \snt:=\prod_{\cM \leq \cN}\sm \ \ \ \ \ \ \ \  \sntJ:=\prod_{\bar{N} \leq \cN}\snJ.
 \end{equation} 
\end{defn}
\begin{defn}[Coloured invariants up to a fixed level]
We denote the product of the invariants up to level $\cN$ as below:
 \begin{equation}
 \begin{aligned}
& \ANt:=\prod_{\cM \leq \cN}\AM \in \LNt\\
& \ANtJ:=\prod_{\bar{N} \leq \cN}J_{\bar{N}}(L) \in \LNtJ. 
\end{aligned}
 \end{equation}

\end{defn}

We will construct our universal rings using these two sequences of morphisms $$\{ \snt, \sntJ \mid \cN \in \N\}.$$ More specifically, we will use the sequence of kernels associated to these maps, as follows. 
\begin{defn}[Kernels and quotient rings]\label{ker} Let us denote by:
\begin{equation}
\begin{aligned}
&\tilde{I}_{\cN}:=\text{Ker}\left(\snt \right) \subseteq \LL\\
&\tilde{I}^{J}_{\cN}:=\text{Ker}\left(\sntJ \right) \subseteq \LL.
\end{aligned}
\end{equation}
Then, let us  denote the quotient rings associated to these ideals and denote them as below::
\begin{equation}
\LLN:=\LL / \tilde{I}_{\cN} \ \ \ \ \ \ \ \ \LLNJ:=\LL / \tilde{I}^{J}_{\cN}.
\end{equation}
\end{defn}
\begin{rmk}(Nested sequences of ideals given kernels of specialisation maps)\label{nested}

Following the definition of the product rings $\LNt$ and $\LNtJ$ and product specialisation maps $\snt$ and $\sntJ$, we remark that they lead to two sequence of nested ideals:
\begin{equation}
\begin{aligned}
&\cdots \supseteq \tilde{I}_{\cN} \supseteq \tilde{I}_{\cN+1} \supseteq \cdots \\
&\cdots \supseteq \tilde{I}^{J}_{\cN} \supseteq \tilde{I}^{J}_{\cN+1} \supseteq \cdots 
\end{aligned}
\end{equation}
\end{rmk}

We recall that we have the product of the rings for smaller levels:
 \begin{equation}
 \LNt:=\prod_{\cM \leq \cN}\LM \ \ \ \ \ \ \ \ \ \  \LNtJ:=\prod_{\bar{N} \leq \cN}\LNJ.
 \end{equation}
 \begin{defn}[Projection maps on quotient rings]\label{L:speciald}
Let us consider the associated projection maps, which we define as $g_{\cN}, g^J_{\cN}$ and $\snbb, \snbbJ$ respectively, as shown in the diagram from Figure \ref{diagqu}.
\begin{figure}[H] \label{L:special}
\begin{center}
\begin{tikzpicture}
[x=1mm,y=1mm]

\node (b1)               at (-20,0)    {$\LL$};
\node (b2) [color=dgreen] at (20,0)   {$\LNt$};
\node (b3) [color=red] at (0,-20)   {$\LLN$};
\node (b4) [color=red] at (25,-20)   {$\LLh$};
\node (a1)               at (-32,0)    {$\PA \in$};
\node (a2)               at (-13,-20)    {${\color{red} \PAA} \in$};
\node (c1) [color=dgreen] at (15,20)   {$\LN$};

\draw[->,color=black]   (b1)      to node[left,xshift=-3mm,yshift=0mm,font=\large]{            $g_{\cN}$} (b3);
\draw[->,color=black]   (b1)      to  node[right,xshift=-2mm,yshift=3mm,font=\large]{$\snt$}                        (b2);
\draw[->,color=black]   (b3)      to node[right,yshift=-2mm,xshift=1mm,font=\large]{$\snbb$}                        (b2);
\draw[->,color=black, dashed]   (b4)      to  node[right,xshift=-2mm,yshift=-3mm,font=\large]{}                        (b3);
\draw[->,color=black,dashed]   (b4)      to  node[right,xshift=0mm,yshift=0mm,font=\large]{$\snb$}                        (b2);
\draw[->,color=black,dashed]   (b2)      to  node[right,xshift=0mm,yshift=0mm,font=\large]{$\PNN$}                        (c1);
\draw[->,color=black,dashed]   (b4)   [out=30,in=20]   to  node[right,xshift=0mm,yshift=0mm,font=\large] {$\usn$}                        (c1);

\node (b2) [color=dgreen] at (-80,0)   {$\LNtJ$};
\node (b3) [color=red] at (-40,-20)   {$\LLNJ$};
\node (b4) [color=red] at (-85,-20)   {$\LLhJ$};
\node (c1) [color=dgreen] at (-75,20)   {$\LNJ$};
\node (a2')               at (-53,-20)    {${\color{red} \PAAJ} \in$};

\draw[->,color=black]   (b1)      to node[left,xshift=-3mm,yshift=0mm,font=\large]{            $g^J_{\cN}$} (b3);
\draw[->,color=black]   (a1)      to  node[right,xshift=-2mm,yshift=3mm,font=\large]{$\sntJ$}                        (b2);
\draw[->,color=black]   (b3)      to node[left,yshift=-2mm,xshift=1mm,font=\large]{$\snbbJ$}                        (b2);
\draw[->,color=black, dashed]   (b4)      to  node[right,xshift=-2mm,yshift=-3mm,font=\large]{}                        (a2');
\draw[->,color=black,dashed]   (b4)      to  node[right,xshift=0mm,yshift=0mm,font=\large]{$\snbJ$}                        (b2);
\draw[->,color=black,dashed]   (b2)      to  node[right,xshift=0mm,yshift=0mm,font=\large]{$\PNNJ$}                        (c1);
\draw[->,color=black,dashed]   (b4)   [out=150,in=130]   to  node[left,xshift=0mm,yshift=0mm,font=\large] {$\usnJ$}                        (c1);
\end{tikzpicture}
\end{center}
\vspace{-3mm}
\caption{\normalsize Two quotients of the weighted Lagrangian intersection}\label{diagqu}
\end{figure}
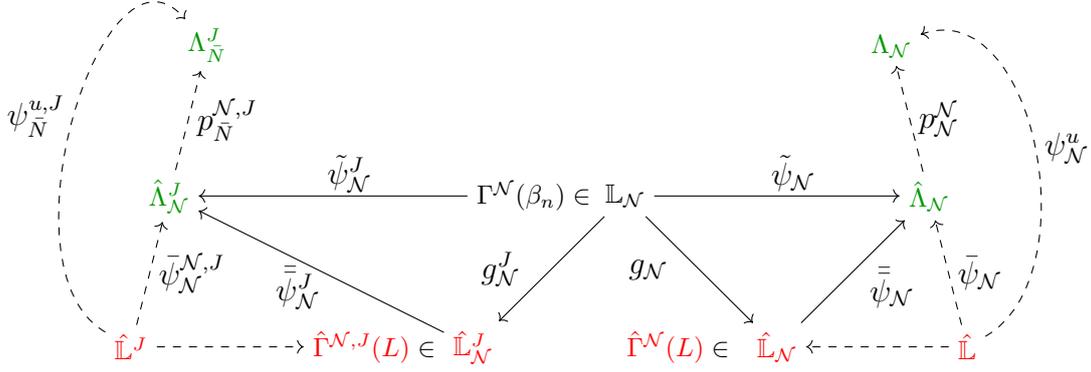

We also define the projection maps
$$ \snbJt: \LLNJ \rightarrow \LMJ,  \ \ \ \ \ \snbt: \LLN \rightarrow \LM $$
which are given by the formulas:
\begin{equation}\label{quotmn}
\smnbJt:= \PMNJ \circ \snbbJ, \ \ \ \ \ \ \smnbt:= \PMN \circ \snbb.
\end{equation}
 \end{defn}
\begin{defn}[Sequence of quotient rings]
In this manner, following Remark \ref{nested}, we obtain two sequences of  quotient rings, with maps between them:
\begin{equation}
\begin{aligned}
 l_{\cN} \hspace{10mm} l_{\cN+1}  \\
\cdots \LLN \leftarrow \LLNN \leftarrow \cdots \\
 l^{J}_{\cN} \hspace{10mm} l^{J}_{\cN+1}  \\
\cdots \LLNJ \leftarrow \LLNNJ \leftarrow \cdots 
\end{aligned}
\end{equation}
\end{defn}
\begin{defn}[Universal limit rings]\label{eq10:limit} We define the projective limits of these sequences of rings and denote them as follows:
\begin{equation}
\LLh:= \underset{\longleftarrow}{\mathrm{lim}} \ \LLN, \ \ \ \ \ \ \ \ \ \ \ \ \LLhJ:= \underset{\longleftarrow}{\mathrm{lim}} \ \LLNJ.
\end{equation}
\end{defn}
\begin{rmk}
We obtain two maps given by projections onto the coefficient rings as below:
 \begin{equation}
\snb: \LLh \rightarrow \LNt, \ \ \ \ \ \ \ \ \ \  \ \ \ \snbJ: \LLhJ \rightarrow \LNtJ.
  \end{equation}
Further, composing with the projection onto the $\cN^{th}$ component and $\bar{N}^{th}$ component respectively, we obtain the following maps:
\begin{equation}
\usn: \LLh \rightarrow \LN, \ \ \ \ \ \ \ \ \ \ \ \ \usnJ: \LLhJ \rightarrow \LNJ,
  \end{equation}
given by the expressions $$ \ \ \usn:=\PNN \circ \snb, \ \ \ \ \ \ \ \ \ \usnJ:=\PNNJ \circ \snbJ .$$
\end{rmk}
\begin{defn}[Projection maps for the ADO and coloured Jones coefficient rings] We obtain also projection maps which relate the coefficient rings at consecutive levels, as below:
 \begin{equation}\label{eq10:2}
\PPN: \LNNt \rightarrow \LNt \ \ \ \ \ \ \text{ and } \ \ \ \ \ \ \PPNJ: \LNNtJ \rightarrow \LNtJ.
  \end{equation}
 \end{defn}
Now we are going to define step by step the two universal invariants, semi-simple and non semi-simple, coming from the same initial data given by our intersection pairing.

\subsection{$\cN^{th}$ Unified Alexander link invariant}
 First, we define a set of link invariants that will be used to build the globalised coloured Alexander invariant. We recall that we have the graded intersection:
 $$\PA \in \LL$$
 and from Theorem \ref{THEOREMA} it recovers the coloured Alexander invariants, as below:
 \begin{equation}\label{rem1}
 \sAt\left( \PA\right)=\AM.
 \end{equation}
Looking at the products component-wise and using the product specialisation maps from Definition \ref{prN} we remark that:
\begin{equation}\label{eq:10.1A}
 \snt\left( \PA\right)=\ANt.
 \end{equation}
\begin{defn}[Level $\cN$ ADO-quotient] 
Let us consider the intersection form obtained from $\PA$ by taking the quotient through the projection $g_{\cN}$ (as in Figure \ref{diagqu}):
\begin{equation}
 \PAA \in \LLN.
 \end{equation}
\end{defn}  
Now we will prove Theorem \ref{INV} which we remind below.
\begin{thm}[$\cN^{th}$ unified Alexander link invariant]\label{levN}
The intersection $\PAA \in \LLN$ is a well-defined link invariant unifying all coloured Alexander polynomials up to level $\cN$:
$$  \PAA|_{\smnbt}=\AM, \ \ \  \forall \cM \leq \cN.$$
\end{thm}
\begin{proof}
Following relation \eqref{eq:10.1A}, we have that:
$$ \snt\left( \PA\right)=\ANt$$
where $L$ is the link obtained from the closure of the braid $\beta_n$. On the other hand, the projection $g_{\cN}$ is defined via the kernel of the map $\snt$ and we have that 
\begin{equation}
 \PAA=g_{\cN}(\PA).
 \end{equation}
 
Putting these properties together, it follows that 
$$ \snbb\left( \PAA\right)=\ANt.$$
On the other hand, from Definition \ref{ker} of the quotient rings, we know that the map $\snbb$ is injective. Moreover, the element that we reach through this map, $\ANt$ is given by the product of all coloured Alexander invariants up to the fixed level $\cN$, so it is a link invariant. 

From the last two properties, we conclude that $\PAA$ is a well-defined link invariant.
 
In order to prove the globalisation property, namely that this link invariant recovers all the ADO invariants up to level $\cN$, we put together relation \eqref{rem1} and the definition of the quotient morphisms from relation \eqref{quotmn}.

This concludes the proof of this theorem, and so we have a well defined globalisation of all coloured Alexander invariants, in a unique link invariant at level $\cN$, which we defined as $\PAA$. 

\end{proof}

\subsection{$\cN^{th}$ Unified Jones invariant}
 In this part we will see that we can use the same geometric set-up as the one from the previous section (concerning unifications of coloured Alexander invariants), from which if we look in a different ring we obtain a set of link invariants that unify the coloured Jones polynomials. 
 
 We start from the intersection $\PA \in \LL$ and we know from Theorem \ref{THEOREMJ} that it recovers the coloured Jones polynomials:
 \begin{equation}\label{rem2}
 \snJ\left( \PA\right)=\ANJ.
 \end{equation}
Then, using the products of the above relations together with the product specialisation maps from Definition \ref{prN}, we obtain that:
\begin{equation}\label{eq:10.1}
 \sntJ\left( \PA\right)=\ANtJ.
 \end{equation}
\begin{defn}[Level $\cN$ Jones-quotient] We define the intersection form obtained from $\PA$ through the quotient via the projection $g^J_{\cN}$ (as in Figure \ref{diagqu}):
\begin{equation}
\begin{aligned}
& \PAAJ \in \LLNJ\\
& \PAAJ=g^J_{\cN}(\PA).
 \end{aligned}
 \end{equation}
\end{defn}  
Now we are ready to prove Theorem \ref{INVJ} which we remind below.
\begin{thm}[$\cN^{th}$ Unified Jones invariant for links] \label{levNJ} For any level $\cN$, the weighted Lagrangian intersection $\PAAJ\in \LLNJ$ is a well-defined link invariant recovering all coloured Jones polynomials with multicolours $\bar{N}$ up to level $\cN$:
$$  \PAAJ|_{\smnbJt}=J_{\cM}(L), \ \ \ \ \ \forall \bar{N} \leq \cN.$$
\end{thm}
\begin{proof}
On one hand we have:
$ \sntJ\left( \PA\right)=\ANtJ$
for $L=\hat{\beta_n}$. 

On the other hand, the projection $g^J_{\cN}$ is defined from the map $\sntJ$ and so:
\begin{equation}
 \PAAJ=g_{\cN}(\PA).
 \end{equation}
Putting these properties together, we obtain the following relation:
$$ \snbbJ\left( \PAAJ\right)=\ANtJ.$$
Now, following Definition \ref{ker} of the quotient rings, we have that the map $\snbbJ$ is injective. Also, by construction, $\ANtJ$ is given by the product of all coloured Jones invariants up to level $\bar{N}$, so it is a link invariant. 

This shows that $\PAAJ$ corresponds to a link invariant through the injective function $\snbb$, so 
$\PAAJ$ is a link invariant.

Following relation \eqref{rem2} and also the definition of the quotient morphisms from \eqref{quotmn}  we conclude that this link invariant recovers all coloured Jones polynomials for links up to level $\cN$.

\end{proof}

\section{Universal Coloured Alexander Invariant} \label{S:U}
In this part our aim is to show that we can unify and define a universal invariant out of the sequence of graded intersections $$\{\PAA\mid \cN \in \N\}.$$ First we recall the Definition \ref{eq10:limit} which contains the formula for the appropiate ring of coefficients where this universal invariant will be defined
$$\LLh:= \underset{\longleftarrow}{\mathrm{lim}} \ \LLN.$$

We will prove that in such a universal ring, we can construct a well defined link invariant which is a universal coloured Alexander invariant, recovering all $\AN$ for all colours $\cN \in \N$.

 \subsection{Definition of the universal ring and invariants} Our universal invariant will be build from the sequence of intersections up to level $\cN$.
 \begin{defn}[Unification of all coloured Alexander link invariants] We consider the projective limit of the graded intersection pairings $\PAA$ over all levels, and denote it as follows:
\begin{equation}
\I:= \underset{\longleftarrow}{\mathrm{lim}} \ \PAA \in \LLh.
\end{equation}
\end{defn}

  \begin{defn}[Limit ring of coefficients and limit invariant]
 Now, let us consider the product of all the rings where the coloured Alexander invariants belong to (where we put no condition about the level)
  \begin{equation}
 \La:=\prod_{\cM \in \N}\LM
 \end{equation}
and the product of all ADO invariants
 \begin{equation}
 \Aa:=\prod_{\cM \leq \cN}\AM \in \La. 
 \end{equation}
Using the above definition, let us denote the associated projection map:
 \begin{equation}
\PhM: \La \rightarrow \LM.
  \end{equation}
We remark that this means that projecting onto the $\cM^{th}$ component we obtain the $\cM^{th}$ coloured Alexander invariant:
\begin{equation}\label{eq10:5}
\PhM(\Aa)=\AM.
\end{equation}

\end{defn}
Putting together all these tools now we are ready to prove Theorem \ref{TH}, which we recall as follows.

\begin{thm}[Universal ADO link invariant] \label{proofADO}The limit of  graded intersections $\I$ is a well-defined link invariant that recovers all coloured Alexander invariants:
\begin{equation}\label{eq10:4}
\AN=  \I|_{\usn}.
\end{equation}
\end{thm}
\begin{proof}

\begin{figure}[H]
\begin{equation}
\centering
\begin{split}
\begin{tikzpicture}
[x=1mm,y=1.6mm]
\node (ll) at (-30,15) {$\LL$};
\node (tl) at (0,62) {$\LLh$};
\node (tr) at (60,62) {$\La$};
\node (ml) at (0,30) {$\LLNN$};
\node (mr) at (60,30) {$\LNNt$};
\node (bl) at (0,0) {$\LLN$};
\node (br) at (60,0) {$\LNt$};
\node (rr) at (90,15) {$\Lambda_{\cM}$};
\node (lle) at (ll.south) [anchor=north,color=red] {\rotatebox{90}{$\subseteq$}};
\node (llee) at (lle.south) [anchor=north,color=red] {$\{\PA\}$};
\node (tle) at (tl.south) [anchor=north,color=red] {$\I$};
\node (tre) at (tr.south) [anchor=north,color=red] {$\Aa$};
\node (mle) at (ml.south east) [anchor=north west,color=red] {$\PAAN$};
\node (ble) at (bl.north east) [anchor=south west,color=red] {$\PAA$};
\node (mre) at (mr.south west) [anchor=north east,color=blue] {$\widehat{\Phi}^{\cN +1}(L)$};
\node (bre) at (br.north west) [anchor=south east,color=blue] {$\widehat{\Phi}^{\cN}(L)$};
\node (rre) at (rr.north) [anchor=south,color=blue] {$\AM$};
\node (trd) at (62,60.5) [anchor=west] {$\coloneqq \displaystyle\prod_{\cM \in \N} \LM$};
\draw[->] (tl) to node[above,font=\small,color=gray]{$\su$} (tr);
\draw[->,densely dashed] (ll) to node[above,font=\small]{$g_{\cN +1}$} (ml);
\draw[->,densely dashed] (ll) to node[below,font=\small]{$g_{\cN}$} (bl);
\draw[->] (mr) to node[below,font=\small]{$\PMNN$} (rr);
\draw[->](br) to node[below,font=\small]{$\PMN$} (rr);
\inclusion{above}{$\snnbb$}{(ml)}{(mr)}
\inclusion{below}{$\snbb$}{(bl)}{(br)}
\draw[->] (ml) to node[left,font=\small]{$\ell_{\cN}$} (bl);
\draw[->] (mr) to node[right,font=\small]{$\PPN$} (br);
\draw[->] (ll) to[out=-71,in=210] node[below,font=\small]{$\snt$} (br);
\draw[->] (ll) to[out=70,in=150] node[above,font=\small]{$\snnt$} (mr);
\draw[|->,color=red] (13,22) to node[right,font=\small,circle,draw,inner sep=1pt]{3} (13,8);
\draw[|->,color=blue] (47,22) to node[left,font=\small,circle,draw,inner sep=1pt]{1} (47,8);
\draw[|->,color=gray] (mle) to node[above,font=\small,circle,draw,inner sep=1pt]{2} (mre);
\draw[|->,color=gray] (ble) to node[below,font=\small,circle,draw,inner sep=1pt]{2} (bre);
\node [color=dgreen] at (28,15) {$\equiv$};
\node [color=dgreen,font=\small,circle,draw,inner sep=1pt] at (32,15) {4};
\draw[|->,densely dashed,color=red] (tle) to (tre);
\draw[|->,color=red] (tle.south east) to node[above,font=\small]{$\usm$} (rre);
\draw[|->,color=gray] (tre) to node[right,font=\small]{$\hat{p}_{\cM}$} (rre);
\draw[->,dotted,color=gray] (tle) to (ml);
\draw[->,dotted,color=gray] (tre) to (mr);
\draw[->,color=blue] (tle) to node[left,font=\small]{$\snbbb$} (mr);
\node[draw,rectangle,inner sep=3pt] at (-25,62) {Limit};
\node[draw,rectangle,inner sep=3pt,color=red] at (-25,57) {Universal ADO};
\end{tikzpicture}
\end{split}
\end{equation}

\caption{The universal ADO link invariant as a limit of the $\cN^{th}$ unified Alexander invariants}\label{limd}
\end{figure}

In order to have a well-defined limit, we should prove that the sequence of level $\cN$ ADO invariants is compatible with the induced quotient maps given by the sequence $\{l_{\cN}\}_{\cN \in \N}$:
\begin{equation}
l_{\cN}\left(\PAAN \right)=\PAA.
\end{equation}
We start with a remark related to the sequence of product type invariants, as follows. 
\begin{rmk}\label{rk1} Following the definition of the projection maps for the coefficient rings from relation \eqref{eq10:2}, we have that the following property:
\begin{equation}
\PPN \left(\ANNt \right)=\ANt.
\end{equation}
\end{rmk}
Next, we prove that the level $\cN$ invariants are related to the above product invariants.
\begin{lem}\label{lemma2}
The link invariant at level $\cN$ recovers the product invariant of all coloured Alexander invariants at levels smaller than $\cN$:
\begin{equation}
\snbb \left(\PAA \right)=\ANt.
\end{equation}
 \end{lem}
\begin{proof}
This property will follow from Theorem \ref{THEOREMAU} together with the definition of the quotient maps, as we will see below. We notice that it is enough to prove that this relation holds when composed with the set of projections $\PMN$ for all $\cM \leq \cN$.
So, we want to show:
\begin{equation}
\PMN \circ \snbb \left(\PAA \right)=\PMN \circ \ANt \left( =\AN \right),  \ \ \ \forall \cM \leq \cN.
\end{equation}
By construction we have that:
\begin{equation}
\PAA =g_{\cN} \left(\PA \right).
\end{equation}
This means that we want to show:
\begin{equation}\label{eq10:3}
\PMN \circ \snbb \circ g_{\cN} \left(\PA \right) =\AN,  \ \ \ \forall \cM \leq \cN.
\end{equation}
On the other hand, we have:
$\snbb \circ g_{\cN}=\snt$ and $\PMN \circ \snt=\sm$. This means that relation \eqref{eq10:3} is equivalent to:
\begin{equation}
\sm \left(\PA \right) =\AN,  \ \ \ \forall \cM \leq \cN.
\end{equation}
This is precisely the statement of the unification result up to level $\cN$ from the Unification Theorem \ref{THEOREMAU}, which concludes the proof.
\end{proof}
Now we will use these properties and conclude the well-behaviour of the invariants with respect to the quotient maps, as follows.
\begin{lem}[Well behaviour of the intersection pairings when changing the level]\label{lemma3} When changing the level from $\cN$ to $\cN+1$, the intersection pairings recover one another through the induced map $l_{\cN}$ as below:
\begin{equation}
l_{\cN} \left(\PAAN \right)=\PAA.
\end{equation}
 \end{lem}
\begin{proof}
We recall that the construction of our quotient rings uses the kernels of specialisation maps:
$$\LLN:=\LL / \text{Ker}\left(\snt \right)$$
and $\snbb$ is the map induced by $\snt$ on this quotient ring, following Definition \ref{L:speciald}. This shows that the maps $\snbb$ and $\snnbb$ are injective. 

Now, following Lemma \ref{lemma2}, we have the property that:
$$\snbb \left(\PAA \right)=\ANt.$$
This together with the injectivity of the map $\snbb$ means that our statement is equivalent to: $$\PPN \left(\ANNt \right)=\ANt.$$
This is precisely the property from Remark \ref{rk1}, and so the statement holds.
\end{proof}
As a conclusion after this discussion, we have that:
\begin{equation}
l_{\cN} \left(\PAAN \right)=\PAA
\end{equation}
and so the sequence $\{\PAA\}_{\cN \in N}$ leads to a well-defined link invariant in the projective limit ring:
$$\I \in \LLh.$$

Next, we want to prove that this invariant recovers all ADO invariants through specialisations of coefficients. For this, we use the following commutation properties. 
\begin{lem}[Commutation of the squares from the diagram]
All squares associated to consecutive levels $\cN$ and $\cN+1$ commute.
\end{lem}
\begin{proof}
This comes from the commutativity when we project onto each factor and by the definition of the quotient maps.   
\end{proof}
The commutation property shows that there exists a well-defined ring homomorphism between the limits, which we denote as below:
  \begin{equation}
\su: \LLh \rightarrow \La.
  \end{equation}
We want to prove that the universal invariant $\I$ recoveres all ADO invariants, for any level, as stated in relation \eqref{eq10:4}:
\begin{equation}
\AN=  \I|_{\usn}.
\end{equation}

Following the commutativity of the squares at all levels, we have that:

\begin{equation}
\su \left( \I \right)=\Aa.
\end{equation}
Also, using the commutativity of the diagrams from Figure \ref{limd} together with Definition \ref{eq10:limit} we obtain that:
\begin{equation}
\PhN \circ \su=\usn.
\end{equation}
Now, following relation \eqref{eq10:5}, we have that:
\begin{equation}
\PhN(\Aa)=\AN.
\end{equation}
Using the previous three relations, we conclude that:
\begin{equation}
\begin{aligned}
\I|_{\usn}&=\usn \left(\I \right)=\\
&=\PhN \circ \su\left(\I \right)=\\
&=\PhN(\Aa)=\\
&=\AN, \ \ \ \ \forall \cN \in \N, \ \cN\geq 2.
\end{aligned}
\end{equation}
This concludes our construction and shows that we have a universal ADO invariant for links $\I$, constructed from graded intersections in configuration spaces, which recovers the coloured Alexander invariants at all levels, through specialisation of coefficients.
\end{proof}
\section{Universal Coloured Jones link invariant} \label{S:JU}
In this part our aim is to unify and show that we can define a second universal link invariant out of the sequence of graded intersections $$\{\PAAJ\mid \cN \in \N\}.$$

We start from Definition \ref{eq10:limit}, where we defined the second universal ring: 
$$\LLhJ:= \underset{\longleftarrow}{\mathrm{lim}} \ \LLNJ.$$

Now we will prove that dually, in this universal ring we have a well-defined link invariant which is a universal coloured Jones invariant, recovering all $J_{\cN}(L)$ for all colours $\cN \in \N$.
This will be constructed from the sequence of intersections up to level $\cN$.
 \begin{defn}[Unification of all coloured Jones link invariants] Let us consider the projective limit of the graded intersection pairings $\PAAJ$ over all levels:
\begin{equation}
\IJJ:= \underset{\longleftarrow}{\mathrm{lim}} \ \PAAJ \in \LLhJ.
\end{equation}
\end{defn}

This construction leads to the statement of Theorem \ref{THJ}, as follows.
\begin{thm}[Universal coloured Jones Invariant] The limit of the level $\cN$ link invariants, $\IJJ$, is a well-defined link invariant recovering all coloured Jones polynomials:
\begin{equation}
J_{\bar{N}}(L)=  \IJJ|_{\usnJ}.
\end{equation}
\end{thm}
\begin{proof}
The proof of this statement follows in an analog manner as the one presented in the above section for the model of the universal ADO invariant, as stated in Theorem \ref{proofADO}.
The key part will be played by the following result.
\begin{lem}[Well behaviour of the Jones intersection pairings when changing the level]\label{lemma3'} When we pass from the level $\cN$ to $\cN+1$, the intersection pairings recover one another through the induced map $l^J_{\cN}$ as follows:
\begin{equation}
l^J_{\cN} \left(\PAANJ \right)=\PAAJ.
\end{equation}
 \end{lem}
This lemma, in turn, follows from the unification result which we proved in  Theorem \ref{THEOREMAU}, this time for the second sequence of specialisations, namely $\{ {\sAtuJ} \}$.
\end{proof}
\section{Structure of the two Universal rings}\label{S:ringJones}

In this part, we discuss the structure of the universal rings defined for our universal Jones invariant and universal Alexander invariant for links. In particular we will describe the formulas for the quotient rings at each level.
\subsection{Structure of the ring for Universal weighted Jones invariant}\label{S:ringJones}
Let us look at the structure of the universal ring that we construct for our universal Jones invariant. This ring, $\LLhJ$, is the limit of the sequence of rings $\LLNJ$,  which are quotients of $\LL$ through the ideals $\tilde{I}^{J}_{\cN}$. 

We recall the specialisations maps that we have used for the construction of this universal ring.
\begin{defn}[Level $\cN$ specialisations] 
The specialisations for the generic case (as in Subsection \ref{sumspec}) are given by:
$$ \snJ: \LL=\Z[\bar{w},(\bar{u})^{\pm 1},(\bar{x})^{\pm 1},y^{\pm 1},d^{\pm 1}] \rightarrow \LNJ$$
\begin{equation}
\begin{cases}
&\sJt(u_i)=\left(\sJt(x_i)\right)^t=d^{1-N_i}\\
&\sJt(x_i)=d^{1-N_i}, \ i\in \{1,...,l\}\\
&\sJt(y)=[N^C_1]_{d^{-1}},\\
&\sJt(w^k_j)=1, \ \text{ if } j\leq N_k-1 ,\\
&\sJt(w^k_j)=0, \ \text{ if } j\geq N_{k},  k\in \{1,...,l\}, j\in \{1,...,\cN-1\}.
\end{cases}
\end{equation}
\end{defn}

\begin{defn}[Product up to level $\cN$] For a fixed level $\cN$, we have the product of the rings for smaller levels:$$\LNtJ:=\prod_{\bar{N} \leq \cN}\LNJ$$ and then $\sntJ$ the associated product of specialisations:
\begin{equation}
\begin{aligned}
 \ \ \   \sntJ:=\prod_{\bar{N} \leq \cN}\snJ.
\end{aligned}
 \end{equation} 
\end{defn}
Then, as we have seen in the construction of the universal invariant, we use the quotient by the kernels of these product specialisations.  More precisely, in Definition \ref{ker}, we considered the ideals
\begin{equation}
\tilde{I}^{J}_{\cN}:=\text{Ker}\left(\sntJ \right) \subseteq \LL
\end{equation}
and then the quotient rings associated to these ideals $\LLNJ=\LL / \tilde{I}^{J}_{\cN}.$

\begin{lem}[\bf \em Structure the rings and ideals used for the universal Jones invariant]\label{ringJ}

\

a) {\bf  Semi-simple ideals}

 For each $\cN$, the ideal $\tilde{I}^{J}_{\cN}$ is given by the following formula: 

\begin{equation} \label{iss}
\begin{aligned}
\tilde{I}^J_{\cN}&=  \langle u_i-x_i,\bigcap_{\bar{N} \leq \cN} \left(y\left(d-d^{-1}\right)-(d^{N^C_1}-d^{-N^C_1})\right.,\\
& \ \ \ \ \ \ \ \ \ \ \ \ \ \ \  \ \  \left. x_i-d^{1-N_i}, w^k_j-1 \text{ if } j\leq N_k-1 \right.,\\
& \ \ \ \ \ \ \ \ \ \ \ \ \ \ \ \ \ \left. w^k_j \ \text{ if } j\geq N_{k},   i,k\in \{1,...,l\}, j\in \{1,...,\cN-1\} \right) \rangle.\\
\end{aligned}
\end{equation}

b) {\bf Semi-simple quotient rings}

Correspondingly, the ${\cN}^{th}$ quotient ring $\LLNJ$ has the formula:
\begin{equation}
\begin{aligned}\label{rss}
\LLNJ= \LL / & \langle u_i-x_i,\bigcap_{\bar{N} \leq \cN} \left(y\left(d-d^{-1}\right)-(d^{N^C_1}-d^{-N^C_1})\right.,\\
& \ \ \ \ \ \ \ \ \ \ \ \ \ \ \  \ \  \left. x_i-d^{1-N_i}, w^k_j-1 \text{ if } j\leq N_k-1 \right.,\\
& \ \ \ \ \ \ \ \ \ \ \ \ \ \ \ \ \ \left. w^k_j \ \text{ if } j\geq N_{k},   i,k\in \{1,...,l\}, j\in \{1,...,\cN-1\} \right) \rangle.
\end{aligned}
\end{equation}
\end{lem}

\begin{proof}
Looking at the structure of these quotients, we notice that:
\begin{equation}
\begin{aligned}
\LLNJ&=\LL / \tilde{I}^{J}_{\cN}=\\
&=\LL / \text{Ker} \langle \prod_{\bar{N} \leq \cN}\snJ \rangle =\\
&=\LL / \bigcap_{\bar{N} \leq \cN}\text{Ker}(\snJ).
\end{aligned}
\end{equation}
\begin{rmk}(Individual kernels)
For each fixed colouring $\bar{N}$, we have:

\begin{equation}
\begin{aligned}
\text{Ker}(\snJ)&=\langle u_i-x_i,y\left(d-d^{-1}\right)-(d^{N^C_1}-d^{-N^C_1}), x_i-d^{1-N_i},\\
& \left( w^k_j-1, \ \text{ if } j\leq N_k-1 ,\right. \\
& \ \ \ \ \ \ \ \ \left. w^k_j \ \text{ if } j\geq N_{k},  k\in \{1,...,l\}, j\in \{1,...,\cN-1\}\right)
\\
& \hspace{80mm}\mid 1 \leq i \leq l \rangle \subseteq \LL.
\end{aligned}
\end{equation}
\end{rmk}
Here, we used our convention that for a set of colours we have $\bar{N} \leq \cN$ if $N_1,...,N_l \leq \cN$.
Then, when we intersect we obtain:
\begin{equation}
\begin{aligned}
&\tilde{I}^J_{\cN}=\bigcap_{\bar{N} \leq \cN}\text{Ker}(\snJ)=\\
&=\langle u_i-x_i,\bigcap_{\bar{N} \leq \cN} \left(y\left(d-d^{-1}\right)-(d^{N^C_1}-d^{-N^C_1})\right.,\\
& \ \ \ \ \ \ \ \ \ \left. x_i-d^{1-N_i}, w^k_j-1, \ \text{ if } j\leq N_k-1 \right.,\\
& \ \ \ \ \ \ \ \ \left. w^k_j \ \text{ if } j\geq N_{k},  k\in \{1,...,l\}, j\in \{1,...,\cN-1\} \right)
\\
& \ \ \ \ \ \ \ \ \ \ \ \ \ \ \ \ \ \ \ \ \ \ \ \ \ \ \ \ \ \ \ \ \ \ \ \ \ \ \ \ \ \ \ \ \ \ \ \ \ \ \ \ \ \ \ \ \ \ \ \ \ \ \ \ \  \mid 1 \leq i \leq l \rangle.
\end{aligned}
\end{equation}
This shows the formula from equation \eqref{iss} and then when considering the quotient shows relation \eqref{rss}.

\end{proof}

\subsection{Refined universal Jones invariant}\label{ssrJ}
In the next part we will define a sequence of larger rings, which we call refined rings in the semi-simple case, and the associated projective limit ring. We will see that the universal Jones invariant has a lift in this refined universal ring, which is a braid invariant called ``Refined universal Jones  invariant''. Then, we will end with a conjecture stating that this refined version is a well-defined link invariant.

\begin{defn}(Refined ideals: semi-simple case)\label{ssrJ0}
Let us consider the ideal:
\begin{equation}
\begin{aligned}
{\tilde{I}^{J}_{\cN}}'=&\langle u_i-x_i,\prod_{ 2 \leq j \leq \cN}(y\left(d-d^{-1}\right)-(d^{j}-d^{-j})), \prod_{ 1 \leq j \leq \cN-1} (x_i-d^{1-j}),\\
& \ \ \ \ \ w^k_1-1,w^k_{j}(w^k_{j}-1) \mid 1 \leq k,i \leq l, 2 \leq j \leq \cN-1 \rangle
\end{aligned}
\end{equation}
and the associated quotient ring:
\begin{equation}
\begin{aligned}
{\LLNJ}'= \LL/ {\tilde{I}^{J}_{\cN}}'.
\end{aligned}
\end{equation}
\end{defn}

\begin{lem}There is a surjective map:
\begin{equation}
r^J_{\cN}: {\LLNJ}' \rightarrow \LLNJ.
\end{equation}
\end{lem}
\begin{proof}
Following the description of ${\tilde{I}^{J}_{\cN}}$, let us look what happens with a fixed variable $w^k_j$. This counts for the relation $$r_{k,j}:=w^k_j-1$$ for those indices $\bar{N}$ such that $j\leq N_k-1 $ and $$\bar{r}_{k,j}:=w^k_j$$ for indices $\bar{N}$ such that $j\geq N_{k}$. On the other hand, since we consider all specialisations, we count all possible $\bar{N}$ bounded by $\cN$. So, for any fixed $j>1$, there exists a colouring $\bar{N}$ such that $j\leq N_k-1 $ and another colouring for which $j\geq N_{k}$. The only exception is $w^k_1$ which gets specialised always to $1$, so it gets counted always with relation $r_{k,1}$.
Overall, the relations that we have in the intersection are:
$$\{ r_{k,1},r_{k,j}\bar{r}_{k,j}\mid 1 \leq k \leq l, 2 \leq j \leq \cN-1 \}.$$
This means that the following relations hold in our quotient:
$$\{ r_{k,1},r_{k,j}\cdot \bar{r}_{k,j}\mid 1 \leq k \leq l, 2 \leq j \leq \cN-1 \}=$$
$$=\{ w^k_1-1,w^k_{j}(w^k_{j}-1)\mid 1 \leq k \leq l, 2 \leq j \leq \cN-1 \}.$$

Related to the variables $x$, for each $i$ the relation $x_i-d^{1-N_i}$ appears associated to a fixed $\bar{N}$. Since we intersect over all colourings bounded by the level, we will obtain the set of relations:
$$\left\lbrace \prod_{ 2 \leq j \leq \cN} (x_i-d^{1-j})\mid 1\leq i \leq l \right\rbrace.$$
Similarly, looking at the variable $y$, we have the relation $$y\left(d-d^{-1}\right)-(d^{N^C_1}-d^{-N^C_1}).$$
Using all the colourings, we obtain the product:
$$\lbrace \prod_{ 2 \leq j \leq \cN} y(d-d^{-1})-(d^{j}-d^{-j}) \rbrace.$$

Overall, we obtain that:
\begin{equation}
\begin{aligned}
\bigcap_{\bar{N} \leq \cN} \text{Ker}(\snJ)& \supseteq \langle u_i-x_i,\prod_{ 2 \leq j \leq \cN}(y\left(d-d^{-1}\right)-(d^{j}-d^{-j})), \prod_{ 1 \leq j \leq \cN}( x_i-d^{1-j}),\\
& \ \ \ \ \ w^k_1-1,w^k_{j}(w^k_{j}-1) \mid 1 \leq k \leq n-1, 1 \leq i \leq l, 2 \leq j \leq \cN-1 \rangle.
\end{aligned}
\end{equation}
This shows that we have:
\begin{equation}
\begin{aligned}
{\tilde{I}^{J}_{\cN}}'&=\langle u_i-x_i,\prod_{ 2 \leq j \leq \cN}(y\left(d-d^{-1}\right)-(d^{j}-d^{-j})), \prod_{ 1 \leq j \leq \cN-1} (x_i-d^{1-j}),\\
& \ \ \ \ \ w^k_1-1,w^k_{j}(w^k_{j}-1) \mid 1 \leq k \leq n-1, 1 \leq i \leq l, 2 \leq j \leq \cN-1 \rangle\\
& \subseteq \text{Ker}(\sntJ)= {\tilde{I}^{J}_{\cN}}.
\end{aligned}
\end{equation}
Then, this leads to the surjection between the associated quotient rings and concludes the proof of the lemma.  
\end{proof}

\begin{defn}(Refined ring and refined weighted intersections: semi-simple case)\phantom{A}\\
Let us consider $$\PAAJr(\beta_n) \in {\LLNJ}'$$ to be the image of the intersection $\PA$ in the quotient ring ${\LLNJ}'$ and call it the ``Refined weighted intersection in the semi-simple case''.

We also consider the following refined universal ring, called ``Refined universal ring in the semi-simple case'':
\begin{equation}
{\LLhJ}'= \underset{\longleftarrow}{\mathrm{lim}} \ {\LLNJ}'.
\end{equation}
\end{defn}
\begin{prop}
We have a well-defined surjective map between the limit rings:
$$r^J: {\LLhJ}' \rightarrow \LLhJ.$$
\end{prop}
\begin{proof}
The surjectivity of this map follows from the property that at each level $r_{\cN}^J$ is surjective and surjectivity is preserved when taking projective limits of surjective maps.
\end{proof}
\begin{lem}(Refined semi-simple braid invariant)
The sequence $\PAAJr(\beta_n) \in {\LLNJ}'$ leads to a well-defined braid invariant $\IRJJ(\beta_n) \in {\LLNJ}'$. This braid invariant recovers all coloured Jones invariants with multi-colours less than $\cN$, through specialisation of coefficients.
\end{lem}
\begin{proof}
This follows from Theorem \ref{THEOREMAU} which tells us that the weighted intersection form $\PA$ recovers all coloured Jones invariants at levels bounded by $\cN$. 
\end{proof}
\begin{conjecture}[{\bf \em Refined universal Jones invariant}] The universal refined  intersections $\IRJJ(\beta_n) \in \LLh'$ are link invariants and lead to a well-defined {``Refined universal Jones invariant''} $\IRJJ(L)$. This invariant recovers the universal geometrical Jones invariant that we constructed, as below:
\begin{equation}
\begin{aligned}
&r^J: {\LLhJ}' \rightarrow \LLhJ\\
&r^J\left(\IRJJ(L)\right)=\IJJ.
\end{aligned}
\end{equation}

\end{conjecture}

\subsection{Structure of the ring for the universal ADO invariant}

Now we will investigate the non semi-simple case and look at the structure of the universal ring that we construct for our universal ADO invariant. We start by recalling the specialisations of coefficients.
\begin{defn}[Level $\cN$ specialisations] \label{pA3A}
The specialisations that we use for the root of unity case (as in subsection \ref{sumspec}) are given by:
$$ \sAt: \LL=\Z[\bar{w},(\bar{u})^{\pm 1},(\bar{x})^{\pm 1},y^{\pm 1},d^{\pm 1}] \rightarrow \LM$$
\begin{equation}
\begin{cases}
&\sAt(u_j)=x_j^{(1-\cM)}\\
&\sAt(y)=(d[\lambda_{C(1)}]_{\xi_{\cM}})=\frac{x_{C(1)}-x_{C(1)}^{-1}}{x^{\cM}_{C(1)}-x_{C(1)}^{-\cM}},\\
&\sAt(d)= \xi_{\cM}^{-1}\\
&\sAt(w^k_j)=1, \ \text{ if } j\leq \cM-1,\\
&\sAt(w^k_j)=0, \ \text{ if } j\geq \cM,  k\in \{1,...,l\}, j\in \{1,...,\cN-1\}.
\end{cases}
\end{equation}

\end{defn}

\begin{defn}[Product up to level $\cN$] For a fixed level $\cN$, we have the product of the rings for smaller levels, as:$$\LNt:=\prod_{\cM \leq \cN}\LM$$ and then $\snt$ the associated product of specialisations:
\begin{equation}
\begin{aligned}
 \ \ \   \snt:=\prod_{\cM \leq \cN}\sAt.
\end{aligned}
 \end{equation} 
\end{defn}
As we have discussed, we use the sequence of quotients by the kernels of these product specialisations.  More precisely, in Definition \ref{ker}, we considered the ideals
\begin{equation}
\tilde{I}_{\cN}:=\text{Ker}\left(\snt \right) \subseteq \LL
\end{equation}
and then the quotient rings associated to these ideals $\LLN=\LL / \tilde{I}_{\cN}.$
\begin{notation}
For $n \in \mathbb N$, let $\varphi_{n}(x)\in \Z[x]$ be the $n^{th}$ cyclotomic polynomial in $x$. 
\end{notation}
\begin{lem}[\bf \em Structure of the rings and ideals for the universal ADO invariant]\label{ringA}

\

a) {\bf  Non semi-simple ideals}

For each $\cN$, the quotient ideals in the non-semi simple case have the following formula:
\begin{equation}\label{inss}
\begin{aligned}
\tilde{I}_{\cN}&= \bigcap_{\cM \leq \cN} \langle  (u_i-x_i^{1-\cM}),  y\left(x^{\cM}_{C(1)}-x^{-\cM}_{C(1)}\right)-(x_{C(1)}-x^{-1}_{C(1)}), \ \varphi_{2\cM}(d),\\
& \left. \ \ \ \ \ \ \ \ \ \ \ w^k_j-1, \ \text{ if } j\leq \cM-1\right.,\\
& \ \ \ \ \ \ \ \ \ \  \ \left. w^k_j \ \text{ if } j\geq \cM,  k,i\in \{1,...,l\}, j\in \{1,...,\cN-1\} \right) \rangle.
\end{aligned}
\end{equation}

a) {\bf  Non semi-simple quotient rings}

Then, the quotient rings $\LLN$ have the following structure: 
\begin{equation}\label{rnss}
\begin{aligned}
\LLN= \ & \LL /  \bigcap_{\cM \leq \cN} \langle  (u_i-x_i^{1-\cM}),  y\left(x^{\cM}_{C(1)}-x^{-\cM}_{C(1)}\right)-(x_{C(1)}-x^{-1}_{C(1)}), \ \varphi_{2\cM}(d),\\
& \left. \ \ \ \ \ \ \ \ \ \ \ w^k_j-1, \ \text{ if } j\leq \cM-1\right.,\\
& \ \ \ \ \ \ \ \ \ \  \  w^k_j \ \text{ if } j\geq \cM,  k,i\in \{1,...,l\}, j\in \{1,...,\cN-1\}  \rangle.
\end{aligned}
\end{equation}
\end{lem}

\begin{proof}
Looking at the structure of the quotient rings, we notice that:
\begin{equation}
\begin{aligned}
\LLN&=\LL / \tilde{I}_{\cN}=\\
&=\LL / \text{Ker} \langle \prod_{\cM \leq \cN}\sAt \rangle =\\
&=\LL / \bigcap_{\cM \leq \cN}\text{Ker}(\sAt).
\end{aligned}
\end{equation}
\begin{rmk}(Individual kernels)
If we fix a colouring $\cM$, we have:

\begin{equation}
\begin{aligned}
\text{Ker}(\sAt)&=\langle u_i-x_i^{1-\cM}, y\left(x^{\cM}_{C(1)}-x^{-\cM}_{C(1)}\right)-(x_{C(1)}-x^{-1}_{C(1)}), \varphi_{2\cM}(d),\\
& \ \ \ \ \  \left( w^k_j-1, \ \text{ if } j\leq \cM-1,\right. \\
& \ \ \ \ \ \left. w^k_j, \ \text{ if } j\geq \cM,  k\in \{1,...,l\}, j\in \{1,...,\cN-1\}\right)
\\
& \hspace{80mm}\mid 1\leq i \leq l \rangle \subseteq \LL
\end{aligned}
\end{equation}
\end{rmk}
Then the intersection gives the following formula:
\begin{equation}
\begin{aligned}
&\tilde{I}_{\cN}=\bigcap_{\cM \leq \cN}\text{Ker}(\sAt)=\\
&= \bigcap_{\cM \leq \cN} \langle  (u_i-x_i^{1-\cM}),  y\left(x^{\cM}_{C(1)}-x^{-\cM}_{C(1)}\right)-(x_{C(1)}-x^{-1}_{C(1)}), \ \varphi_{2\cM}(d),\\
& \left. \ \ \ \ \ \ \ \ \ \ \ w^k_j-1, \ \text{ if } j\leq \cM-1\right.,\\
& \ \ \ \ \ \ \ \ \ \  \  w^k_j \ \text{ if } j\geq \cM,  k,i\in \{1,...,l\}, j\in \{1,...,\cN-1\} \rangle.
\end{aligned}
\end{equation}
This proves the formula from equation \eqref{inss} and then when considering the quotient we obtain relation \eqref{rnss}.
\end{proof}

\subsection{Refined universal ADO  invariant}\label{ssrA}

In the next part, we define a sequence of refined rings, and we will investigate the projective limit of these refined rings in this non semi-simple context. We will prove that the universal ADO invariant has a lift in this refined ring, which is a braid invariant called ``Refined universal ADO invariant''. We will end with a conjecture where we state that this refined version is a well-defined link invariant.

\begin{defn}(Refined ideals: non semi-simple case) Let us consider the ideal:\label{r1A}
\begin{equation}
\begin{aligned}
\tilde{I}_{\cN}'& = \bigcap_{\cM \leq \cN} \langle  (u_i-x_i^{1-\cM}),  y\left(x^{\cM}_{C(1)}-x^{-\cM}_{C(1)}\right)-(x_{C(1)}-x^{-1}_{C(1)}), \ d^{2\cM}-1\\
& \left. \ \ \ \ \ \ \ \ \ \ \ w^k_j-1, \ \text{ if } j\leq \cM-1\right.,\\
& \ \ \ \ \ \ \ \ \ \  \ \left. w^k_j \ \text{ if } j\geq \cM,  k,i\in \{1,...,l\}, j\in \{1,...,\cN-1\} \right) \rangle
\end{aligned}
\end{equation}
and the associated ring:
$$\LLN'=\LL/\tilde{I}_{\cN}'.$$
\end{defn}
\begin{lem}
We have $\tilde{I}_{\cN}' \subseteq \tilde{I}_{\cN}$ and an associated  surjective map:
\begin{equation}
r_{\cN}:\LLN' \rightarrow \LLN.
\end{equation}
\end{lem}

\begin{proof}
We will use that we have the following divisibility relation:
$$\varphi_{2\cM}(d) \ /  \left( d^{2\cM}-1 \right).$$

This shows that we have:
\begin{equation}
\begin{aligned}
\tilde{I}_{\cN}'& =\bigcap_{\cM \leq \cN} \langle  (u_i-x_i^{1-\cM}),  y\left(x^{\cM}_{C(1)}-x^{-\cM}_{C(1)}\right)-(x_{C(1)}-x^{-1}_{C(1)}), \ d^{2\cM}-1\\
& \left. \ \ \ \ \ \ \ \ \ \ \ w^k_j-1, \ \text{ if } j\leq \cM-1\right.,\\
& \ \ \ \ \ \ \ \ \ \  \  w^k_j \ \text{ if } j\geq \cM,  k,i\in \{1,...,l\}, j\in \{1,...,\cN-1\}  \rangle\\
& \subseteq \text{Ker}(\snt)= {\tilde{I}_{\cN}}.
\end{aligned}
\end{equation}
This leads to the well-defined surjective map:
\begin{equation}
r_{\cN}:\LLN' \rightarrow \LLN.
\end{equation}
\end{proof}
Now, we consider the projective limit of these larger rings.
\begin{defn}(Refined ring and refined weighted intersections: non semi-simple case)

Let $$\PAAR \in \LLN'$$ be the image of the intersection $\PA$ in the quotient ring $\LLN'$. We call it the ``Refined weighted intersection in the non semi-simple case''.

Also, we define the following ring, which we call ``Refined universal ring in the non semi-simple case'':
\begin{equation}
\LLh'= \underset{\longleftarrow}{\mathrm{lim}} \ \LLN'.
\end{equation}

\end{defn}
\begin{prop}
We have a well-defined surjective map between the limit rings:
$$r: {\LLh}' \rightarrow \LLh.$$
\end{prop}
\begin{proof}
This property of surjectivity follows from the fact that at each level $r_{\cN}$ is surjective and surjectivity is preserved when taking projective limits of surjective maps.
\end{proof}

\begin{lem}(Refined non semi-simple braid invariant)
The sequence $\PAAR \in \LLN'$ gives a well-defined braid invariant $\IR(\beta_n) \in \LLh'$. This braid invariant recovers all coloured Alexander invariants through specialisation of coefficients.
\end{lem}
\begin{proof}
This comes from the property that the intersection form $\PA$ recovers all coloured Alexander invariants at all levels bounded by $\cN$. 
\end{proof}
\begin{conjecture}[{\bf \em  Refined universal ADO invariant}] The universal refined  intersections $\IR(L) \in \LLh'$ are link invariant and lead to a well defined {``Refined universal Alexander invariant''}.  This refined invariant recovers the universal geometrical ADO invariant that we constructed, as below:
\begin{equation}
\begin{aligned}
&r: {\LLh}' \rightarrow \LLh\\
&r\left(\IR(L)\right)=\I.
\end{aligned}
\end{equation}

\end{conjecture}
\section{Knot case: Recovering the level $\cN$ non-weighted invariants }\label{S:knotcase} In this section, we restrict to the case of knots and we investigate relations between our universal Jones invariant for knots and the non-weighted universal Jones invariant that we constructed in \cite{CrI}. 
We will show that the invariants defined in the weighted set-up are different than the non-weighted knot invariants at each level fixed level.  This phenomena tells us that when we look at the limit we will obtain two different universal geometrical Jones invariants: the weighted one and the non-weighted one. 

The advantage of the weighted construction that we introduced in this paper  is that it provides a new tool for a general machinery for obtaining universal invariants for links, which involves the techniques developed in the previous sections.

We start by recalling that in the case of knots we have the following two invariants:
\begin{itemize}
\item the Weighted universal Jones invariant $\IJJw\in \LLhJ$ and
\item the Non-weighted universal Jones invariant $\IJJk\in \LLhJk$  (\cite{CrI})
\end{itemize}
that are the projective limits of the knot invariants:
\begin{itemize}
\item the $\cN^{th}$ Weighted unified Jones invariant $\PAAJw\in \LLNJ$ and
\item the $\cN^{th}$ Non-weighted unified Jones invariant $\PAAJk\in \LLNJk.$ 
\end{itemize}

The $\cN^{th}$ Non-weighted unified Jones invariant $\PAAJk\in \LLNJk$ comes from the intersection $\PJk$ (which we presented in Theorem \ref{omJo}) for the case of knots, seen in the quotient ring $\LLNJk$.

So, our set-up for a fixed level $\cN$ uses the two state sum of intersections: $$\PA \text{  and  } \PJk$$ defined via the {configuration space of $(n-1)(\cN-1)+2$} points in the  disc, that are given by the set of Lagrangian intersections $$\{\langle (\beta_{n} \cup {\mathbb I}_{n+2} ) \ { \color{red} \FA}, {\color{dgreen} \LA} \rangle\}_{\bar{i}\in \{\bar{0},\dots,\overline{\cN-1}\}}.$$ Then, the construction of the universal knot  invariants involve two sequence of nested ideals, which are used in order to define the two limit rings. 
\begin{notation}[Variables in the knot case]\label{intform} For this part, since we are interested in the case of knot invariants, we have $l=1$ and let us set $y=1$ and consider the rings $$\LL=\Z[w^1_1,...,w^{1}_{\cN-1},x^{\pm 1},d^{\pm 1}] \ \ \text{  and  } \ \ \LL^k=\Z[x^{\pm 1},d^{\pm 1}].$$  By doing this, our intersection forms $\PA$ and $\PAAJ$ recover the normalised versions of the $\cN^{th}$ coloured Jones polynomials (see \cite[Section 2.3]{CrI} for the discussion concerning the knot case). Then, we will use the notation $J_{\cN}$ for the ${\cN}^{th}$ normalised coloured Jones polynomial of a knot. In the above sections for the general link case, we used $J_{\bar{\cN}}$ for the un-normalised version of this invariant. 

Further, by setting $u=x$ in $\PA$ and $\PJk$ we obtain the intersections $$\PA \in \LL=\Z[w^1_1,...,w^{1}_{\cN-1},x^{\pm 1},d^{\pm 1}], \ \ \ \PJk \in \LL^k=\Z[x^{\pm 1},d^{\pm 1}].$$  \end{notation}

We notice that all the specialisations associated to the coloured Jones polynomials induce well-defined specialisations from the above rings, as follows.

\begin{defn}(Non-weighted specialisation for the knot case)
We consider the change of coefficients given by the formula:
\begin{equation}
\begin{cases}
& \sJto: \Z[x^{\pm 1},d^{\pm 1}] \rightarrow \Z[d^{\pm 1}]\\
&\sJto(x)=d^{1-N}.
\end{cases}
\end{equation}
\begin{equation}
\sJtot:=\prod_{N \leq \cN}\sJto.
 \end{equation} 
\end{defn}
\begin{defn}(Weighted specialisation for the knot case)
We consider the change of coefficients given by the formula:
$$ \sJtk: \Z[w^1_1,...,w^{1}_{\cN-1},x^{\pm 1},d^{\pm 1}] \rightarrow \Z[d^{\pm 1}]$$
\begin{equation}
\begin{cases}
&\sJtk(x)=d^{1-N},\\
&\sJtk(w^1_j)=1, \ \text{ if } j\leq N-1,\\
&\sJtk(w^1_j)=0, \ \text{ if } j\geq N, j\in \{1,...,\cN-1\}.
\end{cases}
\end{equation}
\begin{equation}
 \sJtkt:=\prod_{N \leq \cN}\sJtk.
 \end{equation} 
\end{defn}

Our two universal rings are defined using these two sequences of morphisms $$\{ \sJtkt, \sJtot \mid N \in \N\}.$$ More precisely, we consider the sequence of kernels associated to these maps, as follows. 
\begin{defn}[Kernels and quotient rings]\label{kerk} Let us denote:
\begin{equation}
\begin{aligned}
&\tilde{I}_{\cN}^{J}:=\text{Ker}\left(\sJtkt \right) \subseteq \LL\\
&\tilde{I}^{J,k}_{\cN}:=\text{Ker}\left(\sJtot \right) \subseteq \LL^k.
\end{aligned}
\end{equation}
Then, for the following notations, let us define the quotient rings associated to these ideals:
\begin{equation}
\LLNJ:=\LL / \tilde{I}^{J}_{\cN}, \ \ \ \ \ \ \ \ \LLNJk:=\LL^k / \tilde{I}^{J,k}_{\cN}.
\end{equation}
\end{defn}
\begin{rmk}(Nested sequences of ideals given by kernels of specialisations maps)\label{nested}

We obtain a sequence of nested ideals:
\begin{equation}
\begin{aligned}
&\cdots \supseteq \tilde{I}_{\cN}^{J} \supseteq \tilde{I}_{\cN+1}^{J} \supseteq \cdots \\
&\cdots \supseteq \tilde{I}^{J,k}_{\cN} \supseteq \tilde{I}^{J,k}_{\cN+1} \supseteq \cdots 
\end{aligned}
\end{equation}
\end{rmk}
\begin{prop}[Structure of the quotient rings] We obtain the following structure of our ideals and associated quotient rings:
\begin{equation*}
\begin{aligned}
&\tilde{I}^{J}_{\cN}=\bigcap_{N \leq \cN} \langle  x-d^{1-N}, w^1_j-1 \text{ if } j\leq N-1, w^1_j \ \text{ if } j\geq N, j\in \{1,...,\cN-1\}  \rangle.\\
&\tilde{I}^{J,k}_{\cN}=\langle \prod_{i=2}^{\cN} (xd^{i-1}-1) \rangle\\
& \LLNJ= \Z[w^1_1,...,w^{1}_{\cN-1},x^{\pm1}, d^{\pm 1}]/ \bigcap_{N \leq \cN} \langle \left. x-d^{1-N}, w^1_j-1 \text{ if } j\leq N-1 \right.,\\
& \ \ \ \ \ \ \ \ \ \ \ \ \ \ \ \ \ \ \ \ \ \ \ \ \ \ \ \ \ \ \ \ \ \ \ \ \ \ \ \ \ \ \  w^1_j \ \text{ if } j\geq N, j\in \{1,...,\cN-1\}  \rangle.\\
&\LLNJk= \Z[x^{\pm1}, d^{\pm 1}] /\langle \prod_{i=2}^{\cN} (xd^{i-1}-1) \rangle.
\end{aligned}
\end{equation*}
\end{prop}
\begin{proof}
These descriptions are obtained following the formulas for the specialisation maps $\sJto$ and $\sJtk$, and using a similar computation as the one from Section \ref{S:ringJones}.
\end{proof}
\begin{defn}[Sequence of quotient rings]
They lead to the following two sequences of quotient rings, with maps between them:
\begin{equation}\label{seqJ}
\begin{aligned}
 l^J_{\cN} \hspace{10mm} l^J_{\cN+1}  \\
\cdots \LLNJ \leftarrow \LLNJ \leftarrow \cdots \\
 l^{J,k}_{\cN} \hspace{10mm} l^{J,k}_{\cN+1}  \\
\cdots \LLNJk \leftarrow \LLNJk \leftarrow \cdots 
\end{aligned}
\end{equation}
\end{defn}
\begin{rmk}[{\bf \em $\cN^{th}$ Unified Jones invariants}] Let 
$\PAAJw$ and $\PAAJk$ be the images of the intersection forms $\PA$ and $\PJk$ in the $\cN ^{th}$ quotient rings $\LLNJ$ and $\LLNJk$.

Following the discussions from the previous sections, we have that $\PAAJw$ and $\PAAJk$ are knot invariants recovering all coloured Jones polynomials up to level $\cN$.
\end{rmk}

\begin{defn}[Universal limit rings]We consider the projective limits of these two sequences of rings and denote them as follows:
\begin{equation}
\LLhJ:= \underset{\longleftarrow}{\mathrm{lim}} \ \LLNJ, \ \ \ \ \ \ \ \ \ \ \ \ \LLhJk:= \underset{\longleftarrow}{\mathrm{lim}} \ \LLNJk.
\end{equation}
\end{defn}
\begin{defn}The two universal invariants are then obtained as the following projective limits of the weighted and non-weighted $\cN^{th}$ unified Jones invariants:
\begin{equation}
\begin{aligned}
& \IJJw:= \underset{\longleftarrow}{\mathrm{lim}} \ \PAAJw \in \LLhJ\\
& \IJJk:= \underset{\longleftarrow}{\mathrm{lim}} \ \PAAJk \in \LLhJk. 
\end{aligned}
\end{equation}

\end{defn}

Now we are ready to prove that for the case of knots the weighted construction is different than the non-weighted one at each level, as below. 
\begin{thm}[\bf \em Two different level $\cN$ knot invariants] \label{TdifJ}\label{r1}The two quotient rings at a fixed level $\cN$: $\LLNJk$ and $\LLNJ$ are different and the ideals defining them do not permit projection maps:
\begin{equation}
 \LLNJ \neq \LLNJk.
 \end{equation}
This means that the {$\cN^{th}$ Weighted Unified Jones invariant} and the {$\cN^{th}$ Non-weighted Unified Jones invariant} are different, even if they both globalise the set of all coloured Jones polynomials up to level $\cN$:
\begin{equation}
\begin{aligned}
& \PAAJw \neq\PAAJk\\
&\PAAJw \Bigm| _{x=d^{\cM-1}}= \PAAJk \Bigm| _{x=d^{\cM-1}}  =~ { J_{\cM}(K)}, \ \  \forall \cM\leq \cN.
\end{aligned}
\end{equation} 
\end{thm}
\begin{proof}
We recall that $\LL=\Z[w^1_1,...,w^{1}_{\cN-1},x^{\pm 1},d^{\pm 1}]$ and $\LL^k=\Z[x^{\pm 1},d^{\pm 1}]$. Also, following Notation \ref{intform} the two level $\cN$ invariants are quotients of the intersection forms:
$$\PA \in \LL=\Z[w^1_1,...,w^{1}_{\cN-1},x^{\pm 1},d^{\pm 1}], \ \ \ \PJk \in \LL^k=\Z[x^{\pm 1},d^{\pm 1}].$$
Let us consider the projection map:
$$p_{\cN}:
\LL \rightarrow \LL^k$$
\begin{equation}\label{qunp}
\begin{cases}
&p_{\cN}(x)=x,\\
&p_{\cN}(d)=d,\\
&p_{\cN}(w^1_j)=1, j\in \{1,...,\cN-1\}.
\end{cases}
\end{equation}

\begin{rmk}
 We remark that through this projection, the weighted intersection form recovers the non-weighted one, as below:
\begin{equation}
p_{\cN}(\PA)=\PJk.
\end{equation}
\end{rmk}
The above relation can be obtained by following the formulas for the two intersections, and noticing that the non-weighted intersection $\PJk$ is obtained from the weighted one $\PA$ by setting all variables $w^1_{1},...,w^1_{\cN-1}$ to be $1$.

In the next part we will show that even if the above two intersections $\PA, \PJk$ correspond to each other via the projection $p_{\cN}$, after taking the quotients we obtain different knot knvariants:  $$\PAAJw \neq\PAAJk.$$  
Let us look at the intersection $$\PJk \in \LL=\Z[w^1_1,...,w^{l}_{\cN-1},x^{\pm1}, d^{\pm 1}] /\langle w^1_1-1,...,w^1_{\cN-1}-1 \rangle.$$
We remark that in order to obtain the non-weighted quotient $\LLNJk$, we can start from the ring $\LL$ and quotient by the ideals:
$${\tilde{I}^{J,k}_{\cN}}''=\langle \prod_{i=2}^{\cN} (xd^{i-1}-1),w^1_1-1,...,w^1_{\cN-1}-1 \rangle.$$

More specifically, we have:
$$\LLNJk=\LL / {\tilde{I}^{J,k}_{\cN}}''.$$

In this way, the two knot invariants $\PAAJw, \PAAJk$ can be seen as quotients of the same form $\PA\in \LL$ that is projected onto the two quotients $\LLNJ,\LLNJk$ through the ideals $\tilde{I}^{J}_{\cN},\tilde{I}^{J,k}_{\cN}$:

\begin{equation}
\begin{aligned}
&\tilde{I}^{J}_{\cN}=\bigcap_{N \leq \cN} \langle  x-d^{1-N}, w^1_j-1 \text{ if } j\leq N-1, w^1_j \ \text{ if } j\geq N, j\in \{1,...,\cN-1\}  \rangle.\\
&{\tilde{I}^{J,k}_{\cN}}''=\langle \prod_{i=2}^{\cN} (xd^{i-1}-1),w^1_1-1,...,w^1_{\cN-1}-1 \rangle\\
& \LLNJ= \Z[w^1_1,...,w^{1}_{\cN-1},x^{\pm1}, d^{\pm 1}]/ \bigcap_{N \leq \cN} \langle \left. x-d^{1-N}, w^1_j-1 \text{ if } j\leq N-1 \right.,\\
& \ \ \ \ \ \ \ \ \ \ \ \ \ \ \ \ \ \ \ \ \ \ \ \ \ \ \ \ \ \ \ \ \ \ \ \ \ \ \ \ \ \ \  w^1_j \ \text{ if } j\geq N, j\in \{1,...,\cN-1\}  \rangle.\\
&\LLNJk= \Z[w^1_1,...,w^{1}_{\cN-1},x^{\pm1}, d^{\pm 1}] /~\langle \prod_{i=2}^{\cN} (xd^{i-1}-1),w^1_1-1,...,w^1_{\cN-1}-1 \rangle.
\end{aligned}
\end{equation}

If the two invariants $\PAAJw$ and $\PAAJk$, which are images of the same  intersection $\PA$ in the quotients $\LLNJ$ and $\LLNJk$, would be related then the ideals that we quotient by: 
$$\tilde{I}^{J}_{\cN} \text{ and } {\tilde{I}^{J,k}_{\cN}}''$$
should be included one into the other.

On the other hand, we remark that there are elements which are in $\tilde{I}^{J}_{\cN}$ and do not belong to ${\tilde{I}^{J,k}_{\cN}}''$ and vice versa.

This means that the two invariants $\PAAJw$ and $\PAAJk$ are different and concludes the proof this theorem. 

\end{proof}
Next, we investigate the asymptotic behaviour. More precisely, we have the following. 
\begin{thm}[\bf \em Two different geometric universal Jones invariants for knots] \label{r2} There is no well-defined map between the limits:
\begin{equation}
\begin{aligned}
\nexists \ \pi: \LLhJ \rightarrow \  \LLhJk
\end{aligned}
\end{equation}
that sends $\IJJw\in \LLhJ$ to $\IJJk \in \LLhJk$. \\
This shows that our geometric set-up provides two different universal Jones invariants for knots: $\IJJw\in \LLhJ$ and \ $\IJJk\in \LLhJk$, both obtained as sequences of invariants that globalise all coloured Jones polynomials up to a fixed level, as in Figure \ref{re1}.

\end{thm}

\begin{proof}
In order to have a map between the projective limits, we need a map at each level of the projective sequence of rings. 
On the other hand, we have seen in Theorem \ref{TdifJ} that at each level $\cN$ the map $$p_{\cN}:
\LL \rightarrow \LL^k$$
do not pass at the level of the quotient rings, and  
$$ \PAAJw \neq\PAAJk.$$

This means that also the projective limits of these two sequences of invariants are different and provide the two universal invariants: the weighted universal Jones invariant and the non-weighted universal Jones invariant, as presented in Figure \ref{re1}. \end{proof}

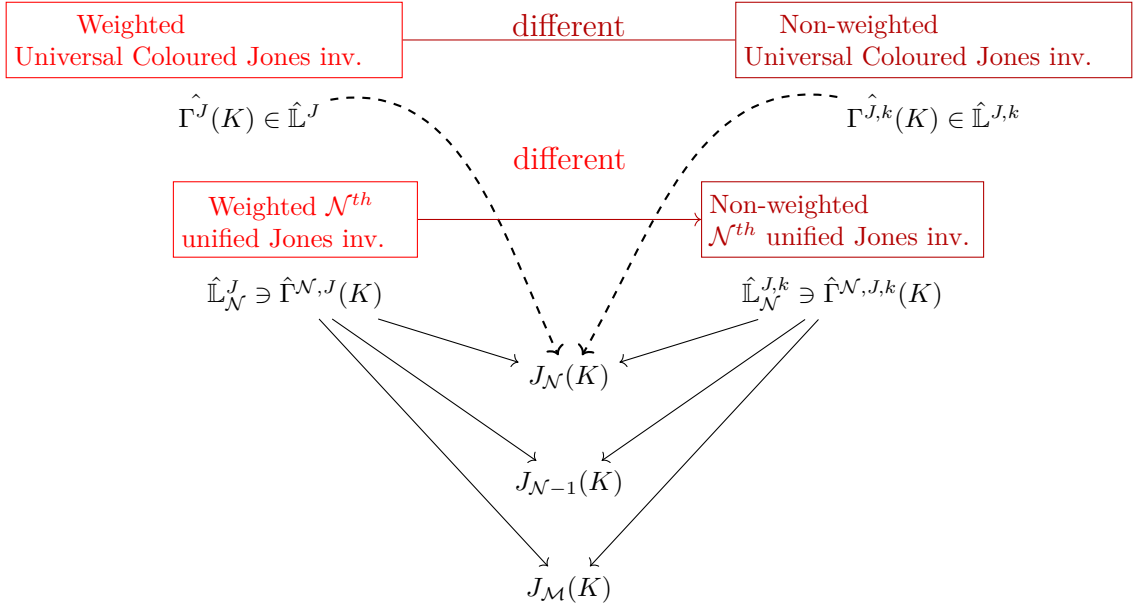
\begin{figure}[H]
\begin{equation*}
\begin{aligned}
&\hspace{2mm} \LLhJ= \underset{\longleftarrow}{\mathrm{lim}} \ \LLNJ \hspace{79mm}\LLhJk= \underset{\longleftarrow}{\mathrm{lim}} \ \LLNJk\\  
&\hspace{-9mm} \LLNJ= \Z[w^1_1,...,w^{1}_{\cN-1},x^{\pm1}, d^{\pm 1}]/  \hspace{58mm}\LLNJk= \Z[x^{\pm1}, d^{\pm 1}] / \\
&\hspace{-9mm} \bigcap_{N \leq \cN} \langle \left. x-d^{1-N}, w^1_j-1 \text{ if } j\leq N-1, \right. w^1_j \ \text{ if } j\geq N, j\in \{1,...,\cN-1\}  \rangle\hspace{5mm} \langle \prod_{i=2}^{\cN} (xd^{i-1}-1) \rangle \end{aligned}
\end{equation*}
\begin{center}
\hspace*{-10mm}\begin{tikzpicture}
[x=1.2mm,y=1.4mm]
\node (Jl)               at (0,8)    {$ \LLNJk \ni \PAAJk$};
\node (Jlk)               at (-60,8)    {$ \LLNJ \ni \PAAJw$};

\node (Jn)               at (-30,0)    {$J_{\cN}(K)$};
\node (Jn-1)               at (-30,-10)    {$J_{\cN-1}(K)$};
\node (Jm)               at (-30,-20)    {$J_{\cM}(K)$};


\node(IJ)[draw,rectangle,inner sep=3pt,color=vdarkred,text width=3.5cm,minimum height=1cm] at (0,15) {Non-weighted \\ $\cN^{th}$ unified Jones inv.};
\node(IJw)[draw,rectangle,inner sep=3pt,color=red,text width=3cm,minimum height=1cm] at (-60,15) { \ \ \ Weighted $\cN^{th}$ \\  unified Jones inv.};
\node(UJ)[draw,rectangle,inner sep=3pt,color=vdarkred,text width=5cm,minimum height=1cm] at (10,32) { \ \ \ \ Non-weighted \\ Universal Coloured Jones inv.};
\node(UJw)[draw,rectangle,inner sep=3pt,color=red,text width=5cm,minimum height=1cm] at (-70,32) { \ \ \ \ \ \ \ Weighted \\ Universal Coloured Jones inv.};

\node (J)               at (10,25)    {$\IJJk \in \LLhJk$};
\node (Jw)               at (-65,25)    {$\IJJw \in \LLhJ$};

\draw[->]             (Jl)      to node[right,xshift=2mm,font=\large]{}   (Jn);
\draw[->]             (Jl)      to node[right,xshift=2mm,font=\large]{}   (Jn-1);
\draw[->]             (Jl)      to node[right,xshift=2mm,font=\large]{}   (Jm);
\draw[->]             (Jlk)      to node[right,xshift=2mm,font=\large]{}   (Jn);
\draw[->]             (Jlk)      to node[right,xshift=2mm,font=\large]{}   (Jn-1);
\draw[->]             (Jlk)      to node[right,xshift=2mm,font=\large]{}   (Jm);

\draw[->,,out=-190, in=60,dashed,thick]             (J)      to node[right,xshift=2mm,font=\large]{}   (Jn);
\draw[->,,out=10, in=115,dashed,thick]             (Jw)      to node[right,xshift=4mm,yshift=2mm,font=\large]{\color{red}different}   (Jn);



\draw[->, vdarkred]             (IJw)      to node[right,xshift=2mm,font=\large]{}   (IJ);
\draw[-, vdarkred]             (UJ)      to node[yshift=2mm,font=\large]{\text{different}}   (UJw);

\end{tikzpicture}
\end{center}
\caption{The weighted and non-weighted universal Jones invariants for knots}\label{re1}
\end{figure}

\section{Refined invariants in modules over Habiro type rings} \label{S:refined} In this part we introduce two Habiro type rings, which we call the quantised and extended Habiro rings. They are two multi-variable extensions of Habiro's famous rings from \cite{H3} and \cite{H2}. In this manner, we will present the two conjectures concerning lifts of our two invariants towards refined universal link invariants in modules over the quantised Habiro ring and extended Habiro ring.
\subsection{Formulas of the two universal rings}

We recall that the construction of our two universal invariants $\I$ and $\IJJ$ starts from the sequences of weighted Lagrangian intersections 
$$ \PA \in \LL=\Z[w^1_1,...,w^{l}_{\cN-1},u_1^{\pm 1},...,u_l^{\pm 1},x_1^{\pm 1},...,x_l^{\pm 1},y^{\pm 1}, d^{\pm 1}].$$
As we have seen, the next step is dedicated to the definition of two sequence of nested ideals in the ring $\LL$, which we denoted by:
$$\cdots \supseteq \tilde{I}_{\cN} \supseteq \tilde{I}_{\cN+1} \supseteq \cdots$$ 
$$\cdots \supseteq \tilde{I}^J_{\cN} \supseteq \tilde{I}^J_{\cN+1} \supseteq \cdots.$$ 
In this manner, we obtain two sequences of associated quotient rings, with maps between them:
\begin{equation}
\begin{aligned}
& \hspace{10mm} l_{\cN} \hspace{10mm} l_{\cN+1}  \\
&\cdots \LLN \leftarrow \LLNN \leftarrow \cdots \ \ \ \ \ \ \ \ \ \ \ \ \ \ \LLh\\
&\hspace{10mm} l^J_{\cN} \hspace{10mm} l^J_{\cN+1}\\
&\cdots \LLNJ \leftarrow \LLNNJ \leftarrow \cdots \ \ \ \ \ \ \ \ \ \ \ \ \ \ \LLhJ.
\end{aligned}
\end{equation}

These two sequences of rings have a very precise description, as we have seen in Lemma \ref{ringA} and Lemma \ref{ringJ}. We recall their formulas below. 
\begin{prop}[Structure of the quotient rings] \label{structure}The sequence of nested ideals in $\LL$ and associated quotient rings that we used for the two universal invariants have the following form:

{\bf $\bullet$ Semi-simple ideals}
\begin{equation} \label{f10}
\begin{aligned}
\tilde{I}^J_{\cN}&=  \langle u_i-x_i,\bigcap_{\bar{N} \leq \cN} \left(y\left(d-d^{-1}\right)-(d^{N^C_1}-d^{-N^C_1})\right.,\\
& \ \ \ \ \ \ \ \ \ \ \ \ \ \ \  \ \  \left. x_i-d^{1-N_i}, w^k_j-1 \text{ if } j\leq N_k-1 \right.,\\
& \ \ \ \ \ \ \ \ \ \ \ \ \ \ \ \ \ \left. w^k_j \ \text{ if } j\geq N_{k},   i,k\in \{1,...,l\}, j\in \{1,...,\cN-1\} \right) \rangle.\\
\end{aligned}
\end{equation}

{\bf $\bullet$ Non semi-simple ideals}
\begin{equation}\label{f20}
\begin{aligned}
\tilde{I}_{\cN}&= \bigcap_{\cM \leq \cN} \langle  (u_i-x_i^{1-\cM}), y\left(x^{\cM}_{C(1)}-x^{-\cM}_{C(1)}\right)-(x_{C(1)}-x^{-1}_{C(1)}) \ ,\varphi_{2\cM}(d)\\
& \left. \ \ \ \ \ \ \ \ \ \ \ w^k_j-1, \ \text{ if } j\leq \cM-1\right.,\\
& \ \ \ \ \ \ \ \ \ \  \  w^k_j \ \text{ if } j\geq \cM,  k,i\in \{1,...,l\}, j\in \{1,...,\cN-1\} \rangle.\\
\end{aligned}
\end{equation}

{\bf $\bullet$ Semi-simple quotients}
\begin{equation} \label{f1}
\begin{aligned}
\LLNJ= \LL / & \langle u_i-x_i,\bigcap_{\bar{N} \leq \cN} \left(y\left(d-d^{-1}\right)-(d^{N^C_1}-d^{-N^C_1})\right.,\\
& \ \ \ \ \ \ \ \ \ \ \ \ \ \ \  \ \  \left. x_i-d^{1-N_i}, w^k_j-1 \text{ if } j\leq N_k-1 \right.,\\
& \ \ \ \ \ \ \ \ \ \ \ \ \ \ \ \ \ \left. w^k_j \ \text{ if } j\geq N_{k},   i,k\in \{1,...,l\}, j\in \{1,...,\cN-1\} \right) \rangle.
\end{aligned}
\end{equation}

{\bf $\bullet$ Non semi-simple quotients}
\begin{equation}\label{f2}
\begin{aligned}
\LLN= & \LL /  \bigcap_{\cM \leq \cN} \langle  (u_i-x_i^{1-\cM}), y\left(x^{\cM}_{C(1)}-x^{-\cM}_{C(1)}\right)-(x_{C(1)}-x^{-1}_{C(1)}), \ \varphi_{2\cM}(d)\\
& \left. \ \ \ \ \ \ \ \ \ \ \ w^k_j-1, \ \text{ if } j\leq \cM-1\right.,\\
& \ \ \ \ \ \ \ \ \ \  \  w^k_j \ \text{ if } j\geq \cM,  k,i\in \{1,...,l\}, j\in \{1,...,\cN-1\}  \rangle.
\end{aligned}
\end{equation}

\end{prop}

Then, the {Universal Jones link invariant} and {Universal ADO link invariant} belong to the projective limits of the above quotient rings: \begin{equation}
\begin{aligned}
&\hspace{-20mm} \IJJ\in \LLhJ= \underset{\longleftarrow}{\mathrm{lim}} \ \LLNJ & & \hspace{20mm} \I\in \LLh= \underset{\longleftarrow}{\mathrm{lim}} \ \LLN.
\end{aligned}
\end{equation}

\subsection{Refined universal link invariants in modules over the Habiro ring}\phantom{A}\label{ss:refH}\\

In Subsection \ref{ssrJ} and Subsection \ref{ssrA} we constructed larger universal rings (that surject onto $\LLhJ$ and $\LLh$), which we call refined universal rings. We introduced them in Definition \ref{ssrJ0} and \ref{r1A} and denoted them by ${\LLhJ}'$ and ${\LLh}'$. We recall their expressions below.
\begin{defn}[Refined quotient ideals and refined rings]\phantom{A}\\\label{extrg}
{\bf $\bullet$ Semi-simple refined quotients} 

\begin{equation}
\begin{aligned}
{\LLNJ}'& = \LL/ \langle u_i-x_i,\prod_{ j=2}^{\cN}(y\left(d-d^{-1}\right)-(d^{j}-d^{-j})), \prod_{ j=1}^{\cN-1} (x_i-d^{1-j}),\\
& \ \ \ \ \ w^k_1-1,w^k_{j}(w^k_{j}-1) \mid 1 \leq k,i \leq l, 2 \leq j \leq \cN-1 \rangle\\
&{\LLhJ}'= \underset{\longleftarrow}{\mathrm{lim}} \ {\LLNJ}'.
\end{aligned}
\end{equation}

\

{\bf $\bullet$ Non semi-simple refined quotients}

\begin{equation}
\begin{aligned}
\LLN'& =\LL/\bigcap_{\cM \leq \cN} \langle  (u_i-x_i^{1-\cM}), y\left(x^{\cM}_{C(1)}-x^{-\cM}_{C(1)}\right)-(x_{C(1)}-x^{-1}_{C(1)}), \ d^{2\cM}-1\\
& \left. \ \ \ \ \ \ \ \ \ \ \ w^k_j-1, \ \text{ if } j\leq \cM-1\right.,\\
& \ \ \ \ \ \ \ \ \ \  \  w^k_j \ \text{ if } j\geq \cM,  k,i\in \{1,...,l\}, j\in \{1,...,\cN-1\} \rangle.
\\
&{\LLh}'= \underset{\longleftarrow}{\mathrm{lim}} \ {\LLN}'.
\end{aligned}
\end{equation}
\end{defn}
Now we will introduce two extensions of Habiro's rings, as follows.
\begin{defn}[Quantised and Extended Habiro ring]\label{qrg}\phantom{!}

a) Let $\hat{\mathbb L}^{H,q}$ be the limit of the rings:
\begin{equation}
\mathbb L^{H,q}_{\cN}= \Z[x_1^{\pm1}, \cdots, x_l^{\pm1}, d^{\pm 1}] /\langle \prod_{j=2}^{\cN} (x_i d^{j-1}-1),1 \leq i \leq l \rangle
\end{equation}
and we call this the ``{\em quantised Habiro ring}''.

b) Let $\hat{\mathbb L}^{H,e}$ be the limit of the rings:
\begin{equation}
\mathbb L^{H,e}_{\cN}= \Z[x_1^{\pm1}, \cdots, x_l^{\pm1}, d^{\pm 1}] /\langle \prod_{\cM=2}^{\cN} (d^{2\cM}-1) \rangle
\end{equation}
and we call this the ``{\em extended Habiro ring}''.
 
\end{defn}
\begin{rmk}[Modules structure over the Habiro rings]\phantom{!}\\
a) The {semi-simple refined universal ring} ${\LLhJ}'$ is a module over the {quantised Habiro ring} $\hat{\mathbb L}^{H,q}$. 
\begin{equation}
\hat{\mathbb L}^{H,q} \curvearrowright {\LLhJ}'.
\end{equation}
b) Dually, the {non semi-simple refined universal ring} ${\LLh}'$ is a module over the {extended Habiro ring} $\hat{\mathbb L}^{H,e}$, as below:
\begin{equation}
\hat{\mathbb L}^{H,e} \curvearrowright {\LLh}'.
\end{equation}
\end{rmk}
With this set-up, we obtain Conjecture \ref{l1} and Conjecture \ref{l2} which state the following (see also Figure \ref{conj}).
\begin{conjecture}[\em \bf Universal Jones invariant lifts over the quantised Habiro ring] \phantom{A}
The universal Jones link invariant $\IJJ\in \LLhJ$ lifts to the {\color{blue} Refined universal Jones link invariant} $\IRJJ(L)\in {\LLhJ}'$, which belongs to a ring that is a module over the {\em \color{blue} quantised Habiro ring} $\hat{\mathbb L}^{H,q}$.
\end{conjecture}

\begin{conjecture}[\em \bf Universal ADO invariant lifts over the extended Habiro ring] \phantom{A}
The universal ADO link invariant $\I\in \LLh$ lifts to the {\color{blue} Refined universal ADO link invariant} $\IR(L)\in {\LLh}'$, which belongs to a ring that is a module over the {\em \color{blue} extended Habiro ring} $\hat{\mathbb L}^{H,e}$.
\end{conjecture}

\begin{figure}[H]
\begin{center}
\begin{equation*}
\begin{aligned}
&\hspace{-16mm} \mathbb L^{H,q}_{\cN}= \Z[x_1^{\pm1}, \cdots, x_l^{\pm1}, d^{\pm 1}] / \hspace{25mm}
\mathbb L^{H,e}_{\cN}= \Z[x_1^{\pm1}, \cdots, x_l^{\pm1}, d^{\pm 1}] / \\
&\hspace{-4mm}\langle \prod_{j=2}^{\cN} (x_i d^{j-1}-1),1 \leq i \leq l \rangle \hspace{28mm} \langle \prod_{\cM=2}^{\cN} (d^{2\cM}-1) \rangle\\
&\hspace{-19mm}\text{Quantised Habiro ring          } \ \  \hat{\mathbb L}^{H,q} \curvearrowright \hspace{34mm} \curvearrowleft \hat{\mathbb L}^{H,e} \ \text{     Extended Habiro ring}
\end{aligned}
\vspace*{-3mm}
\end{equation*}
\phantom{A}
\hspace*{-25mm}\begin{tikzpicture}
[x=1.1mm,y=1.1mm]

\node(UAR)[draw,rectangle,inner sep=3pt,color=vdarkred,text width=6cm,minimum height=0.5cm] at (0,50) {\hspace{15mm}Refined \\ \ \ \ \ \  Universal ADO link invariant\\ \ \  over {extended Habiro Ring} $\mathbb L^{H,e}_{\cN}$};
\node(UJR)[draw,rectangle,inner sep=3pt,color=vdarkred,text width=6cm,minimum height=0.5cm] at (-60,50) { \hspace{15mm} Refined \\  \ \ \ \ \ Universal Jones link invariant \\ \ \  over {quantised Habiro Ring} $\mathbb L^{H,q}_{\cN}$ };

\node(UA)[draw,rectangle,inner sep=3pt,color=red] at (0,30) {Universal ADO link invariant};
\node(UJ)[draw,rectangle,inner sep=3pt,color=red] at (-60,30) {Universal Jones link invariant};
\node (J')               at (-48,60)    {$\IRJJ(L) \in {\LLhJ}'$};
\node (A')               at (-15,60)    {$\IR(L) \in \LLh'$};

\node (J)               at (-52,25)    {$\IJJ \in \LLhJ$};
\node (A)               at (-12,25)    {$\I \in \LLh$};

\draw[->>,dashed, red]             (UJ)      to node[left,xshift=-2mm,font=\large]{${\small \text{Conjecture \ref{l1}}}$}   (UJR);
\draw[->>,dashed, red]             (UA)      to node[right,xshift=2mm,font=\large]{${\small \text{Conjecture \ref{l2}}}$}   (UAR);

\end{tikzpicture}
\caption{Refined universal link invariants in modules over the Habiro ring}
\label{conj}\end{center}

\end{figure}
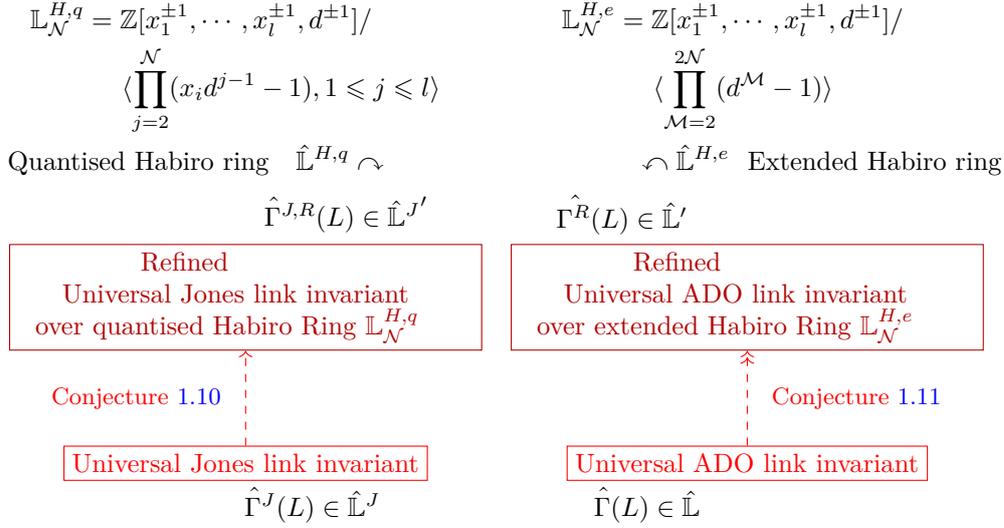

\section{Geometric encoding of semi-simplicity vs non semi-simplicity}\label{ssmodif}
 
From the perspective of representation theory, the extension from quantum invariants of knots to the general link invariant case requires a subtle procedure which involves extra algebraic data.  It originates in the representation theory of the quantum group that defined initially these invariants. More specifically, the underlying algebraic tools needed for the construction of quantum link invariants differ from the semi-simple case to the non-semisimple one, as follows.

\subsection{Representation theoretic origin of non semi-simplicity}

The construction of invariants for links in the semi-simple case involves quantum traces and quantum dimensions (\cite{J}, \cite{RT}). On the other hand, the core of the construction of non semi-simple quantum invariants uses as building blocks modified traces and modified quantum dimensions, following \cite{GPT}, \cite{ADO}. 

In this manner the set of coloured Jones polynomials for links come from the  representation theory of the quantum group $U_q(sl(2))$ at generic $q$ via quantum dimensions, whereas coloured Alexander polynomials for links come from the  representation theory of the same quantum group at roots of unity via modified quantum dimensions. 

\subsection{Pivotal structures and Modified dimensions via the local system}

In the next part we will create a dictionary and explain how we codify these essential algebraic tools provided by modified dimensions through topological lenses. As discussed, our construction uses the topological set-up provided by weighted Lagrangian intersections. These intersections are based on local systems on the configuration space of the punctured disc.

 We will see how we use the geometric data provided by the {monodromies of our local system} and the {variables of our Lagrangian intersection} in order to encode {modified dimensions}. Secondly, we will investigate which variables of our local system capture the {\em difference} between the semi-simplicity of the universal Jones invariant and the non semi-simplicity of the universal ADO invariant.

Let us look at the structure of the quotient rings from Proposition \ref{structure} that lead to our two universal rings. We remark that for the case of the universal coloured Jones invariant, the variables $\bar{u}$ play no special role, since they are specialised in the same way as $\bar{x}$ in all quotients. However, for the roots of unity case, $\bar{u}$ capture deep information which is related to the colour of our non semi-simple invariant. Following the proof of our topological model for the ADO invariants, we remark that the underlying algebraic tool that encodes this relation is the pivotal structure coming from representation theory of the quantum group at roots of unity. This is the reason that led us to consider non-trivial relations between the variables $\bar{x}$ and $\bar{u}$ in the quotient rings. 

Moreover, following the dictionary to the algebraic side, we codified the modified dimension by the variable $y$, which is another strata of the non semi-simple origin of invariants at roots of unity. Let us summarise this in Figure \ref{par}. 

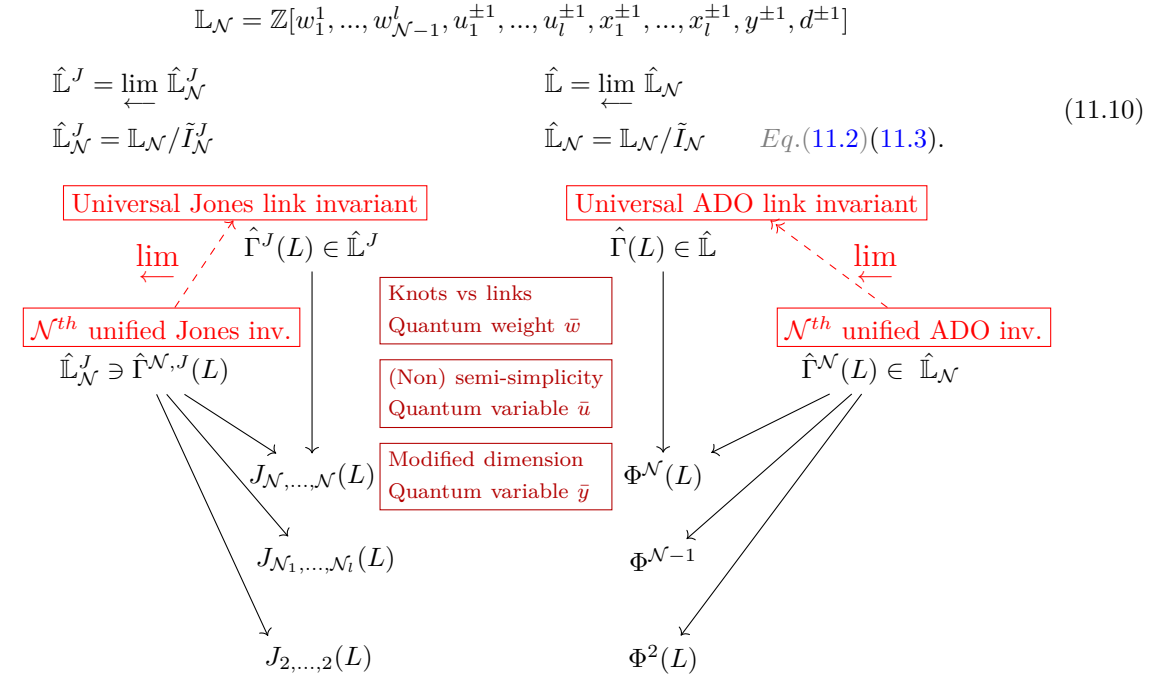
\begin{figure}[H]
$ \hspace{25mm} \small \LL=\Z[w^1_1,...,w^{l}_{\cN-1},u_1^{\pm 1},...,u_l^{\pm 1},x_1^{\pm 1},...,x_l^{\pm 1},y^{\pm 1}, d^{\pm 1}] $
\begin{equation}
\begin{aligned}
&\hspace{-20mm} \LLhJ= \underset{\longleftarrow}{\mathrm{lim}} \ \LLNJ & & \hspace{40mm} \LLh= \underset{\longleftarrow}{\mathrm{lim}} \ \LLN \\
&\hspace{-20mm} \LLNJ= \LL /\tilde{I}^{J}_{\cN} & & \hspace{40mm}\LLN= \LL/\tilde{I}_{\cN} { \ \ \ \ \ \ \color{gray} Eq. \eqref{f1}} \eqref{f2}.\\
\end{aligned}
\end{equation}
\begin{center}
\hspace*{-10mm}\begin{tikzpicture}
[x=1.1mm,y=1.1mm]
\node (Al)               at (16,10)    {$ \PAA \in \ \LLN$};
\node (Jl)               at (-72,10)    {$ \LLNJ \ni \PAAJ$};

\node (Jn)               at (-52,-3)    {$J_{\cN,...,\cN}(L)$};
\node (Jn-1)               at (-52.5,-13)    {$ \ \ \ \ J_{\cN_1,...,\cN_l}(L)$};
\node (Jm)               at (-56,-25)    {$ \ \ \ \ \ \ \ \ \ J_{2,...,2}(L)$};

\node (An)               at (-10,-3)    {${\AN}$};
\node (An-1)               at (-10,-13)    {${\Phi^{\cN-1}}$};
\node (Am)               at (-10,-25)    {${\Phi^2}(L)$};

\node(V1)[draw,rectangle,inner sep=3pt,color=vdarkred,text width=2.84cm,minimum height=0.5cm] at (-30,17) {\footnotesize{Knots vs links \\ Quantum weight $\bar{w}$} };

\node(V2)[draw,rectangle,inner sep=3pt,color=vdarkred,text width=2.84cm,minimum height=0.5cm] at (-30,7) {\footnotesize{(Non) semi-simplicity \\ Quantum variable $\bar{u}$} };

\node(V2)[draw,rectangle,inner sep=3pt,color=vdarkred,text width=2.84cm,minimum height=0.5cm] at (-30,-3) {\footnotesize{Modified dimension \\ Quantum variable $\bar{y}$} };


\node(IA)[draw,rectangle,inner sep=3pt,color=red] at (20,15) {$\cN^{th}$ unified ADO inv.};
\node(IJ)[draw,rectangle,inner sep=3pt,color=red] at (-70,15) {$\cN^{th}$ unified Jones inv.};

\node(UA)[draw,rectangle,inner sep=3pt,color=red] at (0,30) {Universal ADO link invariant};
\node(UJ)[draw,rectangle,inner sep=3pt,color=red] at (-60,30) {Universal Jones link invariant};

\node (J)               at (-52,25)    {$\IJJ \in \LLhJ$};
\node (A)               at (-10,25)    {$\I \in \LLh$};

\draw[->]             (Jl)      to node[right,xshift=2mm,font=\large]{}   (Jn);
\draw[->]             (Jl)      to node[right,xshift=2mm,font=\large]{}   (Jn-1);
\draw[->]             (Jl)      to node[right,xshift=2mm,font=\large]{}   (Jm);

\draw[->]             (J)      to node[right,xshift=2mm,font=\large]{}   (Jn);


\draw[->]             (A)      to node[right,xshift=2mm,font=\large]{}   (An);

\draw[->]             (Al)      to node[right,xshift=2mm,font=\large]{}   (An);
\draw[->]             (Al)      to node[right,xshift=2mm,font=\large]{}   (An-1);
\draw[->]             (Al)      to node[right,xshift=2mm,font=\large]{}   (Am);

\draw[->,dashed, red]             (IJ)      to node[left,xshift=-2mm,font=\large]{$\underset{\longleftarrow}{\mathrm{lim}}$}   (UJ);
\draw[->>,dashed, red]             (IA)      to node[right,xshift=2mm,font=\large]{$\underset{\longleftarrow}{\mathrm{lim}}$}   (UA);

\end{tikzpicture}
\end{center}

\vspace{-4mm}
\caption{Parallel: universal Jones and universal ADO link invariants}\label{par}
\end{figure}

\subsection{Geometric variables encoding quantum tools} Following the above discussion, we see that our weighted Lagrangian intersection $$\PA \in \LL=\Z[w^1_1,...,w^{l}_{\cN-1},u_1^{\pm 1},...,u_l^{\pm 1},x_1^{\pm 1},...,x_l^{\pm 1},y^{\pm 1}, d^{\pm 1}]$$ is given by grading the Lagrangian intersections in configuration spaces by $5$ types of variables: $$\{\bar{x}, \ y \ , \ d \ , \bar{u} \ , \bar{w}\}.$$ These variables come from geometry, being the variables of our local system $\LL$. On the other hand they encode quantum data, as we summarize in the dictionary presented below.
 \begin{enumerate}
\item[•]{\color{blue} \em Variables of the quantum invariant -- Winding around the link} 
\item[] The variables {\color{blue}\ \em $x_i$} record {linking numbers with the link}, as winding numbers around the {\color{blue}$p$-punctures}
\item[•]{\color{blue} \em Relative twisting}
\item[] The power of the variable {\color{blue}$d$} encodes a relative twisting of the submanifolds, given by the winding around the diagonal of the symmetric power of the punctured disc.
\item[•]{\color{carmine} \em Modified dimension}
\item[] The power of the variable {\color{carmine}$y$} counts the winding number around the {\color{carmine} $s$-puncture}
\item[•]{\color{carmine}  \em Pivotal structure}
\item[] The grading {\color{carmine} \bf \em $\bar{u}$} captures the {\color{carmine} semisimplicity/ non-semisimplicity} of the invariant, which is given by a power of ${\color{blue} \bar{x}}$, so a {power of the linking number with the link}.
\item[•]{\color{carmine}  \em Weighted intersection -- unification of all quantum levels}
\item[] The new weights {\color{carmine} \bf \em $\bar{w}$} make possible to unify and see all quantum invariants of colours bounded by the level $\cN$ directly from one topological viewpoint: the weighted intersection $\PA$. 
\end{enumerate}

\clearpage 
\begin{rmk}{\bf Dictionary: geometric variables and quantum tools}\\ Based on the above correspondence, we conclude that the first three types of variables from our ring $\LL$, $\{ \bar{x}, \  d, \  y \}$ are related to the geometry of the local system. More specifically they record monodromies around the punctures and the relative twisting in the configuration space.  

On the other hand, the variables $\bar{u}$ encode algebraic data provided by the pivotal structures of the quantum group. Finally the addition of the quantum weights $\bar{w}$ was a key point for the unification of $U_q(sl(2))$ quantum invariants at bounded colours, and led to the construction of the  universal Jones invariant and universal ADO invariant. We present this dictionary in Figure \ref{ColouredAlexvar}.
\end{rmk}
\begin{figure}[H]
\centering
{\phantom{A} \ \ \ \ \ \ \ \ \ \ \ \ \ \ \ \ \ \ \ \ \ \ \ \ \ \ \ \ \ \ \ \  \ \ \ \ \ \ \ \ \ } \hspace*{-10mm}\includegraphics[scale=0.24]{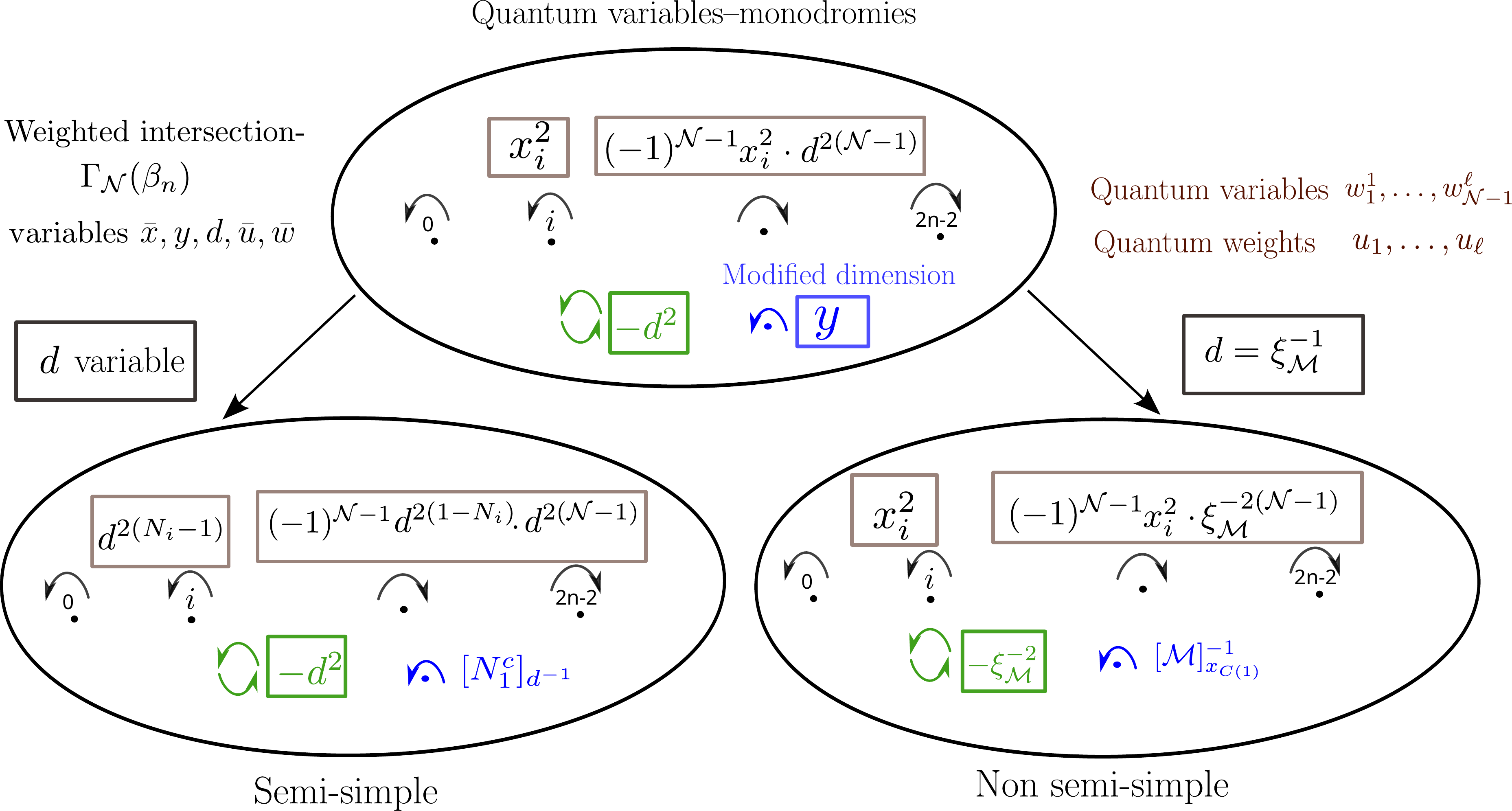}
\vspace{-3mm}
$\hspace{-25mm}\ \ \ \ \ \ \ \ \ \ \ \ \ J_{N_1,...,N_l}(L)- \text{variable  } d \hspace{43mm} \Phi^{\cM}(L)- \text{variables    }x_1,...,x_l$
$$\hspace{-20mm} \ \ \ \ \ \ \ \ \ \text{Coloured Jones polynomial} \hspace{35mm} \ \ \  \text{     Coloured Alexander polynomial}$$

\caption{\normalsize Quantum variables as monodromies of local systems}
\label{ColouredAlexvar}
\end{figure}

\subsection*{Acknowledgements} 
 I would like to especially thank Christian Blanchet, Kazuo Habiro, Jun Murakami and  Emmanuel Wagner for useful discussions related to this project. 
 The author gratefully acknowledges the support of the ANR grant ANR-24-CPJ1-0026-01 at Universit\'e Clermont Auvergne - LMBP. Also, she acknowledges partial support by grants of the Ministry of Research, Innovation and Digitization, CNCS - UEFISCDI, project numbers  PN-IV-P2-2.1-TE-2023-2040 and PN-IV-P1-PCE-2023-2001, within PNCDI IV.

 {\itshape 
 
Universit\'e Clermont Auvergne, CNRS, LMBP, F-63000 Clermont-Ferrand, France, \\Institute of Mathematics “Simion Stoilow” of the Romanian Academy, 21 Calea Grivitei Street, 010702 Bucharest, Romania.}

\

{\itshape cristina.anghel@uca.fr, cranghel@imar.ro}

\noindent \href{http://www.cristinaanghel.ro/}{www.cristinaanghel.ro}

\end{document}